\documentclass[12pt,a4paper]{article}
\pdfoutput=1
\usepackage{amsfonts,amsmath,amssymb,epsf,color,array,colortbl,amscd,bbm,tensor,enumitem}
\usepackage{tikz}
\usetikzlibrary{arrows,decorations.pathmorphing,positioning,matrix,scopes,shapes}
\usepackage{hyperref}

\advance\textheight by \topskip
\allowdisplaybreaks[1]
\usepackage{amsthm}
\theoremstyle{plain}
\newtheorem{thm}{Theorem}[section]
\newtheorem{prop}[thm]{Proposition}

\newtheorem{definition}[thm]{Definition}
\newtheorem{remark}[thm]{Remark}

\renewcommand{\epsilon}{\varepsilon}

\renewcommand{\hom}{\operatorname{Hom}}
\newcommand{\ehom}{\operatorname{End}}
\renewcommand{\phi}{\varphi}
\renewcommand{\d}{\operatorname{d}}
\newcommand{\degr}{\operatorname{deg}\,}
\newcommand{\id}{\operatorname{id}}
\newcommand{\res}{\operatorname{Res}}

\newcommand{\im}{\mathop{\operatorname{im}}}
\newcommand{\symg}[1]{\ensuremath{\mathfrak{S}_{#1}}}
\newcommand{\scr}[1]{\ensuremath{\mathcal{Q}_{#1}}}
\newcommand{\scrp}[2]{\ensuremath{\scr{#1}^{[#2]}}}

\tikzstyle{socbot} = [draw, circle,inner sep=1.8pt,outer sep=1.5pt]
\tikzstyle{socmida} = [draw, diamond,inner sep=1.3pt,outer sep=1.5pt]
\tikzstyle{socmidb} = [draw, regular polygon, regular polygon sides=4,inner sep=0.6pt,outer sep=1.5pt]
\tikzstyle{soctop} = [draw, regular polygon, regular polygon sides=3,inner sep=-0.4pt,outer sep=1.5pt]

\begin{document}
\thispagestyle{empty}
\def\thefootnote{\fnsymbol{footnote}}

\begin{center}\LARGE
  \textbf{%
    On the extended {\boldmath\(W\)-algebra
    of type \(\mathfrak{sl}_2\)\unboldmath} at positive rational level
  }
\end{center}\vskip 2em
\begin{center}\large
  Akihiro Tsuchiya\footnote{Email: {\tt akihiro.tsuchiya@ipmu.jp}} \,%
  and
  Simon Wood\footnote{Email: {\tt simon.wood@ipmu.jp}}
\end{center}
\begin{center}
  Kavli Institute for the Physics and Mathematics of the Universe (WPI),\\
  Todai Institutes for Advanced Study,\\
  The University of Tokyo
\end{center}
\vskip 1em
\begin{center}
  \date{\today}
\end{center}
\vskip 1em
\begin{abstract}
  The extended \(W\)-algebra of type \(\mathfrak{sl}_2\) at positive rational
  level, denoted by \(\mathcal{M}_{p_+,p_-}\), is a vertex operator algebra that
  was originally proposed in \cite{Feigin:2006iv}. This vertex operator
  algebra is an extension of the minimal model vertex operator algebra and
  plays the role of symmetry algebra for certain logarithmic conformal field theories.
  We give a construction of \(\mathcal{M}_{p_+,p_-}\) in terms of screening
  operators and use this construction to prove that \(\mathcal{M}_{p_+,p_-}\)
  satisfies Zhu's \(c_2\)-cofiniteness condition, calculate the structure
  of the zero mode algebra (also known as Zhu's algebra) and classify all
  simple \(\mathcal{M}_{p_+,p_-}\)-modules. 
\end{abstract}

\setcounter{footnote}{0}
\def\thefootnote{\arabic{footnote}}

\newpage

\section{Introduction}

The theory of vertex operator algebras (VOA), which was developed by
Borcherds \cite{Borcherds:1983sq}, is an algebraic counter part to conformal field theory and 
gives an algebraic meaning to the notions of locality and operator product
expansions.
For general facts about VOAs we refer to \cite{Frenkel:2001,Nagatomo:2002,Matsuo:2005}.

Examples of conformal field theories on general Riemann surfaces for which
vertex operator algebraic descriptions are known, are given
by lattice VOAs, VOAs associated to integrable representations of affine Lie
algebras and VOAs associated to minimal representations of the Virasoro
algebra. The abelian categories associated to the representation theory of all
these examples are semi-simple and the number of irreducible representations
is finite. 

In order to define conformal field theories over a Riemann surface
associated to a VOA, the VOA needs to satisfy certain finiteness
conditions. Zhu found such a finiteness condition \cite{Zhu:1996}, which is now called
Zhu's \(c_2\)-cofiniteness condition. For a VOA satisfying Zhu's 
\(c_2\)-cofiniteness condition it is known that the abelian category of its
representations is both Noetherian and Artinian. Additionally the number of
simple objects is finite. In general this abelian category is not 
semisimple \cite{Zhu:1996,Frenkel:1992}, though so far the semi-simple case
is much better understood. 
Conformal field theories associated to VOAs with a non-semisimple
representation theory are called logarithmic or non-semisimple.

Examples of VOAs which satisfy Zhu's
\(c_2\)-cofiniteness condition but not semisimplicity are given by so called \(W_p\) theories for
which the representation category is by now well 
established \cite{Gaberdiel:1996np,Gaberdiel:2007jv,Adamovic:2007er,Nagatomo:2009xp,Tsuchiya:2012}.

In this paper we analyse a different example which generalises the
\(W_p\) theories. This example was originally defined in
\cite{Feigin:2006iv} and was called \(W_{p_+,p_-}\).
Unfortunately the letter ``W''
is rather overused in this context, so we will denote these VOAs by 
\(\mathcal{M}_{p_+,p_-}\) and call them ``extended
\(W\)-algebras of type \(\mathfrak{sl}_2\) at rational level''. The
construction of \(\mathcal{M}_{p_+,p_-}\) in \cite{Feigin:2006iv} is very
similar to the way \(W_p\) was constructed in \cite{Fuchs:2003yu}.

The \(\mathcal{M}_{p_+,p_-}\) are a family of VOAs parametrised by two coprime
integers \(p_+,p_-\geq2\). They are defined by means of a lattice VOA 
\(\mathcal{V}_{p_+,p_-}\) and two screening operators
\(\scr{+},\scr{-}\). 
The \(\mathcal{M}_{p_+,p_-}\) have the same central charge
\begin{align*}
  c_{p_+,p_-}=1-6\frac{(p_+-p_-)^2}{p_+p_-},
\end{align*}
as the minimal models. However, the Virasoro subtheory of \(\mathcal{M}_{p_+,p_-}\) is not isomorphic to
the minimal model VOA \(\text{MinVir}_{p_+,p_-}\). The minimal model VOA \(\text{MinVir}_{p_+,p_-}\) is
obtained from \(\mathcal{M}_{p_+,p_-}\) by taking a quotient.

A number of results in this paper have already been described in \cite{Feigin:2006iv}.
In \cite{Feigin:2006iv} the construction of integration cycles, over which products of screening operators are integrated,
are described by using a Kazhdan-Lusztig type correspondence
between the homology groups of configuration
spaces of points on the projective line with local coefficients and quantum
groups at roots of unity. However, this correspondence is
not yet well understood in this case 
\cite{Kazhdan:1993a,Kazhdan:1993b,Kazhdan:1994a,Kazhdan:1994b,Feigin:2006xa} though it 
has been explored for the \(W_p\) case \cite{Bushlanov:2009cv,Bushlanov:2011ha,Nagatomo:2009xp}.
In \cite{Adamovic:2010,Adamovic:2011} Adamovi\'c and Milas succeeded in
proving the \(c_2\) cofiniteness, in computing relations in Zhu's algebra and in
classifying all simple modules of the VOA \(\mathcal{M}_{2,p_-}\) for \(p_-\) odd.
They proved
these results by making use of a super VOA structure that exists for
\(p_+=2\) and by considering a field \(G(z)\) that is generally not local with
fields of the lattice VOA \(\mathcal{V}_{2,p_-}\), but local with fields in
\(\mathcal{M}_{2,p_-}\).
The zero mode of this field \(G(z)\) acts as a derivation on \(\mathcal{M}_{2,p_-}\).

In order to analyse the VOA \(\mathcal{M}_{p_+,p_-}\) for coprime
\(p_+,p_-\geq2\), we introduce free field VOAs over the discrete valuation
ring \(\mathcal{O}=\mathbb{C}[[\epsilon]]\), the ring of complex formal power series; and
its quotient field \(\mathcal{K}=\mathbb{C}((\epsilon))\), the field of
formal Laurent series. We discuss the representation theory and screening operators of these free
field VOAs as well as the construction of cycles on which products of
screening operators can be integrated. By using the theory of Jack polynomials
we show that the integration of products of screening operators over these
cycles defines intertwining operators of the Virasoro algebra over
\(\mathcal{O}\). These intertwining operators map between free field modules over \(\mathcal{O}\),
called Fock modules. Furthermore, by using the theory of Jack polynomials we
explicitly calculate all the data required for analysing the VOA structure of
\(\mathcal{M}_{p_+,p_-}\). Using formulae concerning integrals of Jack
polynomials, which were originally conjectured by Macdonald \cite{Macdonald:1987} and proved by Kadell \cite{Kadell:1997},
we construct derivations \(E\) and \(F\) of the VOA \(\mathcal{M}_{p_+,p_-}\),
which we call Frobenius homomorphisms. Thus we are able to analyse the zero
mode algebra, prove \(c_2\) cofiniteness and classify all simple modules of
the VOA \(\mathcal{M}_{p_+,p_-}\).

The results of this
paper form a necessary starting point for studying problems such as the \(\mathcal{M}_{p_+,p_-}\) representation theory;
the Kazhdan-Lusztig correspondence between \(\mathcal{M}_{p_+,p_-}\) and quantum groups; and conformal field
theories with \(\mathcal{M}_{p_+,p_-}\) symmetry on the Riemann sphere and elliptic curves --
 as was done in \cite{Belavin:1984,Tsuchiya:1987} for the
semi-simple case --
and more generally on moduli spaces of \(N\) point genus \(g\) stable curves.

This paper is organised as follows: In Section \ref{sec:definitions} we
introduce some basic notation and definitions.
We define VOAs with an emphasis on the Heisenberg and lattice VOAs as well as their screening operators.
We also briefly explain
how to construct the Poisson and zero mode algebra associated to a VOA as well as their implications for
the representation theory of a VOA.

In Section \ref{sec:deformationsec} we develop the techniques required for
analysing the extended \(W\)-algebra \(\mathcal{M}_{p_+,p_-}\). We construct cycles
over which products of screening operators can be integrated.
These integration cycles are elements of the homology groups of configuration spaces of \(N\) points on the
projective line, with local coefficients defined by the monodromy of products of screening operators.
Due to so called resonance problems homology and cohomology groups with these local coefficients exhibit very complicated
behaviour \cite{Orlik:2001,Varchenko:2003}.
To overcome these complications we deform the Heisenberg VOA, including its Virasoro field and screening operators, and
construct the theory over the ring \(\mathcal{O}\) and its 
field of fractions \(\mathcal{K}=\mathbb{C}((\epsilon))\).
The problem is thus translated into 
constructing well behaved cycles such that all matrix elements of integrals of
products of screening operators lie in \(\mathcal{O}\) rather than
\(\mathcal{K}\).
We show this by using the theory of Jack polynomials \cite{MacDonald:1999}.
By setting \(\epsilon=0\) we then obtain integration cycles over \(\mathbb{C}\) for products of screening operators.
By integrating these products of screening operators over the constructed
cycles we obtain local primary fields
\(\scrp{+}{r}(z), \scrp{-}{s}(z), r,s\in \mathbb{N}\) of conformal
weight 1. The zero mode operators \(\scrp{+}{r}, \scrp{-}{s}, r,s\in \mathbb{N}\) of these primary fields define Virasoro
intertwining operators.

In Section \ref{sec:Virrepthy} we review the decomposition of Fock modules as
Virasoro modules due to Feigin and Fuchs
\cite{Feigin:1984,Feigin:1988,FeFu:1990}.
By using the intertwining operators \(\scrp{+}{r}\) and \(\scrp{-}{s}\), we
construct all Virasoro singular vectors of these Fock modules
at central charge \(c_{p_+,p_-}\). We define operators \(E\) and \(F\) which we call
Frobenius homomorphisms. The Frobenius homomorphisms are Virasoro homomorphisms that map
between the kernels of \(\scrp{+}{r}\) and \(\scrp{-}{s}\) respectively and define
derivations of the extended \(W\)-algebra \(\mathcal{M}_{p_+,p_-}\). The
Frobenius homomorphisms \(E,F\) are motivated by so called divided power
operators, constructed in \cite{Feigin:2006iv} by using quantum groups.
The definition and properties of the Frobenius homomorphisms are stated in Theorem \ref{sec:frobopthm}.

In Section \ref{sec:Mppalg} the extended \(W\)-algebra
\(\mathcal{M}_{p_+,p_-}\) is introduced. We give a decomposition of
\(\mathcal{M}_{p_+,p_-}\) as a Virasoro module, determine a generating set of
fields, analyse its zero mode and Poisson algebra and classify and construct
all simple \(\mathcal{M}_{p_+,p_-}\) modules. The screening operators
\(\scrp{+}{r}, \scrp{-}{s}\) and the Frobenius homomorphisms \(E,F\) are crucial to
all these calculations. The structure of \(\mathcal{M}_{p_+,p_-}\) as a VOA is
stated in Theorem \ref{sec:indecstructure}, the structure of the zero mode
algebra \(A_0(\mathcal{M}_{p_+,p_-})\) is stated in Theorem
\ref{sec:zhualgrelations} and 
the classification of all simple \(A_0(\mathcal{M}_{p_+,p_-})\)- and
\(\mathcal{M}_{p_+,p_-}\)-modules is stated in Theorem \ref{sec:classificationofmodules}.

In Section \ref{sec:conclusion} we give our conclusions and state a list of future problems and conjectures
associated to conformal field theories with \(\mathcal{M}_{p_+,p_-}\) symmetry.

\subsubsection*{Acknowledgements}

This work was supported by the World Premier International Research Center
Initiative (WPI Initiative), MEXT, Japan.
Both authors are supported by the Grant-in-Aid for JSPS Fellows number 2301793.
The first author is supported by the JSPS Grant-in-Aid for Scientific Research
number 22540010. The first author would like to thank A.~M.~Semikhatov and I.~Y.~Tipunin for
helpful discussions as well as J.~Shiraishi for helpful discussions about the properties of Jack polynomials.
The second
author is supported the JSPS fellowship for foreign researchers and thanks the
University of Tokyo, Kavli IPMU for its hospitality.
The second author would like to thank T.~Milanov, T.~Abe,
C.~Schnell and C.~Schmidt-Collinet for helpful discussions.

\section{Basic definitions and notation}
\label{sec:definitions}

In this section we review basic definitions and notation for VOAs, in
particular Heisenberg and lattice VOAs.

\subsection{Vertex operator algebras}

For a detailed discussion of vertex operator algebras see \cite{Nagatomo:2002,Matsuo:2005,Frenkel:1992}.

\begin{definition}
  A tuple \((V,|0\rangle,T,Y)\) is called a \emph{vertex operator algebra} (VOA for short)
  where
  \begin{enumerate}
  \item \(V\) is a complex non-negative integer graded vector space
    \begin{align*}
      V=\bigoplus_{n=0}^\infty V[n]\,,
    \end{align*}
    called \emph{the vacuum space of states}.
  \item \(|0\rangle\in V[0]\) is called \emph{the vacuum vector}.
  \item \(T\in V[2]\) is called \emph{the conformal vector}.
  \item \(Y\) is a \(\mathbb{C}\)-linear map
    \begin{align*}
      Y: V&\rightarrow \operatorname{End}_{\mathbb{C}}(V)[[z,z^{-1}]]
    \end{align*}
    called \emph{the vertex operator map}. 
      \end{enumerate}
These data are subject to the axioms:
  \begin{enumerate}
  \item Each homogeneous subspace \(V[n]\) of the space of states is finite dimensional and in particular
    \(V[0]\) is spanned by the vacuum vector.
  \item For each \(A\in V[h]\) there exists a Laurent expansion
    \begin{align*}
      Y(A;z)=A(z)=\sum_{n\in\mathbb{Z}} A[n]z^{-n-h}\,,
    \end{align*}
    where \(Y(A;z)\) is called \emph{a field} and
    the \(A[n]\) are called \emph{field modes}. 
    Each field satisfies the \emph{the state field correspondence}
    \begin{align*}
      Y(A;z)|0\rangle -A \in V[[z]]z\,.
    \end{align*}
    The field corresponding to the vacuum is the identity field
    \begin{align*}
      Y(|0\rangle;z)=\id_V\,.
    \end{align*}
    For certain special fields, such as the Virasoro field below, we denote the mode as an index instead of in brackets.
  \item The field modes of the field corresponding to the conformal vector
    \begin{align*}
      Y(T;z)=T(z)=\sum_{n\in \mathbb{Z}}L_n z^{-n-2}\,,
    \end{align*}
    satisfy the commutation relations of \emph{the Virasoro algebra} with
    fixed \emph{central charge} \(c=c_V\)
    \begin{align*}
      [L_m,L_n]=(m-n)L_{m+n}+\frac{c_V}{12}(m^3-m)\delta_{m+n,0}\,.
    \end{align*}
    The field \(T(z)\) is called \emph{the Virasoro field}.
  \item The zero mode of the Virasoro algebra \(L_0\) acts semi-simply on \(V\) and
    the eigenvalues of \(L_0\) define the grading of \(V\), that is,
    \begin{align*}
      V[h]=\{A\in V| L_0 A= h A\}\,.
    \end{align*}
  \item The Virasoro generator \(L_{-1}\) acts as the derivative with respect
    to \(z\)
    \begin{align*}
      \frac{\d}{\d z}Y(A;z)=Y(L_{-1}A;z)\,,
    \end{align*}
    for all \(A\in V\).
  \item For any two elements \(A,B\in V\) the fields \(Y(A;z)\) and \(Y(B;w)\) are local,
    {\it i.e.} there exists a sufficiently large \(N\in\mathbb{Z}\) such that
    \begin{align*}
      (z-w)^N[Y(A;z),Y(B;w)]=0\,,
    \end{align*}
    as elements of \(\operatorname{End}(V)[[z,z^{-1},w,w^{-1}]]\).
  \item For a homogeneous element \(A\in V[h]\) and an element \(B\in V\) the
    fields \(Y(A;z)\) and \(Y(B;w)\) satisfy 
    \emph{the operator product expansion}
    \begin{align*}
      Y(A;z)Y(B;w)&=Y(Y(A;z-w)B;w)\\\nonumber
      &=\sum_{n\in \mathbb{Z}}Y(A[n] B;w)(z-w)^{-n-h}\,.
    \end{align*}
  \end{enumerate}
\end{definition}
When there is no chance of confusion we will refer to a VOA just by its graded vector space
\(V\).
\begin{remark}\ 
  \begin{enumerate}
  \item For \(A\in V[h]\) the Virasoro generators \(L_0\) and \(L_{-1}\) satisfy
  \begin{align*}
    [L_{-1},A[n]]&=-(n+h-1)A[n-1]\\\nonumber [L_0,A[n]]&=-n A[n]\,.
  \end{align*}
\item The operator product expansion of the Virasoro field with itself is
  \begin{align*}
    T(z)T(w) = \frac{c_V/2}{(z-w)^4}+\frac{2}{(z-w)^2}T(w)+\frac{1}{z-w}\partial T(w)+\cdots\,.
  \end{align*}
\item For any homogeneous element \(A\in V[h]\) such that \(L_n A=0\) for \(n\geq 1\), the operator
  product expansion of the Virasoro field with \(Y(A;z)\) is
  \begin{align*}
    T(z)Y(A;w)=\frac{h}{(z-w)^2}Y(A;w)+\frac{1}{z-w}\partial Y(A;w)+\cdots\,.
  \end{align*}
  Such fields \(Y(A;w)\) are called primary fields.
\end{enumerate}
\end{remark}

\begin{definition}
  A \emph{VOA module} \(M\) is a vector space that carries a representation \(Y_M\)
  \begin{align*}
    Y_M: V\rightarrow \operatorname{End}_{\mathbb{C}}(M)[[z,z^{-1}]]\,,
  \end{align*}
  of the vertex operator \(Y\),  such that for all \(A,B\in V\) and \(C\in M\)
  \begin{enumerate}
  \item \(Y_M(A,z)C\in M((z))\),
  \item \(Y_M(|0\rangle;z)=\id_M\), is the identity on \(M\),
  \item \(Y_M(A;z)Y_M(B;w)=Y_M(Y(A;z-w) B;w)\)\,.
  \end{enumerate}
\end{definition}

Let \(O(V)\) be the vector subspace of \(V\) spanned by vectors
\begin{align*}
  A\circ B= \res_{z=0}Y(A;z)B\frac{(1+z)^{h_A}}{z^2}\d z
  =\sum_{n=0}^{h_A}\binom{h_A}{n}A[-n-1]B\,,
\end{align*}
for \(A\in V[h_A], B\in V\). Furthermore, let \(A\ast B\) be the binary operation
\begin{align*}
  A\ast B= \res_{z=0}Y(A;z)B\frac{(1+z)^{h_A}}{z}\d z
  =\sum_{n=0}^{h_A}\binom{h_A}{n}A[-n]B\,.
\end{align*}

\begin{prop}\label{sec:zhuproperties}
  The space \(A_Z(V)=V/O(V)\) is called \emph{Zhu's algebra} and
  carries the structure of an associative \(\mathbb{C}\) algebra.
  Let \([A],[B]\in A_Z(V)\) denote the classes represented by \(A,B\in V\), then the multiplication in \(A_Z(V)\) is given by
  \begin{align*}
    [A]\cdot [B]=[A\ast B]\,.
  \end{align*}
  \begin{enumerate}
  \item The unit of \(A_Z(V)\) is the class of the vacuum state \(1=[|0\rangle]\).
  \item The class of the conformal vector \([T]\) lies in the centre of \(A_Z(V)\).
  \item \((L_0+L_{-1})A\in O(V)\) for all \(A\in V\).
  \item There is a 1 to 1 correspondence between finite dimensional
    simple \(A_Z(V)\)-modules and 
    simple \(V\)-modules.
    The simple \(A_Z(V)\)-modules are isomorphic to the homogeneous space of
    least conformal weight of the corresponding simple \(V\)-module.
  \item For \(A\in V[h_A]\), \(B\in V\) and \(m\geq n\geq 0\)
    \begin{align*}
      \res_{z=0} Y(A;z)B\frac{(1+z)^{h_A+n}}{z^{2+m}}\d z \in O(V)\,.
    \end{align*}
  \end{enumerate}
\end{prop}

There is an equivalent definition of Zhu's algebra as a quotient of the algebra of zero modes.
Let 
\begin{align*}
  U(V)&=\bigoplus_{d\in\mathbb{Z}}U(V)[d]\\
  U(V)[d]&=\{P\in U(V)|[L_0,P]=-d\}
\end{align*}
be the graded associative algebra of modes of all the fields in \(V\). 
This algebra is called the \emph{current algebra} \cite{Nagatomo:2002}.
Furthermore, let \(F_p(V)\) be the descending filtration
\begin{align*}
  F_p(V)=\bigoplus_{d\leq p}U(V)[d]
\end{align*}
of \(U(V)\) and let
\begin{align*}
  \mathcal{I}=\overline{U(V)\cdot F_{-1}(V)}
\end{align*}
be the closure of the \(U(V)\) left ideal generated by \(F_{-1}(V)\), then
\begin{align*}
  I=\mathcal{I}\cap F_0(V)
\end{align*}
is a closed two sided \(F_0(V)\) ideal.
\begin{definition}
  The \emph{zero mode algebra} is the quotient algebra
  \begin{align*}
    A_0(V)=F_0(V)/I\,.
  \end{align*}
\end{definition}
\begin{prop}\ 
  \begin{enumerate}
  \item There exists a canonical surjective \(\mathbb{C}\) algebra homomorphism
    \begin{align*}
      U(V)[0]&\rightarrow A_0(V)
    \end{align*}
    which maps an element  \(P\in U(V)[0]\subset F_0(V)\) to its class in the
    zero mode algebra \(A_0(V)=F_0(V)/I\).
  \item There exists a well defined canonical isomorphism of \(\mathbb{C}\) algebras from Zhu's algebra \(A_Z(V)\)
    to the zero mode algebra \(A_0(V)\), such that for \(A\in V\)
    \begin{align*}
      A_Z(V)&\rightarrow A_0(V)\\
      [A]&\mapsto [A[0]]\,.
    \end{align*}
  \end{enumerate}
\end{prop}
\begin{remark}
  Since the zero mode algebra and Zhu's algebra are canonically isomorphic, we
  identify these two algebras and denote both by \(A_0(V)\). For \(A\in V\) we
  denote the corresponding class in \(A_0(V)\) by \([A]=[A[0]]\).
\end{remark}

Let \(c_2(V)\) be the subspace of \(V\) given by
\begin{align*}
  c_2(V)=\text{span}\{A[-(h_A+n)]B|A\in V[h_A], B\in V, n\geq 1\}\,.
\end{align*}

\begin{definition}
  The VOA \(V\) is said to satisfy \emph{Zhu's \(c_2\)-cofiniteness condition} if the quotient
  \begin{align*}
    \mathfrak{p}(V)=V/c_2(V)
  \end{align*}
  is finite dimensional.
\end{definition}

\begin{prop}
  The quotient space \(\mathfrak{p}(V)=V/c_{2}(V)\) carries the structure of a commutative
  Poisson algebra. Let \([A]_{\mathfrak{p}},[B]_{\mathfrak{p}}\) be the classes of \(A\in V[h_A]\) and \(B\in V\), then
      the multiplication and bracket are given by
  \begin{align*}
    [A]_{\mathfrak{p}}\cdot [B]_{\mathfrak{p}}&=[A[-h_A]B]_{\mathfrak{p}}\,,\\
    \{[A]_{\mathfrak{p}},[B]_{\mathfrak{p}}\}&=[A[-h_A+1]B]_{\mathfrak{p}}\,.
  \end{align*}
\end{prop}

\begin{remark}
  The space \(O(V)\) is not spanned by homogeneous vectors; therefore Zhu's
  algebra is not graded, it is merely filtered by conformal weight. The space
  \(c_2(V)\), on the other hand, is spanned by homogeneous vectors, so the Poisson algebra 
  \(\mathfrak{p}(V)\) is graded by conformal weight.
\end{remark}

Let \(F_0(A_0(V))\subset F_1(A_0(V))\subset\cdots\) be the filtration of
\(A_0(V)\) by conformal weight, that is
\begin{align*}
  F_p(A_0(V))=\left\{[A]| A\in \bigoplus_{h=0}^pV[h]\right\}\,.
\end{align*}
Then
the gradification of \(A_0(V)\) is the graded algebra
\begin{align*}
  \text{Gr}(A_0(V))=\bigoplus_{p\geq0} G_p(A_0(V))\,,
\end{align*}
where
\begin{align*}
  G_0(A_0(V))&=F_0(A_0(V)),&G_p(A_0(V))=F_p(A_0(V))/F_{p-1}(A_0(V)),\, p\geq1\,.
\end{align*}

\begin{prop}\label{sec:zhutopoisson}
  There exists a surjection of graded \(\mathbb{C}\) algebras
  \begin{align*}
    \mathfrak{p}(V)\rightarrow \text{Gr}(A_0(V))\,.
  \end{align*}
\end{prop}
For proofs of the properties of Zhu's algebra see \cite{Zhu:1996,Frenkel:1992}
and for a proof of the existence of the canonical isomorphism between 
\(A_Z(V)\) and \(A_0(V)\) see \cite{Nagatomo:2002}. The connection between Zhu's algebra and the zero mode algebra
has also been commented on in the physics literature \cite{Brungs:1998ij}.

\subsection{The Heisenberg vertex operator algebra}
\label{sec:heisenbergalg}

The Heisenberg VOA is a central building block for all the VOAs considered in this paper.
Before we define the Heisenberg VOA, we must first define the Heisenberg algebra and its highest weight modules,
called Fock modules.
\begin{definition}\ 
  \begin{enumerate}
  \item Let \(U(\mathfrak{b}_\pm)\) and \(U(\mathfrak{b}_0)\) be
    \(\mathbb{Z}\)-graded polynomial algebras over \(\mathbb{C}\) given by
    \begin{align*}
      U(\mathfrak{b}_\pm)&=\mathbb{C}[b_{\pm1},b_{\pm2},\dots]\,,&
      U(\mathfrak{b}_0)&=\mathbb{C}[b_0]\,,
    \end{align*}
    where the degree of \(b_{n}\) is \(\degr(b_n)=-n\).
  \item The associative \(\mathbb{Z}\)-graded degree wise 
    completed \(\mathbb{C}\)-algebra \(U(\overline{\mathfrak{b}})\) is given by
    \begin{align*}
      U(\overline{\mathfrak{b}})=U(\mathfrak{b}_-)\hat\otimes U(\mathfrak{b}_+)
    \end{align*}
    as a vector space, where \(\hat\otimes\) denotes the degree wise completed
    tensor product. The algebra structure on \(U(\overline{\mathfrak{b}})\) is
    defined by the \emph{Heisenberg commutation relations}
    \begin{align*}
      [b_m,b_n]=m\delta_{m,-n}\cdot\id\,,\ m,n\in\mathbb{Z}\setminus\{0\}\,.
    \end{align*}
  \item \emph{The Heisenberg algebra} is the \(\mathbb{Z}\)-graded associative
    algebra \(U(\mathfrak{b})\) given by
    \begin{align*}
      U(\mathfrak{b})=U(\overline{\mathfrak{b}})\otimes U(\mathfrak{b}_0)
    \end{align*}
    and satisfies the commutation relations
    \begin{align*}
      [b_m,b_n]=m\delta_{m,-n}\cdot\id\,,\ m,n\in\mathbb{Z}\,.
    \end{align*}
  \end{enumerate}
\end{definition}

 \begin{definition}
  Let \(\beta\in \mathbb{C}\).
  \begin{enumerate}
  \item We define the left \(U(\mathfrak{b})\)-module \(F_\beta\) -- called a
    \emph{Fock module}.  It is generated by the state \(|\beta\rangle\), which
    satisfies
    \begin{align*}
      b_0|\beta\rangle&=\beta|\beta\rangle, &b_n|\beta\rangle&=0\,,\quad n\geq 1\,,
    \end{align*}
    such that
    \begin{align*}
      U(\mathfrak{b}_-)&\rightarrow F_\beta\\\nonumber
      P&\mapsto  P|\beta\rangle
    \end{align*}
    is an isomorphism of complex vector spaces.
  \item We define the right \(U(\mathfrak{b})\)-module \(F_\beta^\vee\) -- called a
    dual Fock module.  It is generated by the state \(\langle\beta|\), which
    satisfies
    \begin{align*}
      \langle\beta|b_0&=\langle\beta|\beta,&      \langle\beta|b_n&=0\,,\quad n\leq -1\,,\\
    \end{align*}
    such that
    \begin{align*}
      U(\mathfrak{b}_+)&\rightarrow F_\beta^\vee\\\nonumber
      P&\mapsto \langle\beta|P
    \end{align*}
    is an isomorphism of complex vector spaces.
  \item The two Fock modules \(F_\beta,F_\beta^\vee\) are equipped with an inner product
    \begin{align*}
      F_\beta^\vee \times F_\beta\rightarrow\mathbb{C}\,,
    \end{align*}
    characterised by \(\langle \beta|\beta\rangle=1\). 
  \end{enumerate}
  The parameter \(\beta\) is called \emph{the Heisenberg weight}.
\end{definition}

\begin{remark}
  Let \(\hat b\) be the conjugate of \(b_0\), that is, \(\hat b\) satisfies the
  commutation relations
  \begin{align*}
    [b_m,\hat b]=\delta_{m,0}1\,.
  \end{align*}
  For each \(\gamma\in\mathbb{C}\), \(e^{\gamma\hat b}\) defines a 
  Heisenberg weight shifting map
  \begin{align*}
    e^{\gamma\hat b}:F_\beta\rightarrow F_{\beta+\gamma}\,,
  \end{align*}
  that satisfies
  \begin{enumerate}
  \item \(e^{\gamma\hat b}|\beta\rangle=|\beta+\gamma\rangle\),
  \item \(e^{\gamma\hat b}\) commutes with \(U(\mathfrak{b}_-)\) and
    \(U(\mathfrak{b}_+)\).
  \end{enumerate}
\end{remark}

For \(\alpha_0\in\mathbb{C}\) let
\begin{align*}
  T=\frac12 (b_{-1}^2+\alpha_0 b_{-2})|0\rangle\in F_0\,.
\end{align*}
\begin{prop}
  The Fock space \(F_0\) carries the structure of a VOA, with
  \begin{align*}
    Y(|0\rangle;z)&=\id\,,\qquad    Y(b_{-1}|0\rangle;z)=b(z)=\sum_{n\in\mathbb{Z}}b_{n}z^{-n-1}\,,\\
    Y(T;z)&=T(z)=\frac12 (:b(z)^2:+\alpha_0 \partial b(z))\,.
  \end{align*}
  We denote this VOA by
  \(\mathcal{F}_{\alpha_0}=(F_0,|0\rangle,T,Y)\) and call it the
  \emph{Heisenberg VOA}.
\end{prop}

\begin{remark}
The operator product expansion of the field \(b(z)\) with itself
    is given by
    \begin{align*}
      b(z)b(w)=\frac{1}{(z-w)^2}+\cdots
    \end{align*}
and the central charge of \(\mathcal{F}_{\alpha_0}\) is given by
  \begin{align*}
    c_{\alpha_0}=1-3\alpha_0^2\,.
  \end{align*}
The roots of the polynomial \(\kappa^2-(\tfrac{\alpha_0}{2}+1)\kappa+1\) are called the level of \(\mathcal{F}_{\alpha_0}\).
The values of \(\alpha_0\) of interest to this paper will result in positive rational levels.
\end{remark}

\begin{prop}\label{sec:primaries}\ 
  \begin{enumerate}
  \item The Fock module \(F_\beta\) is a simple \(\mathcal{F}_{\alpha_0}\)-module for all \(\beta\in\mathbb{C}\).
  \item The abelian category of \(\mathcal{F}_{\alpha_0}\)-modules, 
    \(\mathcal{F}_{\alpha_0}\)-mod, is semisimple\ and the set of simple
    objects is given by \(\{F_\beta\}_{\beta\in\mathbb{C}}\).
    \item The generating state \(|\beta\rangle\) of \(F_\beta\) satisfies
      \begin{align*}
        L_0|\beta\rangle&=h_\beta|\beta\rangle\,,&L_n|\beta\rangle&=0\,,\quad n\geq1
      \end{align*}
      where
      \begin{align*}
        h_\beta=\frac12 \beta(\beta-\alpha_0)\,.
      \end{align*}
    \end{enumerate}
\end{prop}

 We introduce an auxiliary field \(\phi(z)\), 
 which is a formal primitive of \(b(z)\)
 \begin{align*}
   \phi(z)=\hat b+b_0\log z-\sum_{n\neq 0}\frac{b_n}{n}z^{-n}
 \end{align*}
 and which satisfies the operator product expansion
\begin{align*}
  \phi(z)\phi(w)=\log(z-w)+\cdots\,.
\end{align*}

\begin{definition}
  For all \(\beta\in\mathbb{C}\), let \(V_\beta(z)\) denote the field
  \begin{align*}
    V_\beta(z)&=:e^{\beta \phi(z)}:=e^{\beta\hat b}z^{\beta
      b_0}\overline{V_\beta}(z)\\\nonumber
    \overline{V_\beta}(z)&=e^{\beta\sum_{n\geq1}\tfrac{b_{-n}}{n}z^n}
    e^{-\beta\sum_{n\geq1}\tfrac{b_n}{n}z^{-n}}\in U(\overline{\mathfrak{b}})\hat\otimes\, \mathbb{C}[z,z^{-1}]\,,
  \end{align*}
  where \(z^{\beta b_0}=\exp(\beta b_0\log(z))\).
\end{definition}
The multivaluedness of \(z^{\beta b_0}\) will be
addressed in Section \ref{sec:renormcycles}.

\begin{prop}
  For all \(\beta\in\mathbb{C}\) the fields \(V_\beta(z)\) satisfy:
  \begin{enumerate}
  \item For \(\gamma\in \mathbb{C}\) the field \(V_\beta(z)\) defines a map
    \begin{align*}
      V_\beta(z):F_\gamma\rightarrow F_{\beta+\gamma}[[z,z^{-1}]]z^{\beta\gamma}\,.
    \end{align*}
  \item The field \(V_\beta(z)\) corresponds to the state 
    \(|\beta\rangle\), that is,
    \begin{align*}
      V_\beta(z)|0\rangle-|\beta\rangle\in F_\beta[[z]]z\,.
    \end{align*}
  \item The field \(V_\beta(z)\) is primary and satisfies the operator product expansion
    \begin{align*}
      T(z)V_\beta(w)=\frac{h_\beta}{(z-w)^2}V_\beta(w)+\frac{1}{z-w}\partial V_\beta(w)+\cdots\,.
    \end{align*}
  \item For \(\beta_1,\dots,\beta_k\in\mathbb{C}\) the product of the \(k\)
    fields \(V_{\beta_i}(z_i)\) satisfies the operator product expansion
    \begin{align*}
      \prod_{i=1}^k V_{\beta_i}(w_i)=e^{\sum_{i=1}^k\beta_i\hat b}\prod_{i=1}^k w_i^{\beta_i b_0}
      \prod_{1\leq i<j\leq k}(w_j-w_i)^{\beta_i \beta_j}
      :\prod_{i=1}^k \overline{V}_{\beta_i}(w_i):\,,
    \end{align*}
    where \(:\prod_{i=1}^k \overline{V}_{\beta_i}(w_i):\) is an element of 
    \(U(\overline{\mathfrak{b}})\hat\otimes \mathbb{C}[z_1,z_1^{-1},\dots,z_k,z_k^{-1}]\)
    \begin{align*}
      :\prod_{i=1}^k \overline{V}_{\beta_i}(w_i):=
      e^{\sum_{i=1}^k\beta_i\sum_{n\geq1}\tfrac{b_{-n}}{n}w_i^n}e^{-\sum_{i=1}^k\beta_i\sum_{n\geq1}\tfrac{b_n}{n}w_i^{-n}}\,.
    \end{align*}
  \end{enumerate}
\end{prop}

\begin{remark}
  The Virasoro generator \(L_0\) is diagonalisable on the Fock spaces \(F_\beta\). The eigenvalues of \(L_0\)
  define a grading of \(F_\beta\)
  \begin{align*}
    F_\beta=\bigoplus_{n\geq 0}F_\beta[h_\beta+n]\,,
  \end{align*}
  where
  \begin{align*}
    F_\beta[h]=\{u\in F_\beta| L_0 u=h u\}\,.
  \end{align*}
  The dimension of these homogeneous subspaces is
  \begin{align*}
    \dim F_\beta[h_\beta+n]=p(n)\,,
  \end{align*}
  where \(p(n)\) is the number of partitions of the integer \(n\).
\end{remark}

The Heisenberg algebra \(U(\mathfrak{b})\) admits an anti-involution
\begin{align*}
  \sigma: U(\mathfrak{b})&\rightarrow
  U(\mathfrak{b})\\\nonumber
  b_n&\mapsto \delta_{n,0}\alpha_0-b_{-n}\,,
\end{align*}
such that
\begin{align*}
  \sigma(L_n)=L_{-n}\,.
\end{align*}

\begin{definition}
  The dual Fock module \(F_\beta^\vee\) is isomorphic to the graded dual
  space of \(F_\beta\)
  \begin{align*}
    F_\beta^\vee=\bigoplus_{n\geq 0}\hom(F_\beta[h_\beta+n],\mathbb{C})\,.
  \end{align*}
  The anti-involution \(\sigma\) induces the structure of a left
  \(U(\mathfrak{b})\)-module on \(F_\beta^\vee\) by
  \begin{align*}
    \langle b_n\phi,u\rangle=\langle\phi,\sigma(b_n)u\rangle 
  \end{align*}
  for all \(n\in\mathbb{Z},\ \phi\in F_\beta^\ast,\  u\in F_\beta\). 
  We denote this left module by \(F_\beta^\ast\) and call it the
  \emph{contragredient dual of} \(F_\beta\).
\end{definition}

\begin{prop}
  Taking the contragredient defines  a contravariant functor
  \begin{align*}
    \ast:\mathcal{F}_{\alpha_0}\text{-mod}\rightarrow \mathcal{F}_{\alpha_0}\text{-mod}\,,
  \end{align*}
  satisfying
  \begin{align*}
    F_\beta^\ast=F_{\alpha_0-\beta}\,,
  \end{align*}
  such that \((F_\beta^\ast)^\ast=F_\beta\).
\end{prop}

\subsection{The lattice vertex operator algebra {\boldmath\(\mathcal{V}_{p_+,p_-}\)\unboldmath}}
\label{sec:Vpppm}
The lattice VOA \(\mathcal{V}_{p_+,p_-}\) is defined for special values of the parameter \(\alpha_0\) and
by restricting the weights of the Fock spaces to a certain lattice. Let 
\(p_+,p_-\geq2\) be two coprime integers, such that
\begin{align*}
  \alpha_+&=\sqrt{\frac{2p_-}{p_+}}&\alpha_-&=-\sqrt{\frac{2p_+}{p_-}}&\alpha_0&=\alpha_++\alpha_-\\
  \kappa_+&=\frac{\alpha_+^2}{2}=\frac{p_-}{p_+}&\kappa_-&=\frac{\alpha_-^2}{2}=\frac{p_+}{p_-}
  &\alpha&=p_+\alpha_+=-p_-\alpha_-\,.
\end{align*}
The parameters \(\kappa_+=\kappa_-^{-1}\) are the roots of the polynomials \(\kappa^2-(\tfrac{\alpha_0}{2}+1)\kappa+1\) 
and are called the \emph{level} of \(\mathcal{V}_{p_+,p_-}\).
Next we define the rank 1 lattices
\begin{align*}
  Y&=\mathbb{Z}\sqrt{2p_+p_-}&X=\hom_{\mathbb{Z}}(Y,\mathbb{Z})=\mathbb{Z}\frac{1}{\sqrt{2p_+p_-}}\,.
\end{align*}
Both \(\alpha_+\) and \(\alpha_-\) lie in \(X\) and we define the 
parametrisation
\begin{align*}
  \beta_{r,s}=\frac{1-r}{2}\alpha_++\frac{1-s}{2}\alpha_-\,,\quad r,s\in\mathbb{Z}\,.
\end{align*}
Note that \(\beta_{r,s}=\beta_{r+p_+,s+p_-}\) and we use the shorthand
\begin{align*}
  \beta_{r,s;n}=\beta_{r-np_+,s}=\beta_{r,s+np_-}\,.
\end{align*}
When denoting the weights of Fock spaces, we will only write the indices and
drop the ``\(\beta\)'' from \(\beta_{r,s;n}\) or \(\beta_{r,s}\),
that is
\begin{align*}
  F_{\beta_{r,s;n}}&=F_{r,s;n}\,,&F_{\beta_{r,s}}&=F_{r,s}\,.
\end{align*}

\begin{remark}\ 
  \begin{enumerate}
  \item After specialising the parameter \(\alpha_0\) to
    \(\alpha_0=\alpha_++\alpha_-\)
    we denote the Heisenberg VOA by \(\mathcal{F}_{p_+,p_-}\) instead of
    \(\mathcal{F}_{\alpha_0}\).
  \item Taking the contragredient of a Fock space \(F_{r,s;n}\) reverses
    the sign of the indices:
    \begin{align*}
      F_{r,s;n}^\ast = F_{-r,-s;-n}\,.
    \end{align*}
  \end{enumerate}
\end{remark}

\begin{definition}
  The \emph{lattice VOA \(\mathcal{V}_{p_+,p_-}\)} is the tuple 
  \((V_{[0]},|0\rangle,\tfrac12(b_{-1}^2-\alpha_0 b_2)|0\rangle,Y)\), where
    the underlying vector space of \(\mathcal{V}_{p_+,p_-}\) is given by
    \begin{align*}
      V_{[0]}=\bigoplus_{\beta\in
        Y}F_\beta=\bigoplus_{n\in\mathbb{Z}}F_{n\alpha}\,.
    \end{align*}
    The fields corresponding to \(|0\rangle, b_{-1}|0\rangle\) and \(T\) are
    those of \(\mathcal{F}_{p_+,p_-}\) and
    \begin{align*}
      Y(|\beta\rangle;z)&=V_\beta(z),\ \beta\in Y\,.
    \end{align*}
\end{definition}
\begin{remark}
  The relations for the vertex operator map in the definition above uniquely define the VOA structure of \(\mathcal{V}_{p_+,p_-}\).
  The central charge of the Virasoro field \(T(z)\) is
    \begin{align*}
      c_{p_+,p_-}=1-6\frac{(p_+-p_-)^2}{p_+p_-}\,,
    \end{align*}
    and the conformal weight of the generating state \(|\beta\rangle\)
    of a Fock module \(F_\beta\) is \(h_\beta=\frac12\beta(\beta-\alpha_0)\).
    We define \(h_{r,s}=h_{\beta_{r,s}},\ r,s\in\mathbb{Z}\), then we have
    \begin{align*}
      h_{r,s}=\frac{r^2-1}{4}\kappa_+-\frac{rs-1}{2}+\frac{s^2-1}{4}\kappa_-\,.
    \end{align*}
\end{remark}
\begin{prop}
  The abelian category \(\mathcal{V}_{p_+,p_-}\)-mod 
  of \(\mathcal{V}_{p_+,p_-}\)-modules is semi-simple with
  \(2p_+p_-\) simple objects. These simple objects are parametrised 
  by the classes of \(X/Y\)
  \begin{align*}
    V_{[\beta]}=\bigoplus_{\gamma\in \beta+Y}F_{\gamma}\,,\quad \beta\in X\,.
  \end{align*}
\end{prop}
\begin{remark}
  Using the \(\beta_{r,s;n}\) we parametrise the simple \(\mathcal{V}_{p_+,p_-}\)-modules as
  \begin{align*}
    V_{r,s}^+&=\bigoplus_{n\in\mathbb{Z}} F_{r,s;2n}\\\nonumber
    V_{r,s}^-&=\bigoplus_{n\in\mathbb{Z}} F_{r,s;2n+1}\,,
  \end{align*}
  for \(1\leq r\leq p_+,\) \(1\leq s\leq p_-\). In this notation 
  \(V_{[0]}=V_{1,1}^+\).
\end{remark}

By the formula for conformal weights in Proposition \ref{sec:primaries}, the two
Heisenberg weights \(\alpha_+,\alpha_-\) have conformal weight
\(h_{\alpha_\pm}=1\).
These are the only Heisenberg weights with conformal weight 1.
We call the fields corresponding to \(|\alpha_\pm\rangle\) screening operators and denote them by
\begin{align*}
  \scr{\pm}(z)=:e^{\alpha_\pm \phi(z)}:.
\end{align*}
Since \(h_{\alpha_\pm}=1\), these fields define intertwining operators,
that is, the map
\begin{align*}
  \scr{\pm}=\oint \scr{\pm}(z) \text{d}z:V_{[0]}\rightarrow V_{[\alpha_\pm]}
\end{align*}
is \(\mathbb{C}\)-linear and commutes with the Virasoro algebra.
Note that since \(\alpha_\pm\notin Y\) the fields \(\scr{\pm}(z)\) do not belong to \(\mathcal{V}_{p_+,p_-}\).
We will later define the extended \(W\)-algebra \(\mathcal{M}_{p_+,p_-}\) as
the subVOA of \(\mathcal{V}_{p_+,p_-}\) given by
the intersection of the kernels of \(\scr{+}\) and \(\scr{-}\). Screening operators
were originally developed by Dotsenko and Fateev \cite{Dotsenko:1984,Dotsenko:1985}.

\section{Deformation of screening operators}
\label{sec:deformationsec}

For the purposes of this paper it is necessary to consider integrals of products of screening operators and not just
the residues of individual screening operators. In order to perform these integrals one needs
to consider homology groups of configuration spaces of \(N\) points on the
projective line with local coefficients. It is necessary
to use local coefficients because these products of screening operators are not single valued, but have non-trivial monodromies that
are roots of unity. The homology groups with such local coefficients exhibit very complicated behaviour and in order to make them
tractable we deform the Heisenberg VOA, that is, we deform its conformal
structure and its screening operators. The associated local systems then 
no longer exhibits
monodromy at roots of unity and the homology groups of these deformed local systems are very simple. After analysing the deformed case in
detail, we will show that one can take a meaningful limit to the undeformed case.

\subsection{Deformation of the Heisenberg vertex operator algebra}
\label{sec:deformheisenberg}

Let \(\mathcal{O}=\mathbb{C}[[\epsilon]]\) be the ring of formal power series with coefficients in
\(\mathbb{C}\) and let \(\mathcal{K}=\mathbb{C}((\epsilon))\) be the fraction
field of \(\mathcal{O}\).
To any module over \(\mathcal{O}\) we can associate a \(\mathbb{C}\) vector
space by taking the tensor product \(-\otimes_{\mathcal{O}}\mathbb{C}\), that
is, by setting \(\epsilon\) to zero.
We enlarge the ground field \(\mathbb{C}\) of the Heisenberg algebra \(U(\mathfrak{b})\) introduced in Section \ref{sec:heisenbergalg} to
the rings \(\mathcal{O}\) and \(\mathcal{K}\).
\begin{definition}
  Let \(\tensor[_{\mathcal{K}}]{U}{}(\mathfrak{b}_\pm)\) and 
  \(\tensor[_{\mathcal{O}}]{U}{}(\mathfrak{b}_\pm)\) be the Heisenberg algebra
  over \(\mathcal{K}\) and \(\mathcal{O}\) respectively, that is
  \begin{align*}
    \tensor[_{\mathcal{K}}]{U}{}(\mathfrak{b}_\pm)&=\mathcal{K}[b_{\pm1},b_{\pm2},\dots]&
    \tensor[_{\mathcal{O}}]{U}{}(\mathfrak{b}_\pm)&=\mathcal{O}[b_{\pm1},b_{\pm2},\dots]\\
    \tensor[_{\mathcal{K}}]{U}{}(\mathfrak{b}_0)&=\mathcal{K}[b_0]&
    \tensor[_{\mathcal{O}}]{U}{}(\mathfrak{b}_0)&=\mathcal{O}[b_0]\\
    \tensor[_{\mathcal{K}}]{U}{}(\overline{\mathfrak{b}})&
    =\tensor[_{\mathcal{K}}]{U}{}(\mathfrak{b}_-)\hat\otimes_{\mathcal{K}}\,
    \tensor[_{\mathcal{K}}]{U}{}(\mathfrak{b}_+)&
    \tensor[_{\mathcal{O}}]{U}{}(\overline{\mathfrak{b}})
    &=\tensor[_{\mathcal{O}}]{U}{}(\mathfrak{b}_-)\hat\otimes_{\mathcal{O}} 
    \tensor[_{\mathcal{O}}]{U}{}(\mathfrak{b}_+)\\
    &=\bigoplus_{d\in\mathbb{Z}}\tensor[_{\mathcal{K}}]{U}{}(\overline{\mathfrak{b}})[d]&
    &=\bigoplus_{d\in\mathbb{Z}}\tensor[_{\mathcal{O}}]{U}{}(\overline{\mathfrak{b}})[d]\\
    \tensor[_{\mathcal{K}}]{U}{}(\mathfrak{b})&=\tensor[_{\mathcal{K}}]{U}{}(\overline{\mathfrak{b}})\otimes_{\mathcal{K}}\tensor[_{\mathcal{K}}]{U}{}(\mathfrak{b}_0)&
    \tensor[_{\mathcal{O}}]{U}{}(\mathfrak{b})
    &=\tensor[_{\mathcal{O}}]{U}{}(\overline{\mathfrak{b}})\otimes_{\mathcal{O}}
    \tensor[_{\mathcal{O}}]{U}{}(\mathfrak{b}_0)\,.
  \end{align*}
The Heisenberg algebra over \(\mathcal{K}\)
contains the \(\mathcal{O}\) 
subalgebra \(\tensor[_{\mathcal{O}}]{U}{}(b)\) as an \(\mathcal{O}\) lattice.
\end{definition}

We deform the parameters \(\alpha_\pm\) and \(\kappa_\pm\) as follows.
Let
\begin{align*}
  \alpha_\pm(\epsilon)=\alpha_\pm^{(0)}+\alpha_\pm^{(1)}\epsilon+\alpha_\pm^{(2)}\epsilon^2+\cdots\in\mathcal{O}
\end{align*}
such that \(\alpha_\pm^{(0)}=\alpha_{\pm}\)
and that \(\alpha_\pm^{(1)}\neq0\) as well as
\(\alpha_+(\epsilon)\alpha_-(\epsilon)=-2\). Furthermore, let
\begin{align*}
  \kappa_\pm(\epsilon)&=\tfrac12\alpha_\pm(\epsilon)^2\in\mathcal{O}\,,\\
  \alpha_0(\epsilon)&=\alpha_+(\epsilon)+\alpha_-(\epsilon)\in\mathcal{O}\,.
\end{align*}

We define the rank 2 abelian group
\begin{align*}
  \tensor[_{\mathcal{O}}]{X}{}=\mathbb{Z}\tfrac{\alpha_+(\epsilon)}{2}\oplus\mathbb{Z}\tfrac{\alpha_-(\epsilon)}{2}\subset\mathcal{O}
\end{align*}
and for \(r,s\in\mathbb{Z}\)
\begin{align*}
  \beta_{r,s}(\epsilon)=\frac{1-r}{2}\alpha_+(\epsilon)+\frac{1-s}{2}\alpha_-(\epsilon)\in\tensor[_{\mathcal{O}}]{X}{}\,.
\end{align*}

\begin{definition}
  \begin{enumerate}
  \item For each \(\beta\in\tensor[_{\mathcal{O}}]{X}{}\) we define 
    the left \(\tensor[_{\mathcal{K}}]{U}{}(\mathfrak{b})\)-module
  \(\tensor[_{\mathcal{K}}]{F}{_\beta}\) generated by \(|\beta\rangle\)
  \begin{align*}
    b_0|\beta\rangle&=\beta|\beta\rangle\,,&b_n|\beta\rangle&=0,\quad n\geq 1\,,\\
  \end{align*}
  such that
  \begin{align*}
    \tensor[_{\mathcal{K}}]{U}{}(\mathfrak{b}_-)&\rightarrow \tensor[_{\mathcal{K}}]{F}{_\beta}\\
    P&\mapsto P|\beta\rangle
  \end{align*}
  is an isomorphism of \(\mathcal{K}\)-vector spaces.
\item Let \(\tensor[_{\mathcal{O}}]{F}{_\beta}\) be the subspace of 
  \(\tensor[_{\mathcal{K}}]{F}{_\beta}\) given by
  \begin{align*}
    \tensor[_{\mathcal{O}}]{F}{_\beta}=\tensor[_{\mathcal{O}}]{U}{}(\mathfrak{b})
    |\beta\rangle\,.
  \end{align*}
\end{enumerate}

\end{definition}

The Virasoro field and other fields
are defined in the same way as in Section \ref{sec:definitions}
\begin{align*}
  T(z)&=\tfrac12:b(z)^2:+\tfrac{\alpha_0(\epsilon)}{2}\partial b(z)\\
  V_\beta(z)&=:e^{\beta\phi(z)}:\,,
\end{align*}
with \(\beta\) now in \(\mathcal{O}\) instead of \(\mathbb{C}\).
Also as in Section \ref{sec:definitions}, we drop \(\beta\) from the index of
Fock spaces
\(\tensor[_{\mathcal{K}}]{F}{_{\beta_{r,s}}}=\tensor[_{\mathcal{K}}]{F}{_{r,s}}\),
\(\tensor[_{\mathcal{O}}]{F}{_{\beta_{r,s}}}=\tensor[_{\mathcal{O}}]{F}{_{r,s}}\). 
By evaluating operator product expansions it follows that the
central charge and conformal weights are given by the same
formulae as before
\begin{align*}
  c_{p_+,p_-}(\epsilon)&=1-3\alpha_0(\epsilon)^2 \in\mathcal{O}\,,&
  h_\beta(\epsilon)&=\tfrac12 \beta(\beta-\alpha_0(\epsilon))\in\mathcal{O}\,.
\end{align*}
Set \(h_{r,s}(\epsilon)=h_{\beta_{r,s}(\epsilon)},\ r,s\in\mathbb{Z}\), then
\begin{align*}
  h_{r,s}(\epsilon)=\frac{r^2-1}{4}\kappa_+(\epsilon)-\frac{rs-1}{2}
  +\frac{s^2-1}{4}\kappa_-(\epsilon)\,.
\end{align*}

\begin{prop}
  Let \(\tensor[_{\mathcal{K}}]{\mathcal{F}}{_{p_+,p_-}}
  =(\tensor[_{\mathcal{K}}]{F}{_0},|0\rangle,\tfrac12(b_{-1}^2+\alpha_0(\epsilon)b_{-2})|0\rangle,Y)\), then
  \(\tensor[_{\mathcal{K}}]{\mathcal{F}}{_{p_+,p_-}}\) has the
  structure of a VOA over the field
  \(\mathcal{K}\).
\end{prop}

\begin{prop}
  For each \(A\in\tensor[_{\mathcal{O}}]{F}{_0}\), the field \(Y(A;z)\) preserves the \(\mathcal{O}\) lattice
  \(\tensor[_{\mathcal{O}}]{F}{_0}\) of \(\tensor[_{\mathcal{K}}]{F}{_0}\), that is
  \begin{align*}
    Y(A;z)\in \operatorname{End}_{\mathcal{O}}(\tensor[_{\mathcal{O}}]{F}{_0})[[z,z^{-1}]]\,.
  \end{align*}
\end{prop}

The two Heisenberg weights \(\alpha_\pm(\epsilon)\) have conformal weight
\(h_{\alpha_{\pm}}(\epsilon)=1\). Therefore the fields
\begin{align*}
  \scr{+}(z)&=:e^{\alpha_+(\epsilon)\phi(z)}:\\
  \scr{-}(z)&=:e^{\alpha_-(\epsilon)\phi(z)}:
\end{align*}
define screening operators for \(\tensor[_{\mathcal{K}}]{\mathcal{F}}{_{p_+,p_-}}\).

\subsection{The construction of renormalisable cycles}
\label{sec:renormcycles}

In this section we construct cycles over which we can integrate products of
the screening operators \(\scr{+}(z)\) and \(\scr{-}(z)\).
In order to construct these cycles, we make extensive use of local systems as
well as de Rham theory twisted by these local systems. We refer readers
unfamiliar with these topics to Aomoto and Kita's book
\cite{AomotoKita:2011}. We give a very brief overview following Chapter 2 of
\cite{AomotoKita:2011} to fix notation.

For \(m\geq1\) let \(Y_m\) be the complex manifold
\begin{align*}
  Y_m=\{(y_1,\dots,y_m)\in\mathbb{C}^m|y_i\neq y_j,
  y_i\neq 0,1\}
\end{align*}
Let
\(\rho, \sigma,\tau \in\mathcal{O}\), then
\begin{align*}
  G_m(\rho,\sigma,\tau;y)=\prod_{i=1}^m y_i^\rho (1-y_i)^\sigma
  \prod_{1\leq i\neq j\leq m}(y_i-y_j)^\tau
\end{align*}
is a multivalued function on \(Y_m\). The logarithmic derivative of
\(G_m(\rho,\sigma,\tau;y)\)
\begin{align*}
  \omega_m(\rho,\sigma,\tau)&= \d\log G_m(\rho,\sigma,\tau,y)
  =\sum_{i=1}^m\left(\frac{\rho}{y_i}-\frac{\sigma}{1-y_i}\right)\d y_i
  +\sum_{1\leq i<j\leq m}\tau\frac{\d y_i-\d y_j}{y_i-y_j}
\end{align*}
is a single valued 1-form, which defines the twisted differential
\begin{align*}
  \nabla_{\omega_m(\rho,\sigma,\tau)}&=\d+\omega_m(\rho,\sigma,\tau)\wedge\ \,.
\end{align*}
The local systems
\(\tensor[_{\mathcal{O}}]{\mathcal{L}}{_m}(\rho,\sigma,\tau)\)
and \(\tensor[_{\mathcal{K}}]{\mathcal{L}}{_m}(\rho,\sigma,\tau)\) are defined
to be the local solutions of the differential equation
\begin{align*}
  \nabla_{\omega_m(\rho,\sigma,\tau)} f(y)=0
\end{align*}
over \(\mathcal{O}\) and \(\mathcal{K}\) respectively. Note that
\(G_m(\rho,\sigma,\tau;y)\) is such a local solution, since
\(\nabla_{\omega_m(\rho,\sigma,\tau)} G_m(\rho,\sigma,\tau;y)=0\). The local systems
\(\tensor[_{\mathcal{O}}]{\mathcal{L}}{_m}(\rho,\sigma,\tau)\)
and \(\tensor[_{\mathcal{K}}]{\mathcal{L}}{_m}(\rho,\sigma,\tau)\) are fibre
bundles with base manifold \(Y_m\)  and with fibres \(\mathcal{O}\) and \(\mathcal{K}\) respectively.
The duals of these local systems are denoted by
\(\tensor*[_{\mathcal{O}}]{\mathcal{L}}{_m^\vee}(\rho,\sigma,\tau)=
\hom(\tensor[_{\mathcal{O}}]{\mathcal{L}}{_m}(\rho,\sigma,\tau),\mathcal{O})\)
and \(\tensor*[_{\mathcal{K}}]{\mathcal{L}}{_m^\vee}(\rho,\sigma,\tau)=
\hom(\tensor[_{\mathcal{K}}]{\mathcal{L}}{_m}(\rho,\sigma,\tau),\mathcal{K})\).
The twisted homology groups with coefficients in \(\tensor*[_{\mathcal{O}}]{\mathcal{L}}{_m^\vee}(\rho,\sigma,\tau)\)
and \(\tensor*[_{\mathcal{K}}]{\mathcal{L}}{_m^\vee}(\rho,\sigma,\tau)\) are
denoted by
\(H_p(Y_m,\tensor*[_{\mathcal{O}}]{\mathcal{L}}{_m^\vee}(\rho,\sigma,\tau))\)
and
\(H_p(Y_m,\tensor*[_{\mathcal{K}}]{\mathcal{L}}{_m^\vee}(\rho,\sigma,\tau))\)
and the twisted cohomology groups by
\begin{align*}
  H^p(Y_m,\tensor*[_{\mathcal{O}}]{\mathcal{L}}{_m}(\rho,\sigma,\tau))&=
  \hom_{\mathcal{O}}(H_p(Y_m,\tensor*[_{\mathcal{O}}]{\mathcal{L}}{_m^\vee}(\rho,\sigma,\tau)),\mathcal{O})\,,\\
  H^p(Y_m,\tensor*[_{\mathcal{K}}]{\mathcal{L}}{_m}(\rho,\sigma,\tau))&=
  \hom_{\mathcal{K}}(H_p(Y_m,\tensor*[_{\mathcal{K}}]{\mathcal{L}}{_m^\vee}(\rho,\sigma,\tau)),\mathcal{K})\,.
\end{align*}
The boundary maps required for defining the twisited homology groups are
constructed in the following way. Let \(\Delta_p\) be a \(p\)-simplex in some
smooth triangulation \(K\) of \(Y_m\). We fix a branch \(\phi\) of
\(G_m(\rho,\sigma,\tau;y)\) on \(\Delta_p\) and denote the pair by the symbol
\(\Delta_p\otimes\phi\). The boundary \(\partial(\Delta_p\otimes\phi)\) of
\(\Delta_p\otimes\phi\) is given by the boundary of \(\Delta_p\) with the
branch fixed by restricting \(\phi\) to the boundary of \(\Delta_p\). See
Figure \ref{fig:cycle} for an explicit example of a closed cylce.  For a
\(p\)-cycle \(\Gamma_p\) and a single valued \(p-1\)-form \(\psi_{p-1}(y)\),
the twisted version of Stokes theorem is given by
\begin{align*}
   \int_{\partial \Gamma_p} G_m(\rho,\sigma,\tau;y)\psi_{p-1}(y)=
   \int_{\Gamma_p} \d G_m(\rho,\sigma,\tau;y)\psi_{p-1}(y)=
   \int_{\Gamma_p} G_m(\rho,\sigma,\tau;y)\nabla_{\omega_m(\rho,\sigma,\tau)}\psi_{p-1}(y)\,.
\end{align*}
The differential \(\nabla_{\omega_m(\rho,\sigma,\tau)}\) defines a twisted
de Rham theory.

The theorem of twisted de Rham theory states that the twisted cohomology groups are
isomorphic to twisted de Rham cohomology groups
\begin{align*}
  H^p(Y_m,\tensor*[_{\mathcal{O}}]{\mathcal{L}}{_m}(\rho,\sigma,\tau))&\cong
  \tensor[_{\mathcal{O}}]{H}{^p}(Y_m,\nabla_{\omega_m(\rho,\sigma,\tau)})\,,&
  H^p(Y_m,\tensor*[_{\mathcal{K}}]{\mathcal{L}}{_m}(\rho,\sigma,\tau))&\cong
  \tensor[_{\mathcal{K}}]{H}{^p}(Y_m,\nabla_{\omega_m(\rho,\sigma,\tau)})\,.
\end{align*}
The pairing between the twisted homology groups and twisted de Rham cohomology
groups is given by integration as with ordinary de Rham theory.
It is known that \cite[Chapter 2]{AomotoKita:2011}
\begin{align*}
  \dim
  H^p(Y_m,\tensor*[_{\mathcal{O}}]{\mathcal{L}}{_m}(\rho,\sigma,\tau))=0,\ p>m\,.
\end{align*}
The theory of local systems elegantly sidesteps the potentially problematic multivaluedness of
\(G_m(\rho,\sigma,\tau;y)\) (or the multivaluedness of \(z^{\beta b_0}\)
mentioned in Section \ref{sec:heisenbergalg}), by expressing everything in terms of a de Rham
theory with twisted differential \(\nabla_{\omega_m(\rho,\sigma,\tau)}\).

The symmetric group \(\symg{m}\) acts in a compatible fashion on both \(Y_m\) and
\(\tensor[_{\mathcal{K}}]{\mathcal{L}}{_m}(\rho,\sigma,\tau)\), therefore the cohomology group
\(H^p(X_N,\tensor[_{\mathcal{K}}]{\mathcal{L}}{_m}(\rho,\sigma,\tau))\) carries the structure of a 
finite dimensional representation of \(\symg{m}\). We can therefore decompose
\(H^p(X_N,\tensor[_{\mathcal{K}}]{\mathcal{L}}{_m}(\rho,\sigma,\tau))\) into a direct sum of
irreducible \(\symg{m}\) modules.
For any \(\symg{m}\) module \(M\), let \(M^{\symg{m}-}\) be the skew symmetric part of \(M\).
For the purposes of this paper we are only interested in the skew symmetric parts of the \(m\)th cohomology
groups \(H^m(Y_m,\tensor[_{\mathcal{K}}]{\mathcal{L}}{_m}(\rho,\sigma,\tau))^{\symg{m}-}\).
\begin{prop}\label{sec:goodcycles}
Let
\(\rho=\rho_0+\rho_1\epsilon+\cdots,\sigma=\sigma_0+\sigma_1\epsilon+\cdots,
\tau=\tau_0+\tau_1\epsilon+\cdots\in\mathcal{O}\) such that the constant terms
of \(\sigma,\tau\) lie in \(\mathbb{C}\setminus\mathbb{Q}_{\leq0}\) and
\begin{align*}
  d(d+1)\tau \notin\mathbb{Z}\,,\ d(d-1)\tau + d\rho\notin\mathbb{Z}\,,\ 
  d(d-1)\tau+d \sigma\notin\mathbb{Z}\,,\ 1\leq d\leq m\,,
\end{align*}
then
  there exists a construction of a closed cycle \(\Delta_m(\rho,\sigma,\tau)\)
  with non-trivial homology class
    \([\Delta_m(\rho,\sigma,\tau)]\in
    H_m(Y_m,\tensor*[_{\mathcal{K}}]{\mathcal{L}}{_m^\vee}(\rho,\sigma,\tau))\)
    such that
  \begin{enumerate}
  \item For \(f(y)\in \mathcal{K}[y_1^\pm,\dots,y_m^\pm]\)
    \begin{align*}
      \int_{[\Delta_m(\rho,\sigma,\tau)]}G_m(\rho,\sigma,\tau;y)f(y)\tfrac{\d
        y_1\cdots\d y_m}{y_1\cdots y_m}
      =
      \int_{\Delta_m}G_m(\rho,\sigma,\tau;y)f(y)\tfrac{\d
        y_1\cdots\d y_m}{y_1\cdots y_m}\,,
    \end{align*}
    where the right hand side is an indefinite integral over the \(m\)-simplex \(\Delta_m=\{1> y_1>\cdots>y_m>0\}\).
  \item The integration of \(G_m(\rho,\sigma,\tau)\) over \(\Delta_m\) is
    given by the Selberg integral
    \begin{align*}
      S_m(\rho,\sigma,\tau)&=\int_{\Delta_m}G_m(\rho,\sigma,\tau;y) \tfrac{\d
        y_1\cdots\d y_m}{y_1\cdots y_m}\\
      &=\frac{1}{m!}\prod_{i=1}^m
      \frac{\Gamma(1+i\tau)\Gamma(\rho+(i-1)\tau)\Gamma(1+\sigma+(i-1)\tau)}
      {\Gamma(1+\tau)\Gamma(1+\rho+\sigma+(m+i-2)\tau)}
    \end{align*}
  \end{enumerate}
\end{prop}
\begin{proof}\ 
  \begin{enumerate}
  \item The class of cylces \([\Delta_m(\rho,\sigma,\tau)]\) was explicitly constructed in
    \cite[Sections 5]{Tsuchiya:1986}
    by means of the isomorphism of homology groups
    \(H_m(Y_m,\tensor*[_{\mathcal{K}}]{\mathcal{L}}{_m^\vee}(\rho,\sigma,\tau))
    \cong
    H_m^{\operatorname{l.f}}(Y_m,\tensor*[_{\mathcal{K}}]{\mathcal{L}}{_m^\vee}(\rho,\sigma,\tau))\)
    between the twisted homology group and the locally finite homology
    group. If \(\phi\) is the principal branch of \(G_m(\rho,\sigma,\tau;y)\)
    on \(\Delta_m\), then 
    \([\Delta_m\otimes\phi]\in
    H_m^{\operatorname{l.f}}(Y_m,\tensor*[_{\mathcal{K}}]{\mathcal{L}}{_m^\vee}(\rho,\sigma,\tau))\).
    The corresponding class in 
    \(H_m(Y_m,\tensor*[_{\mathcal{K}}]{\mathcal{L}}{_m^\vee}(\rho,\sigma,\tau))\)
    was then constructed by means of a blow up \(\hat{Y}_m\rightarrow Y_m\).
  \item This integral is due to Selberg \cite{Selberg:1944}. For an overview of the many contributions of the Selberg integral to
    mathematics see \cite{Forrester:2008}.
  \end{enumerate}
\end{proof}

For \(m=1\) the exponent \(\tau\) does not appear in the multi-valued function \(G_1(\rho,\sigma,\tau;y)\)
and the cycle \(\Delta_1(\rho,\sigma,\tau)\) is also known as the
regularisation of the open interval \((0,1)\). See Figure \ref{fig:cycle} for
an explicit depiction of \(\Delta_1(\rho,\sigma,\tau)\).
\begin{figure}[tp]
  \centering 
  \begin{tikzpicture}
    \draw[->,>=angle 60] (0,-1.2) -- (0,2.8);
    \draw[->,>=angle 60] (-1.2,0) -- (8,0);
    \draw[-] (6,-0.1) node[below] {1} -- (6,0.1);
    \draw (0,0) node[anchor=north west] {0};
    \draw[->,>=angle 60,very thick] (1,0) arc (0:357:1);
    \draw[->,>=angle 60,very thick] (5,0) arc (180:537:1);
    \draw[->,>=angle 60,very thick] (1,0) -- (3,0) node[below] {\([\delta,1-\delta]\)};
    \draw[-,very thick] (2.8,0) -- (5,0);
    \draw (1,0) node[anchor=south west] {\(\delta\)};
    \draw (5,0) node[anchor=south east] {\(1-\delta\)};
    \draw (0,1) node[anchor=south west] {\(S^1_\delta(0)\)};
    \draw (6,1) node[anchor=south west] {\(S^1_\delta(1)\)};
    \draw (3,2) node {\(G_1(\rho,\sigma,\tau;y)=y^\rho(1-y)^\sigma\)};
  \end{tikzpicture}
  \caption{Regularisation of the open interval \((0,1)\): The closed interval
    from \(\delta\) to \(1-\delta\) for some small \(\delta>0\) is denoted by
    \([\delta,1-\delta]\), the counter-clockwise circle around 0 of radius
    \(\delta\) starting at \(\delta\) by \(S^1_\delta(0)\) and the counter-clockwise circle around 1 of radius
    \(\delta\) starting at \(1-\delta\) by \(S^1_\delta(1)\). The boundary of
    the interval \([\delta,1-\delta]\) is given by its end points but with opposite
    orientations \(\partial[\delta,1-\delta]=\langle 1-\delta\rangle-\langle\delta\rangle\). The start and
    end points of the two circles \(S^1_\delta(0)\) and \(S^1_\delta(1)\) are
    the same, however, since they are on
    different branches the circles are not closed.
    The boundaries are given by \(\partial S^1_\delta(0)=e^{2\pi i\rho}\langle\delta\rangle-\langle
    \delta\rangle=(e^{2\pi i\rho}-1)\langle\delta\rangle\), \(\partial
    S^1_\delta(1)=(e^{2\pi i\sigma}-1)\langle 1-\delta\rangle\).
    The
    regularisation of the open interval \((0,1)\)
    is given by \(\Delta_1(\rho,\sigma,\tau)=\frac{1}{e^{2\pi
          i\rho}-1}S^1_\delta(0)+[\delta,1-\delta]-\frac{1}{e^{2\pi
          i\sigma}-1}S^1_\delta(1)\), which is closed \(\partial \Delta_1(\rho,\sigma,\tau)=0\).}
  \label{fig:cycle}
\end{figure}
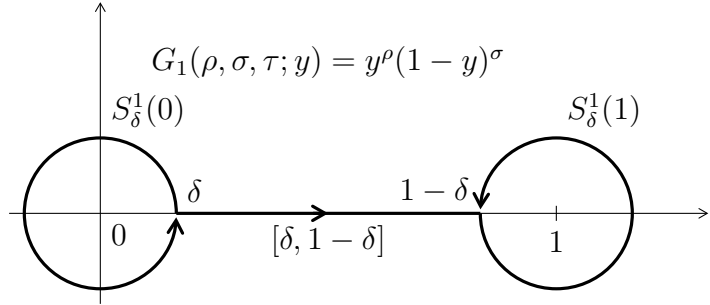
The integral of \(G_1(\rho,\sigma,\tau;y)\) over the class
\([\Delta_1(\rho,\sigma,\tau)]\) is the famous Euler beta function.
\begin{align*}
  B(\rho,\sigma+1)=\int_{[\Delta_1(\rho,\sigma,\tau)]}y^\rho(1-y)^\sigma\frac{\d y}{y}
  =\int_{0}^1y^\rho(1-y)^\sigma\frac{\d y}{y}
  =\frac{\Gamma(\rho)\Gamma(\sigma+1)}{\Gamma(\rho+\sigma+1)}
\end{align*}
For \(m\geq2\) the cycles \([\Delta_m(\rho,\sigma,\tau)]\) are more difficult
to visualise, because the visualisations would have to be drawn in
\(\mathbb{C}^m\). The \([\Delta_m(\rho,\sigma,\tau)]\) are essentially just
higher dimensional analogues of \([\Delta_1(\rho,\sigma,\tau)]\), that is, one
starts with an \(m\)-simplex \(\Delta_m=\{1>y_1>\cdots >y_n>0\}\) and cuts out
the \(k\)-faces of \(\Delta_m\) on which \(G_m(\rho,\sigma,\tau;y)\) is
singular, where \(m>k\geq0\). The \(k\)-faces that were cut away are then
replaced by the boundaries of tubular neighbourhoods of said \(k\)-faces. This
procedure is carried out explicitly in \cite[Section 5]{Tsuchiya:1986}.

We will now use these classes of cycles \([\Delta_m(\rho,\sigma,\tau)]\) in order to
integrate \(N\)-fold products of screening operators.
Consider the \(N\)-fold product of \(\scr{\pm}(z)\)
\begin{align*}
  \prod_{i=1}^N \scr{\pm}(z_i)=e^{N\alpha_{\pm}(\epsilon)\hat b}\prod_{i=1}^N z_i^{\alpha_\pm(\epsilon)b_0}
  \prod_{1\leq i\neq j\leq N}(z_i-z_j)^{\alpha_\pm(\epsilon)^2/2}
  :\prod_{i=1}^N\overline{\scr{\pm}}(z_i):\,,
\end{align*}
where \(:\prod_{i=1}^N\overline{\scr{\pm}}(z_i):\in \tensor[_{\mathcal{O}}]{U}{}(\overline{\mathfrak{b}})\hat\otimes_{\mathcal{O}}\mathcal{O}[z_1^\pm,\dots,z_n^\pm]^{\symg{n}}\)
\begin{align*}
  :\prod_{i=1}^N\overline{\scr{\pm}}(z_i):=\prod_{k\geq 1}e^{\alpha_\pm(\epsilon)\sum_{i=1}^N \tfrac{z_i^{k}}{k}b_{-k}}
  \prod_{k\geq 1}e^{-\alpha_\pm(\epsilon)\sum_{i=1}^N \tfrac{z_i^{-k}}{k}b_{k}}\,.
\end{align*}
The superscript \(\symg{n}\) of \(\mathcal{O}[z_1^\pm,\dots,z_n^\pm]^{\symg{n}}\) indicates that \(:\prod_{i=1}^N\overline{\scr{\pm}}(z_i):\) is symmetric with
respect to permuting the variables \(z_i\).

Let \(r\geq 1,s\in\mathbb{Z}\), then if we evaluate the operator \(\scr{+}(z_1)\cdots \scr{+}(z_r)\)
on \(\tensor[_{\mathcal{K}}]{F}{_{r,s}}\) we have
\begin{align*}
  \prod_{i=1}^r\scr{+}(z_i)=e^{r\alpha_+(\epsilon)\hat b}U_r(z_1,\dots,z_r;\kappa_+(\epsilon))\prod_{i=1}^r z_i^{s-1}
  :\prod_{i=1}^r\overline{\scr{+}}(z_i):\,.
\end{align*}
where
\begin{align*}
  U_r(z_1,\dots,z_r;\kappa_+(\epsilon))&=
  \prod_{1\leq i\neq j\leq r}(z_{i}-z_{j})^{\kappa_+(\epsilon)}
  \prod_{i=1}^r z_i^{(1-r)\kappa_+(\epsilon)}\\
  &=\prod_{1\leq i\neq j\leq r}(1-\tfrac{z_i}{z_j})^{\kappa_+(\epsilon)}\,.
\end{align*}
Similarly one can also evaluate \(\scr{-}(z_1)\cdots \scr{-}(z_s)\)
on \(\tensor[_{\mathcal{K}}]{F}{_{r,s}}\) for \(s\geq 1\), \(r\in\mathbb{Z}\).

We can use the twisted de Rham theory, developed above, to integrate the
multivalued function
\(U_N(z_1,\dots,z_N;\kappa)\).
Consider the \(N\) dimensional
complex manifold
\begin{align*}
  X_N=\{(z_1,\dots,z_N)\in \mathbb{C}^N| z_i\neq z_j, z_i\neq 0 \}\,.
\end{align*}
and fix
\begin{align*}
  \kappa=\kappa_{0}+\kappa_{1}\epsilon+\cdots\in \mathcal{O}
\end{align*}
such that \(\kappa_{0}\in\mathbb{C}\setminus\mathbb{Q}_{\leq0}\)
and \(\kappa_{1}\neq 0\).
Then
\begin{align*}
  U_N(z;\kappa)=\prod_{1\leq i\neq j\leq N}(1-\tfrac{z_i}{z_j})^{\kappa}
\end{align*}
defines a multivalued holomorphic function on \(X_N\). Denote by
\(\tensor[_{\mathcal{K}}]{\mathcal{L}}{_N}(\kappa)\) the local system defined by the multivaluedness of
\(U_N(z;\kappa)\) over \(\mathcal{K}\) and its dual by \(\tensor*[_{\mathcal{K}}]{\mathcal{L}}{_N^\vee}(\kappa)\). We introduce new variables
\(z,y_1,\dots,y_{N-1}\) such that
\begin{align*}
  z_1&=z\,,\ z_i=zy_{i-1}\,, i=2,\dots,N\,.
\end{align*}
Then it is clear that \(X_N\cong\mathbb{C}^\ast \times Y_{N-1}\). If we
write \(U_N(z;\kappa)\) in terms of the new variables, we see that the \(z\)
dependence drops out
\begin{align*}
  U_N(z,zy_1,\dots;\kappa)=\prod_{i=1}^{N-1}y_i^{(1-N)\kappa}(1-y_i)^{2\kappa}\prod_{1\leq
    i\neq j\leq N-1}(y_i-y_j)^\kappa=G_{N-1}((1-N)\kappa,2\kappa,\kappa)\,.
\end{align*}
By the K\"unneth formula we therefore have
\begin{align*}
  H_N(X_N,\tensor*[_{\mathcal{K}}]{\mathcal{L}}{_N^\vee}(\kappa))=
  H_1(\mathbb{C}^\ast,\mathcal{K})\otimes
  H_{N-1}(Y_{N-1},\tensor*[_{\mathcal{K}}]{\mathcal{L}}{_m^\vee}((1-N)\kappa,2\kappa,\kappa))\,.
\end{align*}
It is known that \cite{AomotoKita:2011}
\begin{align*}
  \dim_{\mathcal{K}}
  H^N(X_N,\tensor*[_{\mathcal{K}}]{\mathcal{L}}{_N}(\kappa))
  &=(N-1)!\,,\\
  \dim_{\mathcal{K}}
  H^N(X_N,\tensor*[_{\mathcal{K}}]{\mathcal{L}}{_N}(\kappa))^{\symg{N}-}
  &=1\,.
\end{align*}
Let \(2\pi i[\gamma]\in H_1(\mathbb{C}^\ast,\mathcal{K})\) be the class of a
circle about the origin, then we can use Proposition \ref{sec:goodcycles}.

\begin{definition}\label{sec:cycledef}
  Let \([\Gamma_N(\kappa)]\) and \([\overline{\Gamma}_N(\kappa)]\) be the renormalised cycles
  \begin{align*}
    [\overline{\Gamma}_{N}(\kappa)]&=\frac{1}{S_{N-1}((1-N)\kappa,2\kappa,\kappa)}[\Delta_{N-1}((1-N)\kappa,2\kappa,\kappa)]\,,&
    [\Gamma_{N}(\kappa)]&=[\gamma]\times [\overline{\Gamma}_N(\kappa)]
  \end{align*}
  such that
  \begin{align*}
    \int_{[\Gamma_N(\kappa)]}U_N(z;\kappa)\prod_{i=1}^N\frac{\d
      z_i}{z_i}&=1\,,&
    \int_{[\overline{\Gamma}_N(\kappa)]}U_N(z,zy_1,\dots,zy_{N-1};\kappa)\prod_{i=1}^{N-1}\frac{\d
      y_i}{y_i}&=1\,.
  \end{align*}
\end{definition}

\begin{definition}\label{sec:symringmap}
  For \(f\in\mathcal{K}[z_1^{\pm},\dots,z_N^{\pm}]^{\symg{N}}\)
  the cycle \([\Gamma_N(\kappa)]\) defines 
  a \(\mathcal{K}\)-linear map
  \begin{align*}
    \langle \phantom{f}
    \rangle_\kappa^N:\mathcal{K}[z_1^{\pm},\dots,z_N^{\pm}]^{\symg{N}}\rightarrow
    \mathcal{K}
  \end{align*}
  by the formula
  \begin{align*}
    \langle f(z)
    \rangle_\kappa^N=\int_{[\Gamma_N(\kappa)]}U_N(z;\kappa)f(z)\prod_{i=1}^N\frac{\d
      z_i}{z_i}\,.
  \end{align*}
\end{definition}

\begin{prop}\ 
  \begin{enumerate}
  \item \(\langle 1\rangle_\kappa^N=1\)
  \item \(\langle f\rangle_\kappa^N=0\) if \(\deg f\neq 0\)
  \item \(\langle f\rangle_\kappa^N=\langle \overline{f}\rangle_\kappa^N\)
    where \(\overline{f}(z_1,\dots,z_N)=f(z_1^{-1},\dots,z_N^{-1})\).
  \end{enumerate}
\end{prop}

\begin{definition}\label{sec:innerproddef}
  Let \((\phantom{f},\phantom{g})_\kappa^N\) be the bilinear 
  \(\mathcal{K}\)-form
  \begin{align*}
    (\phantom{f},\phantom{g})_\kappa^N:\mathcal{K}[z_1,\dots,z_N]^{\symg{N}}\otimes_{\mathcal{K}}
    \mathcal{K}[z_1,\dots,z_N]^{\symg{N}}\rightarrow \mathcal{K}
  \end{align*}
  defined by
  \begin{align*}
    ( f,g)_\kappa^N=\langle
    \overline{f}g\rangle_{\kappa^{-1}}^N\,.
  \end{align*}
\end{definition}
Note that the inner product above is defined in terms of \(\kappa^{-1}\)
in order to
be closer to the notation of Macdonald's book \cite{MacDonald:1999}.
Following Macdonald's book \cite{MacDonald:1999}, we will evaluate this
bilinear form in the next section by using the theory of Jack polynomials.
We will also be able to show that it defines an inner product of symmetric
polynomials over \(\mathcal{O}\), that is
  \begin{align*}
    (\ ,\
    )_\kappa^N:\mathcal{O}[z_1,\dots,z_N]^{\symg{N}}\otimes_{\mathcal{O}}
    \mathcal{O}[z_1,\dots,z_N]^{\symg{N}}\rightarrow \mathcal{O}\,.
  \end{align*}

\subsection{The theory of Jack polynomials}
\label{sec:thyofjackpoly}

We introduce the theory of Jack polynomials following 
Macdonald's book \cite[Chapter 6]{MacDonald:1999}. 
Fix the parameter \(\kappa\) to be
\begin{align*}
  \kappa=\kappa_{0}+\kappa_{1}\epsilon+\dots\in\mathcal{O}\,,
\end{align*}
such that \(\kappa_{0}\in\mathbb{C}\setminus\mathbb{Q}_{\leq0}\) and 
\(\kappa_{1}\neq 0\).

\begin{definition}
  \emph{The rings of symmetric polynomials} over \(\mathcal{O}\) and
  \(\mathcal{K}\) in \(N\) variables \(x_i\) are given by
  \begin{align*}
    \tensor[_{\mathcal{K}}]{\Lambda}{_N}&=\mathcal{K}[x_1,\dots,x_N]^{\symg{N}}&
    \tensor[_{\mathcal{O}}]{\Lambda}{_N}&=\mathcal{O}[x_1,\dots,x_N]^{\symg{N}}\\
    &=\mathcal{K}[p_1,\dots,p_N]&&=\mathcal{O}[p_1,\dots,p_N]
  \end{align*}
  where the
  \begin{align*}
    p_n=\sum_{i= 1}^Nx_i^n\,,
  \end{align*}
  are called \emph{power sums}.  The ring of symmetric polynomials
  \(\tensor[_{\mathcal{K}}]{\Lambda}{_N}\) forms a graded commutative
  algebra with \(\deg p_n=n\). 
  Rings of symmetric polynomials in different numbers of variables are related by the homomorphisms
  \begin{align*}
    \rho_{N,M}:\tensor[_{\mathcal{K}}]{\Lambda}{_N}&\rightarrow \tensor[_{\mathcal{K}}]{\Lambda}{_M}\\
    x_i&\mapsto x_i\quad i\leq M\\
    x_i&\mapsto 0\quad i>M\,,
  \end{align*}
  where \(N>M\).  The ring of symmetric polynomials in a countably
  infinite number variables is given by the projective limit
  \begin{align*}
    \tensor[_{\mathcal{K}}]{\Lambda}{}=\varprojlim_{N}\tensor[_{\mathcal{K}}]{\Lambda}{_N}\,,
  \end{align*}
  relative to the homomorphisms \(\rho_{N,M}\). One can return to the finite
  variable case by the projection
  \begin{align*}
    \rho_N:\tensor[_{\mathcal{K}}]{\Lambda}{}&\rightarrow
    \tensor[_{\mathcal{K}}]{\Lambda}{_N}\\
    x_i&\mapsto x_i\quad i\leq N\\
    x_i&\mapsto 0\quad i>N\,.
  \end{align*}
\end{definition}

A convenient way of parametrising symmetric polynomials is by partitions of
integers.
\begin{definition}
  A \emph{partition}
  \(\lambda=(\lambda_1,\lambda_2,\dots)\) is a weakly descending
  sequence of non-negative integers. We refer to
  \(|\lambda|=\sum_i \lambda_i\)
  as the \emph{degree} of \(\lambda\) and to
  \(\ell(\lambda)=\#\{\lambda_i\neq 0\}\)
  as the \emph{length} of \(\lambda\).

To each partition \(\lambda\) we associate a Young diagram, that is a
collection of left aligned rows of boxes where the \(i\)th
row consists of \(\lambda_i\) boxes. 
The boxes of a diagram are labelled by two integers \((i,j)\), where \(i\) labels the row and \(j\) the column.
For every partition \(\lambda\) there is also conjugate partition \(\lambda^\prime\) which is obtained by exchanging rows
and columns in the Young diagram. For example the conjugate of the partition \((4,2)\) is \((2,2,1,1)\).
For a box \(s=(i,j)\) in a Young diagram let
\begin{align*}
  a_\lambda(s)&=\lambda_i-j& a^\prime_\lambda(s)&=j-1\\
  \ell_\lambda(s)&=\lambda^\prime_j-i&\ell^\prime_\lambda(s)&=i-1\,.
\end{align*}
Partitions admit a partial ordering \(\geq\) called the dominance ordering. For two partitions \(\lambda,\mu\) of equal degree,
\(\lambda\geq \mu\) if and only if
\begin{align*}
  \sum_{i=1}^n \lambda_i\geq \sum_{i=1}^n \mu_i\,,\quad \text{for any } n\geq 1\,.
\end{align*}
\end{definition}

Two examples of bases of \(\tensor[_{\mathcal{K}}]{\Lambda}{}\) parametrised by partitions are the power sums
\begin{align*}
  p_\lambda(x)&=p_{\lambda_1}(x)p_{\lambda_2}(x)\cdots
\end{align*}
and the symmetric monomials
\begin{align*}
  m_\lambda(x)&=\sum_{\sigma}\prod_{i\geq 1}x_{\sigma(i)}^{\lambda_i}\,,
\end{align*}
where the first sum runs over all distinct permutations of the entries of the partition \(\lambda\).
Note that in the case of symmetric polynomials in
\(N\) variables one must restrict oneself to
partitions of length at most \(N\), since
\begin{align*}
  \rho_N(m_\lambda(x))=0\,,\quad \text{for } \ell(\lambda)>N\,.
\end{align*}

\begin{definition}
  Let \((\phantom{f},\phantom{g})_\kappa\) be the inner product
  of symmetric polynomials defined by
  \begin{align*}
    (\phantom{f},\phantom{g})_\kappa: \tensor[_{\mathcal{K}}]{\Lambda}{}\otimes_{\mathcal{K}}\tensor[_{\mathcal{K}}]{\Lambda}{}&\rightarrow \mathcal{K}\\
    ( p_\lambda(x),p_\mu(x))_\kappa&\mapsto \delta_{\lambda,\mu}z_\lambda
    \kappa^{\ell(\lambda)}\,,
  \end{align*}
  where
  \begin{align*}
    z_\lambda=\prod_{i\geq 1}i^{m_i(\lambda)} m_i(\lambda)!
  \end{align*}
  and \(m_i(\lambda)\) is the multiplicity of \(i\) in \(\lambda\).
\end{definition}

\begin{prop}\label{sec:uppertriang}
  For \(\kappa=\kappa_{0}+\kappa_{1}\epsilon+\cdots\in\mathcal{O}\) such that \(\kappa_{0}\in\mathcal{O}\setminus\mathbb{Q}_{\leq0}\)
  and \(\kappa_{1}\neq 0\),
  the symmetric polynomials \(\tensor[_{\mathcal{K}}]{\Lambda}{}\) admit a basis of polynomials called
  Jack polynomials \(P_\lambda(x;\kappa)\) that satisfy
  \begin{enumerate}
  \item \(( P_\lambda(x;\kappa),P_\mu(x;\kappa))_\kappa=0\) if \(\lambda\neq\mu\)
  \item
    \(P_\lambda(x;\kappa)=\sum_{\lambda\geq\mu}u_{\lambda,\mu}(\kappa)
    m_\mu(x)\), where \(u_{\lambda,\lambda}(\kappa)=1\),
    \(u_{\lambda,\mu}(\kappa)\in\mathcal{O}\) and the partial
    ordering \(\geq\) is the dominance ordering.
  \end{enumerate}
\end{prop}

\begin{definition}
  Let
  \begin{align*}
    b_\lambda(\kappa)=(( P_\lambda(x;\kappa),P_\lambda(x;\kappa))_\kappa)^{-1}\,,
  \end{align*}
  then the polynomials
  \begin{align*}
    Q_\lambda(x;\kappa)=b_\lambda(\kappa) P_\lambda(x;\kappa)\,,
  \end{align*}
  form a basis dual to \(P_\lambda(x;\kappa)\),
  such that,
  \begin{align*}
    ( P_\lambda(x;\kappa),Q_\mu(x;\kappa))_\kappa=\delta_{\lambda,\mu}\,.
  \end{align*}
\end{definition}

Jack polynomials in an infinite number of variables satisfy
a number of remarkable properties which we summarise in the
following proposition.
\begin{prop}\label{sec:polytofieldmap}\ 
  \begin{enumerate}
  \item For \(\kappa=\kappa_{0}+\kappa_{1}\epsilon+\cdots\in\mathcal{O}\)
    such that 
    \(\kappa_{0}\in\mathbb{C}\setminus\mathbb{Q}_{\leq0}\) and 
    \(\kappa_{1}\neq 0\),
    the coefficients \(b_\lambda(\kappa)\) are given by
    \begin{align*}
      b_\lambda(\kappa)=((
      P_\lambda(x;\kappa),P_\lambda(x;\kappa))_\kappa)^{-1}=
      \prod_{s\in\lambda}\frac{\kappa a_\lambda(s)+\ell_\lambda(s)+1}{\kappa a_\lambda(s)+\ell_\lambda(s)+\kappa}
    \end{align*}
    and are units of \(\mathcal{O}\), that is, 
    the coefficient of \(\epsilon^0\) is non-zero.
  \item
    \begin{align*}
      \prod_{i,j\geq1}(1-x_i
      y_j)^{-1/\kappa}=\prod_{k\geq1}e^{\tfrac{1}{\kappa}\tfrac{p_k(x)p_k(y)}{k}}
      =\sum_{\lambda}P_\lambda(x;\kappa)Q_\lambda(y;\kappa)
    \end{align*}
  \item Let \(\omega_\beta,\beta\in\mathcal{O}\) be the \(\mathcal{K}\) algebra endomorphism of 
    \(\tensor[_{\mathcal{K}}]{\Lambda}{}\) given by
    \begin{align*}
      \omega_\beta(p_r)=(-1)^{r-1}\beta p_r,
    \end{align*}
    for \(r\geq1\), then the Jack polynomials satisfy
    \begin{align*}
      \omega_\kappa(P_\lambda(x;\kappa))=Q_{\lambda^\prime}(x;\kappa^{-1})\,.
    \end{align*}
  \item Let \(\Xi_X\) be the \(\mathcal{O}\) algebra homomorphism defined by
    \begin{align*}
      \Xi_X: \tensor[_{\mathcal{K}}]{\Lambda}{}&\rightarrow \mathcal{K}\\
      p_r&\mapsto X
    \end{align*}
    for \(r\geq1\) and any \(X\in\mathcal{O}\), then the Jack polynomials satisfy
      \begin{align*}
    \Xi_X(Q_\lambda(x;\kappa))=
    \prod_{s\in\lambda}\frac{X+\kappa a^\prime_\lambda(s)-\ell^\prime_\lambda(s)}{\kappa a_\lambda(s)+\ell_\lambda(s)+\kappa}\,.
  \end{align*}
  The map \(\Xi_X\) allows us to decompose certain products in terms of Jack polynomials.
  \begin{align*}
    \Xi_X(\prod_{i,j\geq1}(1-x_i
      y_j)^{-1/\kappa})&=\prod_{k\geq1}e^{\tfrac{X}{\kappa}\tfrac{p_k(y)}{k}}=\prod_{i\geq1}(1-y_i)^{-\tfrac{X}{\kappa}}\\
      &=\sum_{\lambda}P_\lambda(y;\kappa)\Xi_X(Q_\lambda(x;\kappa))
  \end{align*}
  \end{enumerate}
\end{prop}

We return to symmetric polynomials in a finite number of
variables. 
\begin{prop}\label{sec:finiteJackprop}
  \begin{enumerate}
  \item A partition \(\lambda\) defines a non-zero Jack polynomial \(P_{\lambda}(x;\kappa)\in\tensor[_{\mathcal{K}}]{\Lambda}{_N}\)
    if and only if \(\ell(\lambda)\leq N\).
  \item For the partition \(\lambda=(M,\dots,M)\), that is a partition
    consisting of \(N\) copies of an integer \(M\), the Jack
    polynomial
    \(P_{\lambda}(x;\kappa)\in\tensor[_{\mathcal{K}}]{\Lambda}{_N}\)
    is given by
    \begin{align*}
      P_{\lambda}(x;\kappa)=m_\lambda(x)=\prod_{i=1}^N x_i^M\,.
    \end{align*}
  \item For a partition \(\lambda=(\lambda_1,\dots, \lambda_N)\) and
    \(M\geq0\), let
    \(\lambda+M=(\lambda_1+M,\dots,\lambda_N+M)\), then the Jack polynomial 
    \(P_{\lambda}(x;\kappa)\in\tensor[_{\mathcal{K}}]{\Lambda}{_N}\) satisfies
    \begin{align*}
      P_{\lambda}(x;\kappa)\cdot \prod_{i=1}^N x_i^M=P_{\lambda+M}(x;\kappa)\,.
    \end{align*}
  \end{enumerate}
\end{prop}

\begin{definition}
  For the symmetric polynomials in a finite number of variables, let \((\ ,\
    )_\kappa^N\) be the inner product
  \begin{align*}
    (\ ,\
    )_\kappa^N:\tensor[_{\mathcal{K}}]{\Lambda}{_N}\otimes_{\mathcal{K}}\tensor[_{\mathcal{K}}]{\Lambda}{_N}
    \rightarrow \mathcal{K}
  \end{align*}
  such that
  \begin{align*}
    (
    f,g)_\kappa^N=\int_{[\Gamma_N(\kappa)]}U_N(x;\kappa^{-1})\overline{f}g\prod_{i=1}^N\frac{\d
      x_i}{x_i}\,,
  \end{align*}
  where \(\overline{f}=f(x_1^{-1},\dots,x_N^{-1})\).
\end{definition}

\begin{prop}\label{sec:finitevarinnerprod}
  For \(\kappa=\kappa_{0}+\kappa_{1}\epsilon+\cdots\in\mathcal{O}\)
  such that 
  \(\kappa_{0}\in\mathbb{C}\setminus\mathbb{Q}_{\leq0}\) and 
  \(\kappa_{1}\neq 0\),
  the inner product \((\ ,\ )_\kappa\) and  the inner product \((\ ,\
  )_\kappa^N\) of Definition \ref{sec:innerproddef} satisfy:
  \begin{enumerate}
  \item For two partitions \(\lambda,\mu\)
    \begin{align*}
      ( P_\lambda(x;\kappa),P_\mu(x;\kappa))_\kappa&=\delta_{\lambda,\mu}
      b_\lambda(\kappa)^{-1}\in\mathcal{O}\\
      b_\lambda(\kappa)&=
      \prod_{s\in\lambda}\frac{\kappa a_\lambda(s)+\ell_\lambda(s)+1}{\kappa a_\lambda(s)+\ell_\lambda(s)+\kappa}\,.
    \end{align*}
  \item For two partitions \(\lambda,\mu\) with \(\ell(\lambda),\ell(\mu)\leq
    N\)
    \begin{align*}
      ( P_\lambda(x;\kappa),P_\mu(x;\kappa))_\kappa^N&=\delta_{\lambda,\mu} \frac{b_\lambda^{N}(\kappa)}{b_\lambda(\kappa)}\in\mathcal{O}\\
      b_\lambda^N(\kappa)&=\prod_{s\in\lambda}\frac{N+\kappa
        a^\prime_\lambda(s)-\ell^\prime_\lambda(s)}
      {N+(a^\prime_\lambda(s)+1)\kappa-\ell^\prime_\lambda(s)-1}\,.
    \end{align*}
  \item For positive integers \(M,N\), let \(\lambda_{N,M}=(M,\dots,M)\)
    denote the length \(N\) partition consisting of \(N\) copies of \(M\).
    The coefficients \(b_{\lambda_{N,M}}^N(\kappa),b_{\lambda_{N,M}}(\kappa)\)
    satisfy the remarkable identity
    \begin{align*}
      b_{\lambda_{N,M}}^N(\kappa)=b_{\lambda_{N,M}}(\kappa)\,.
    \end{align*}
  \end{enumerate}
\end{prop}
As a direct consequence of the above proposition we have the following
proposition.
\begin{prop}\label{sec:innerprodclosesonO}
  Let \(\kappa=\kappa_{0}+\kappa_{1}\epsilon+\cdots\in\mathcal{O}\)
    such that 
    \(\kappa_{0}\in\mathcal{O}\setminus\mathbb{Q}_{\leq0}\)
    and \(\kappa_{1}\neq 0\).
      \begin{enumerate}
      \item The restriction of the map \(\langle\phantom{F}\rangle_{1/\kappa}^N\) of Definition
        \ref{sec:symringmap} to \(\mathcal{O}[z_1^{\pm},\dots,z_N^\pm]^{\symg{N}}\)
        induces a map
        \begin{align*}
          \langle\phantom{F}\rangle_{1/\kappa}^N:\mathcal{O}[z_1^{\pm},\dots,z_N^\pm]^{\symg{N}}\rightarrow \mathcal{O}\,.
        \end{align*}
      \item The restriction of the inner product
        \((\phantom{f},\phantom{g})_\kappa^N\) of Definition
        \ref{sec:innerproddef} to \(\tensor[_{\mathcal{O}}]{\Lambda}{^N}\) induces an inner product        
        \begin{align*}
          (\phantom{f},\phantom{g})_\kappa^N:\tensor[_{\mathcal{O}}]{\Lambda}{^N}\otimes_{\mathcal{O}}\tensor[_{\mathcal{O}}]{\Lambda}{^N}
          \rightarrow \mathcal{O}\,,
        \end{align*}
      \end{enumerate}
\end{prop}

The following proposition is crucial for defining Frobenius homomorphisms in
Section \ref{sec:frobhomsec}.
It was first conjectured by Macdonald \cite{Macdonald:1987}
and was first proven by Kadell \cite{Kadell:1997}.

\begin{prop}\label{sec:Kadellintegral}
  Let
\(\rho=\rho_0+\rho_1\epsilon+\cdots,\sigma=\sigma_0+\sigma_1\epsilon+\cdots,
\tau=\tau_0+\tau_1\epsilon+\cdots\in\mathcal{O}\) such that the constant terms
of \(\sigma,\tau\) lie in \(\mathbb{C}\setminus\mathbb{Q}_{\leq0}\) and
\begin{align*}
  d(d+1)\tau \notin\mathbb{Z}\,,\ d(d-1)\tau + d\rho\notin\mathbb{Z}\,,\ 
  d(d-1)\tau+d \sigma\notin\mathbb{Z}\,,\  1\leq d\leq m\,.
\end{align*}
 Let
     \(\lambda\) be a partition and 
    \(Q_\lambda(x;\kappa)\in\tensor[_{\mathcal{O}}]{\Lambda}{_N}\) a dual
    Jack polynomial. Then the Kadell integral is given by
    \begin{align*}
      I_{\lambda,N}(\rho,\sigma,1/\kappa)&=\int_{[\Delta_N(\rho,\sigma,1/\kappa)]}
        \prod_{i=1}^N x_i^\rho(1-x_i)^\sigma\prod_{1\leq i\neq j\leq
          N}(x_i-x_j)^{1/\kappa} Q_\lambda(x;\kappa)\tfrac{\d x_1\cdots \d
          x_N}{x_1\cdots x_N}\\
        &=S_N(\rho,\sigma,1/\kappa)\frac{\Xi_{\kappa\rho+N-1}(Q_\lambda(x;\kappa))}
        {\Xi_{\kappa(1+\rho+\sigma)+2(N-1)}(Q_\lambda(x;\kappa))}\Xi_{N}(Q_\lambda(x;\kappa))\,.
    \end{align*}
\end{prop}

\subsection{Integrating on the \(\mathcal{O}\)-lattice}

In this section we consider the action of the screening operators on
the Fock modules \(\tensor[_{\mathcal{O}}]{F}{_{r,s}}\) over \(\mathcal{O}\).

\begin{definition}\label{sec:Oscreeningops}\ 
  Let \(r\geq 1\) and \(s\geq 1\). We define the fields \(\scrp{+}{r}(z)\) and
  \(\scrp{-}{s}(z)\) by
  \begin{align*}
    \scrp{+}{r}(z)&=\int_{[\overline{\Gamma}_r(\kappa_+(\epsilon))]}\scr{+}(z)z^{r-1}\prod_{i=1}^{r-1}\scr{+}(zy_i)\prod_{i=1}^{r-1}\d y_i\,,\\
    \scrp{-}{s}(z)&=\int_{[\overline{\Gamma}_s(\kappa_-(\epsilon))]}\scr{-}(z)z^{s-1}\prod_{i=1}^{s-1}\scr{-}(zy_i)\prod_{i=1}^{s-1}\d y_i\,.
  \end{align*}
\end{definition}

\begin{prop}\label{sec:screeningfields}
  Let \(r\geq1,\ s\geq1\) and \(k\in\mathbb{Z}\).
  \begin{enumerate}
  \item   The fields \(\scrp{+}{r}(z)\) and \(\scrp{-}{s}(z)\) lie in
    \(\hom_{\mathcal{K}}(\tensor[_{\mathcal{K}}]{F}{_{r,k}},\tensor[_{\mathcal{K}}]{F}{_{-r,k}})[[z,z^{-1}]]\)
    and
    \(\hom_{\mathcal{K}}(\tensor[_{\mathcal{K}}]{F}{_{k,s}},\tensor[_{\mathcal{K}}]{F}{_{k,-s}})[[z,z^{-1}]]\)
    respectively and they
    are primary fields of conformal weight 1, that is, they satisfy
    \begin{align*}
      T(w)\scrp{+}{r}(z)&=\frac{1}{(w-z)^2}\scrp{+}{r}(z)+\frac{1}{(w-z)}\partial \scrp{+}{r}(z)+\cdots\\
      [L_n,\scrp{+}{r}(z)]&=z^n(z\frac{\d}{\d z}+(n+1))\scrp{+}{r}(z)\\
      T(w)\scrp{-}{s}(z)&=\frac{1}{(w-z)^2}\scrp{-}{s}(z)+\frac{1}{(w-z)}\partial \scrp{-}{s}(z)+\cdots\\
      [L_n,\scrp{-}{s}(z)]&=z^n(z\frac{\d}{\d z}+(n+1))\scrp{-}{s}(z)
    \end{align*}
    for \(n\in\mathbb{Z}\) and \(k=r,s\).
  \item The fields \(\scrp{+}{r}(z),\ \scrp{-}{s}(z)\) admit the the mode expansions
    \begin{align*}
      \scrp{+}{r}(z)&=\sum_{n\in\mathbb{Z}}\scrp{+}{r}[n]z^{-n-1}\,,&
      \scrp{+}{r}[n]&=\res_{z=0}z^{n}\scrp{+}{r}(z)\d z\in
      \hom_{\mathcal{K}}(\tensor[_{\mathcal{K}}]{F}{_{r,s}},\tensor[_{\mathcal{K}}]{F}{_{-r,s}})\\
      \scrp{-}{s}(z)&=\sum_{n\in\mathbb{Z}}\scrp{-}{s}[n]z^{-n-1}\,,&
      \scrp{-}{s}[n]&=\res_{z=0}z^{n}\scrp{-}{s}(z)\d z\in
      \hom_{\mathcal{K}}(\tensor[_{\mathcal{K}}]{F}{_{r,s}},\tensor[_{\mathcal{K}}]{F}{_{r,-s}})\,.
    \end{align*}
  \item For \(n\in\mathbb{Z}\), the modes \(\scrp{+}{r}[n],\scrp{-}{s}[n]\)
    satisfy, respectively,
    \begin{align*}
      \scrp{+}{r}[n](\tensor[_{\mathcal{O}}]{F}{_{r,s}})&\subseteq
      \tensor[_{\mathcal{O}}]{F}{_{-r,s}}\,,\quad r\geq 1, s\in\mathbb{Z}\,,\\
      \scrp{-}{s}[n](\tensor[_{\mathcal{O}}]{F}{_{r,s}})&\subseteq
      \tensor[_{\mathcal{O}}]{F}{_{r,-s}}\,,\quad r\in\mathbb{Z},s\geq 1\,.
    \end{align*}
  \item The zero modes \(\scrp{+}{r}=\scrp{+}{r}[0]\) and
    \(\scrp{-}{s}=\scrp{-}{s}[0]\) are Virasoro homomorphisms, that is, they are
    Virasoro intertwining operators of Fock modules. 
  \end{enumerate}
\end{prop}
\begin{proof}
  We only prove the \(\scrp{+}{r}\) case, since the \(\scrp{-}{s}\) case
  follows by the same arguments. 

  1. The field \(\scrp{+}{r}(z)\) maps from
  \(\tensor[_{\mathcal{K}}]{F}{_{r,k}}\) to
  \(\tensor[_{\mathcal{K}}]{F}{_{-r,k}}\), because each \(\scr{+}(z)\) shifts
  the Heisenberg charge by \(\alpha_+(\epsilon)\) and
  \(\beta_{r,k}(\epsilon)+r\alpha_+(\epsilon)=\beta_{-r,k}(\epsilon)\).
  We prove the operator product expansion formula. The formula for the
  commutator then directly follows from the operator product expansion. By the
  definition of \(\scrp{+}{r}(z)\) we have the equation:
  \begin{align*}
    T(w)\scrp{+}{r}(z)&=\int_{[\overline{\Gamma}_r(\kappa_+(\epsilon))]}T(w)\scr{+}(z)z^{r-1}
    \prod_{i=1}^{r-1}\scr{+}(zy_i)\prod_{i=1}^{r-1}\d y_i\\
    &=\int_{[\overline{\Gamma}_r(\kappa_+(\epsilon))]}
    \left(\frac{1}{(w-z)^2}\scr{+}(z)+\frac{1}{w-z}\partial
      \scr{+}(z)\right)z^{r-1}\prod_{i=1}^{r-1}\scr{+}(zy_i)\prod_{i=1}^{r-1}\d
    y_i\\
    &+\sum_{i=1}^{r-1}\int_{[\overline{\Gamma}_r(\kappa_+(\epsilon))]}
    \scr{+}(z)z^{r-1}\prod_{j=1}^{i-1}\scr{+}(zy_j)\\
    &\phantom{+\sum_{i=1}^{r-1}\int_{[\overline{\Gamma}_r(\kappa_+(\epsilon))]}}\times
    \left(\frac{1}{(w-zy_i)^2}\scr{+}(zy_i)+\frac{1}{w-zy_i}(\partial\scr{+})(zy_i)\right)
    \prod_{j=i+1}^{r-1}\scr{+}(zy_j)\prod_{j=1}^{r-1}\d y_j\,.
  \end{align*}
  Straightforward but somewhat tedious algebraic manipulations simplify the
  above operator product expansions to
  \begin{align*}
    T(w)\scrp{+}{r}(z)&=\frac{1}{(w-z)^2}\scrp{+}{r}(z)+\frac{1}{w-z}\partial\scrp{+}{r}(z)\\
    &-\int_{[\overline{\Gamma}_r(\kappa_+(\epsilon))]}\d_{y}\sum_{i=1}^{r-1}(-1)^{i}
    \left(\frac{1}{w-zy_i}-\frac{y_i}{w-z}\right)\scr{+}(z)z^{r-2}\prod_{j=1}^{r-1}\scr{+}(zy_i)
    \prod_{\substack{1\leq j\leq r-1\\j\neq i}}\d y_j\,,
  \end{align*}
  where \(\d_y\) is the total derivative with respect to the \(y_i\)
  variables, which implies that the integral on the right hand side vanishes.

  2. The mode expansions follow from the fact \(\scrp{+}{r}(z)\) and
  \(\scrp{-}{s}(z)\) conformal weight 1 primary fields.

  3. We note two properties of Jack polynomials that will be helpful for this
  proof. Let \(\lambda=(\lambda_1,\dots,\lambda_r)\) be a partition of length at most \(r\). Due to
  the upper triangular decomposition of Jack polynomials into symmetric
  monomials in Proposition \ref{sec:uppertriang}, there exists a symmetric
  polynomial \(\widetilde{Q}_\lambda(z_1,\dots,z_r;\kappa_+(\epsilon))\) in
  positive powers of the \(z_i\) variables of
  degree \(r \lambda_1-|\lambda|\) such that
  \begin{align*}
    Q_\lambda(z_1^{-1},\dots,z_r^{-1};\kappa_+(\epsilon))=
    \widetilde{Q}_\lambda(z_1,\dots,z_r;\kappa_+(\epsilon))\prod_{i=1}^r z_i^{-\lambda_1}\,.
  \end{align*}
  Let \(\mu=(\lambda_1-\lambda_r,\dots,\lambda_{r-1}-\lambda_r,0)\), then
  \begin{align*}
    Q_\lambda(z_1,\dots,z_r;\kappa_+(\epsilon))=Q_{\mu}(z_1,\dots,z_r;\kappa_+(\epsilon))\prod_{i=1}^r z_i^{\lambda_r}\,,
  \end{align*}
  due to the third part of Proposition \ref{sec:finiteJackprop}.
  
  The mode \(\scrp{+}{r}[n]\) is given by
  \begin{align*}
    \scrp{+}{r}[n]&=\int_{[\Gamma_r(\kappa_+(\epsilon))]}z_1^{n+1}\scr{+}(z_1)\cdots \scr{+}(z_r)\d
    z_1\cdots \d z_r\\
    &=\int_{[\gamma]\times[\overline{\Gamma}_r(\kappa_+(\epsilon))]}
    z^{n+1+s}G_{r-1}((1-r)\kappa_+(\epsilon),2\kappa_+(\epsilon),\kappa_+(\epsilon);y)\\
    &\ \times\prod_{i=1}^{r-1}y_i^{s}
    :\overline{\scr{+}}(z)\prod_{i=1}^{r-1}\overline{\scr{+}}(z y_i):\frac{\d z\d y_1\cdots\d y_{r-1}}{zy_1\cdots y_{r-1}}\,,
  \end{align*}
  where \(G_{r-1}((1-r)\kappa_+(\epsilon),2\kappa_+(\epsilon),\kappa_+(\epsilon);y)\) is the multivalued function of Section 
  \ref{sec:renormcycles}.
  Due to the two properties of Jack polynomials listed at the beginning of this
  proof, the product \(:\overline{\scr{+}}(z)\prod_{i=1}^{r-1}\overline{\scr{+}}(z y_i):\)
  admits a decomposition into Jack polynomials of the form
  \begin{align*}
    :\overline{\scr{+}}(z)\prod_{i=1}^{r-1}\overline{\scr{+}}(zy_i):
    =\sum_{k\in\mathbb{Z}}\sum_{\ell(\lambda)\leq r-1}
    z^{rk}\prod_{i=1}^{r-1}y_i^{k}Q_\lambda(z,zy_1,\dots,zy_{r-1};\kappa_-(\epsilon)) A_{\lambda,k}\,,
  \end{align*}
  where the second summation index runs over all partitions \(\lambda\) of length at most \(r-1\) and
  \(A_{\lambda,k}\) is an element of
  \(U(\bar{\mathfrak{b}})[|\lambda|+k]\). Therefore the action of \(\scrp{+}{r}[n]\) reduces
  to evaluating integrals of the form
  \begin{align*}
    \int_{[\overline{\Gamma}_r(\kappa_+(\epsilon))]}G_{r-1}((1-r)\kappa_+(\epsilon)+k,2\kappa_+(\epsilon),\kappa_+(\epsilon);y)
    Q_\lambda(y;\kappa_-(\epsilon))\frac{\d y_1\cdots\d y_{r-1}}{y_1\cdots y_{r-1}}\,,
  \end{align*}
  where \(k\in\mathbb{Z}\) and \(\lambda\) is a partition of length at most \(r-1\).
  By using the Kadell integral formulae of Proposition
  \ref{sec:Kadellintegral} one sees that all these integrals lie in
  \(\mathcal{O}\) and thus
  \(\scrp{+}{r}[n](\tensor[_{\mathcal{O}}]{F}{_{r,s}})\subseteq
  \tensor[_{\mathcal{O}}]{F}{_{-r,s}}\).

  4. The zero mode \(\scrp{+}{r}[0]\) being a Virasoro homomorphisms follows from \(\scrp{+}{r}(z)\) being a primary field
  of conformal weight 1.
\end{proof}

\begin{definition}
  By setting \(\epsilon=0\) in Definition \ref{sec:Oscreeningops} we define
  the primary fields:
  \begin{align*}
    \scrp{+}{r}(z)&\in\hom_{\mathbb{C}}(F_{r,s},F_{-r,s})[[z,z^{-1}]]\,,\quad r\geq1, s\in\mathbb{Z}\,,\\
    \scrp{-}{s}(z)&\in\hom_{\mathbb{C}}(F_{r,s},F_{r,-s})[[z,z^{-1}]]\,,\quad r\in\mathbb{Z}, s\geq1\,.
  \end{align*}
  In the same way we also define the Virasoro homomorphisms:
  \begin{align*}
  \scrp{+}{r}&=\res_{z=0}\scrp{+}{r}(z)\d z\,,&
  \scrp{-}{s}&=\res_{z=0}\scrp{-}{s}(z)\d z\,.
  \end{align*}
\end{definition}

For each \(\gamma\in\mathbb{C}\setminus\{0\}\), let \(\rho_\gamma\) be the
\(\mathbb{C}\) algebra isomorphism
\begin{align*}
  \rho_\gamma:\Lambda&\rightarrow U(\mathfrak{b}_-)\\
  p_n(x)&\mapsto \gamma b_{-n},\ n=1,2,\dots\,.
\end{align*}
\begin{prop}\label{sec:fockspacesingvec}
  Let \(\lambda_{N,M}=(M,\dots,M)\) be the partition consisting of \(N\) copies of \(M\).
  The action of the screening operators on \(|\beta_{r,s}\rangle\) 
  is given by
  \begin{enumerate}
  \item For \(r\geq 1\) and \(s\in\mathbb{Z}\)
  \begin{align*}
    \scrp{+}{r}:F_{r,s}&\rightarrow F_{-r,s}\,,
  \end{align*}
  such that
  \begin{align*}
    \scrp{+}{r}|\beta_{r,s}\rangle=\left\{
      \begin{array}{cc}
        0&s\geq 1\\
        \rho_{\tfrac{2}{\alpha_+}}(Q_{\lambda_{r,-s}}(x;\kappa_-))
        |\beta_{-r,s}\rangle
        &s\leq 0
      \end{array}\right.\,.
  \end{align*}
\item For \(r\in\mathbb{Z}\) and \(s\geq 1\)
  \begin{align*}
        \scrp{-}{s}:F_{r,s}&\rightarrow F_{r,-s}\,,
  \end{align*}
  such that
  \begin{align*}
    \scrp{-}{s}|\beta_{r,s}\rangle=\left\{
      \begin{array}{cc}
        0&r\geq 1\\
        \rho_{\tfrac{2}{\alpha_-}}(Q_{\lambda_{s,-r}}(x;\kappa_+))
        |\beta_{r,-s}\rangle
        &r\leq 0
      \end{array}\right.\,.
  \end{align*}
\item For \(r,s\geq 1\) 
  \begin{align*}
    \scrp{+}{r}|\beta_{r,-s}\rangle=(-1)^{rs}b_{\lambda_{r,s}}(\kappa_-)\scrp{-}{s}|\beta_{-r,s}\rangle\,.
  \end{align*}
\end{enumerate}
\end{prop}
\begin{proof}
  This proposition is proved by using Proposition \ref{sec:finitevarinnerprod}
  to evaluate the action of the screening operators on
  \(|\beta_{r,s}\rangle\).

  1. Let \(r\geq1\) and \(s\in\mathbb{Z}\). Recall that the conformal weight
  of \(\beta_{r,s}\) is
  \begin{align*}
    h_{r,s}=\frac{r^2-1}{4}\kappa_+-\frac{rs-1}{2}+\frac{s^2-1}{4}\kappa_-\,.
  \end{align*}
  If \(s\geq1\), then \(h_{r,s}<h_{-r,s}\). This means that \(F_{-r,s}\) has no
  states of conformal weight \(h_{r,s}\) and that therefore
  \(\scrp{+}{r}|\beta_{r,s}\rangle=0\).
  If \(s\leq0\), then evaluating \(\scrp{+}{r}\) on \(|\beta_{r,s}\rangle\) yields
  \begin{align*}
    \scrp{+}{r}|\beta_{r,s}\rangle&=\int_{[\Gamma_r(\kappa_+)]}U_r(z;\kappa_-(\epsilon))\prod_{i=1}^r
    z_i^{s}
    \prod_{k\geq1}e^{\alpha_+ \tfrac{p_k(z)}{k}b_{-k}}|\beta_{-r,s}\rangle\frac{\d z_1\cdots \d z_r}{z_1\cdots z_r}\\
    &=\int_{[\Gamma_r(\kappa_+)]}U_r(z;\kappa_-(\epsilon))\prod_{i=1}^r z_i^{s}
    \sum_{\lambda} P_\lambda(z;\kappa_-) \rho_{\tfrac{2}{\alpha_+}}(Q_\lambda(x;\kappa_-))|\beta_{-r,s}\rangle
    \rangle\frac{\d z_1\cdots \d z_r}{z_1\cdots z_r}\\
    &=\sum_\lambda
    (P_{\lambda_{r,-s}}(z;\kappa_-),P_\lambda(z;\kappa_-))_{\kappa_-}^r
    \rho_{\tfrac{2}{\alpha_+}}(Q_\lambda(x;\kappa_-))|\beta_{-r,s}\rangle
    =\rho_{\tfrac{2}{\alpha_+}}(Q_{\lambda_{r,-s}}(x;\kappa_-))|\beta_{-r,s}\rangle\,,
  \end{align*}
  where we have used the identity
  \begin{align*}
    \prod_{k\geq1}e^{\alpha_+
      \tfrac{p_k(z)}{k}b_{-k}}&=\rho_{\tfrac{2}{\alpha_+}}\left(\prod_{k\geq1}e^{\tfrac{1}{\kappa_-}
        \tfrac{p_k(z)p_k(x)}{k}}\right)\\
    &=\sum_\lambda P_\lambda(z;\kappa_-)\rho_{\tfrac{2}{\alpha_+}}(Q_\lambda(x;\kappa_-))\,.
  \end{align*}
  
  2. follows by the same arguments as 1.

  3. follows from part 3 of Proposition \ref{sec:polytofieldmap}.
\end{proof}
\begin{remark}
  The results of Proposition \ref{sec:fockspacesingvec} are truly
  remarkable. The screening operators on the left-hand side of points 1 and 2
  are defined in terms of integrals over renormalised cycles and Jack
  polynomials in a finite number of variables, while the right-hand side is
  written in terms of Jack polynomials in the infinite variable case with
  Heisenberg generators inserted.
\end{remark}
For generic values of the central charge \(c\notin\mathbb{Q}\), the singular vectors of Fock
modules were identified with Jack polynomials by direct calculation in
\cite{Mimachi:1995}.

\section{Virasoro representation theory}
\label{sec:Virrepthy}

The way in which Fock spaces decompose into Virasoro modules was determined
by Feigin and Fuchs in \cite{Feigin:1984,Feigin:1988,FeFu:1990}.
For a more modern and detailed account see \cite{Iohara:2010}. In this section we will see how Fock-modules decompose as
Virasoro modules, calculate the kernels and images of screening operators mapping between Fock-modules and introduce
infinite sums of kernels and images that will later turn out to be \(\mathcal{M}_{p_+,p_-}\)-modules.

Let \(\mathcal{L}\)
be the Virasoro algebra at fixed central charge
\begin{align*}
  c_{p+,p_-}=1-6\frac{(p_+-p_-)^2}{p_+p_-}
\end{align*}
and let \(U(\mathcal{L})\) be the universal enveloping algebra of
\(\mathcal{L}\). The Virasoro vertex operator
algebra \(\operatorname{Vir}_{p_+,p_-}\) is given by the restriction of
\(\mathcal{F}_{p_+,p_-}\) to the subVOA
 \(\operatorname{Vir}_{p_+,p_-}=(U(\mathcal{L})|0\rangle,|0\rangle,
 \tfrac12(b_{-1}^2-\alpha_0 b_2)|0\rangle,Y)\).
Furthermore let \(\tensor[_{\mathcal{K}}]{U(\mathcal{L})}{}\) and
\(\tensor[_{\mathcal{O}}]{U(\mathcal{L})}{}\) be the universal enveloping
algebras of the Virasoro algebra at central charge \(c_{p_+,p_-}(\epsilon)\)
over \(\mathcal{K}\) and \(\mathcal{O}\) respectively.
By \(U(\mathcal{L}_-), \tensor[_{\mathcal{K}}]{U(\mathcal{L}_-)}{}\) and
\(\tensor[_{\mathcal{O}}]{U(\mathcal{L}_-)}{}\)
we denote the universal enveloping algebras of Virasoro generators with negative mode
numbers over \(\mathbb{C}, \mathcal{K}\) and \(\mathcal{O}\) respectively.

\subsection{Virasoro representation theory over {\boldmath\(\mathcal{K}\) and \(\mathcal{O}\)\unboldmath}}

Let
\begin{align*}
  \tensor[_{\mathcal{K}}]{K}{_{1,1}}&=\ker(\scr{+}:\tensor[_{\mathcal{K}}]{F}{_{1,1}}\rightarrow\tensor[_{\mathcal{K}}]{F}{_{-1,1}})
  \cap \ker(\scr{-}:\tensor[_{\mathcal{K}}]{F}{_{1,1}}\rightarrow\tensor[_{\mathcal{K}}]{F}{_{1,-1}})\\
  \tensor[_{\mathcal{O}}]{K}{_{1,1}}&=\ker(\scr{+}:\tensor[_{\mathcal{O}}]{F}{_{1,1}}\rightarrow\tensor[_{\mathcal{O}}]{F}{_{-1,1}})
  \cap \ker(\scr{-}:\tensor[_{\mathcal{O}}]{F}{_{1,1}}\rightarrow\tensor[_{\mathcal{O}}]{F}{_{1,-1}})\,,
\end{align*}
then \((\tensor[_{\mathcal{O}}]{K}{_{1,1}},|0\rangle,\tfrac12(b_{-1}^2+\alpha_0(\epsilon)b_{-2})|0\rangle,Y)\) carries the structure
of a VOA over \(\mathcal{O}\) and is an \(\mathcal{O}\)-lattice of the
VOA \((\tensor[_{\mathcal{K}}]{K}{_{1,1}},|0\rangle,\tfrac12(b_{-1}^2+\alpha_0(\epsilon)b_{-2})|0\rangle,Y)\).

\begin{remark}
  It is known that
  \((\tensor[_{\mathcal{K}}]{K}{_{1,1}},|0\rangle,\tfrac12(b_{-1}^2+\alpha_0(\epsilon)b_{-2})|0\rangle,Y)\)
  is isomorphic to the Virasoro VOA over \(\mathcal{K}\) at central charge \(c_{p_+,p_-}(\epsilon)\).
  However, as we will see in Section \ref{sec:Mppalg}, the VOA
  \((\tensor[_{\mathcal{O}}]{K}{_{1,1}}\otimes_{\mathcal{O}}\mathbb{C},|0\rangle,\tfrac12(b_{-1}^2+\alpha_0(0)b_{-2})|0\rangle,Y)\) is
  larger than just the Virasoro VOA at central charge \(c_{p_+,p_-}\).
\end{remark}

For each \(h\in\mathcal{K}\), let \(\tensor[_{\mathcal{K}}]{M}{}(h)\) be the 
\(\tensor[_{\mathcal{K}}]{U}{}(\mathcal{L})\)-Verma module with highest weight
\(h\). 
\begin{prop}\label{sec:vermatofock}\ 
  \begin{enumerate}
  \item The Verma module \(\tensor[_{\mathcal{K}}]{M}{}(h)\) is not
    simple as a \(\tensor[_{\mathcal{K}}]{U}{}(\mathcal{L})\)
    module if and only if \(h=h_{r,s}(\epsilon)\) for some \(r\geq
    1,s\geq1\).
  \item For each \(\beta\in\mathcal{K}\), consider the left
    \(\tensor[_{\mathcal{K}}]{U}{}(\mathcal{L})\)-module
    \(\tensor[_{\mathcal{K}}]{F}{_\beta}\), then there is a canonical 
    \(\tensor[_{\mathcal{K}}]{U}{}(\mathcal{L})\)-module map
    \begin{align*}
      \tensor[_{\mathcal{K}}]{M}{_{h_\beta}}&\rightarrow\tensor[_{\mathcal{K}}]{F}{_\beta}\\
      u_{h_\beta}&\mapsto |\beta\rangle\,,
    \end{align*}
    where \(u_{h_\beta}\) is the highest weight state that generates
    \(\tensor[_{\mathcal{K}}]{M}{_{h_\beta}}\). This map is not an isomorphism if and only if
    \(\beta=\beta_{r,s}(\epsilon)\) for some \(r\geq1, s\geq 1\).
  \item Let \(r\geq1,s\geq1\). The sequences
    \begin{align*}
      0\rightarrow&\tensor[_{\mathcal{K}}]{L}{}(h_{r,s}(\epsilon))\longrightarrow
      \tensor[_{\mathcal{K}}]{F}{_{r,s}}\overset{\scrp{+}{r}}{\longrightarrow}
      \tensor[_{\mathcal{K}}]{F}{_{-r,s}}\rightarrow 0\\
      0\rightarrow&\tensor[_{\mathcal{K}}]{L}{}(h_{r,s}(\epsilon))\longrightarrow
      \tensor[_{\mathcal{K}}]{F}{_{r,s}}\overset{\scrp{-}{s}}{\longrightarrow}
      \tensor[_{\mathcal{K}}]{F}{_{r,-s}}\rightarrow 0\\
      0\rightarrow&\tensor[_{\mathcal{K}}]{F}{_{r,-s}}\overset{\scrp{+}{r}}{\longrightarrow}
      \tensor[_{\mathcal{K}}]{F}{_{-r,-s}}\longrightarrow
      \tensor[_{\mathcal{K}}]{L}{}(h_{-r,-s}(\epsilon))\rightarrow 0\\
      0\rightarrow&\tensor[_{\mathcal{K}}]{F}{_{-r,s}}\overset{\scrp{-}{s}}{\longrightarrow}
      \tensor[_{\mathcal{K}}]{F}{_{-r,-s}}\longrightarrow
      \tensor[_{\mathcal{K}}]{L}{}(h_{-r,-s}(\epsilon))\rightarrow 0
    \end{align*}
    are exact as \(\tensor[_{\mathcal{K}}]{U}{}(L)\)-modules and the screening
    operators \(\scrp{+}{r}\) and \(\scrp{-}{s}\) commute.
  \item\label{item:vermafockisom} The map
    \begin{align*}
      \tensor[_{\mathcal{K}}]{M}{}(h_{r,s}(\epsilon))&\rightarrow
      \tensor[_{\mathcal{K}}]{F}{_{-r,-s}}\,,\\
      u_{h_{r,s}(\epsilon)}&\mapsto |\beta_{-r,-s}(\epsilon)\rangle
    \end{align*}
    is an isomorphism of left
    \(\tensor[_{\mathcal{K}}]{U}{}(\mathcal{L})\)-modules and so is its dual
    \begin{align*}
    \tensor[_{\mathcal{K}}]{F}{_{r,s}}=\tensor[_{\mathcal{K}}]{F}{_{-r,-s}^\ast}&\rightarrow
      \tensor[_{\mathcal{K}}]{M}{}(h_{r,s}(\epsilon))^\ast\,.
    \end{align*}
    If we restrict the domains of the above
    \(\tensor[_{\mathcal{K}}]{U}{}(\mathcal{L})\)-module isomorphisms to
    \(\tensor[_{\mathcal{O}}]{M}{}(h_{r,s}(\epsilon))\)
    and \(\tensor[_{\mathcal{O}}]{F}{_{r,s}}\) we obtain \(\tensor[_{\mathcal{O}}]{U}{}(\mathcal{L})\)-module homomorphisms
    \begin{align*}
      \tensor[_{\mathcal{O}}]{M}{}(h_{r,s}(\epsilon))&\rightarrow \tensor[_{\mathcal{O}}]{F}{_{-r,-s}}\\
      \tensor[_{\mathcal{O}}]{F}{_{r,s}}&\rightarrow \tensor[_{\mathcal{O}}]{M}{}(h_{r,s}(\epsilon))^\ast\,,
    \end{align*}
    where
    \begin{align*}
      \tensor[_{\mathcal{O}}]{M}{}(h_{r,s}(\epsilon))&=\tensor[_{\mathcal{O}}]{U}{}(\mathcal{L})u_{h_{r,s}(\epsilon)}
      \subset \tensor[_{\mathcal{K}}]{M}{}(h_{r,s}(\epsilon))\\
      \tensor[_{\mathcal{O}}]{M}{}(h_{r,s}(\epsilon))^\ast&=\bigoplus_{d\geq 0}\operatorname{Hom}_{\mathcal{O}}
      (\tensor[_{\mathcal{O}}]{M}{}(h_{r,s}(\epsilon))[h_{r,s}(\epsilon)+d],\mathcal{O})\,.
    \end{align*}
  \item For each \(r\geq 1,s\geq 1\) there exists a unique element 
    \(S_{r,s}(\kappa)\in\tensor[_{\mathcal{O}}]{U}{}(\mathcal{L}_{-})\) such that
    \begin{align*}
      S_{r,s}(\kappa)=(L_{-1})^{rs}+\cdots
    \end{align*}
    which satisfies
    \begin{align*}
      L_nS_{r,s}(\kappa)u_{h_{r,s}(\epsilon)}&=0,\ n\geq 1\\
      L_0S_{r,s}(\kappa)u_{h_{r,s}(\epsilon)}&=(h_{r,s}(\epsilon)+rs)S_{r,s}(\kappa)u_{h_{r,s}(\epsilon)}\\
      h_{r,s}(\epsilon)+rs&=h_{r,-s}(\epsilon)=h_{-r,s}(\epsilon)\,.
    \end{align*}
  \item The conformal weight \(h_{r,s}(\epsilon)+rs=h_{r,-s}(\epsilon)=h_{-r,s}(\epsilon)\) is precisely the conformal
    weight of \(S_{r,s}(\kappa)u_{h_{r,s}(\epsilon)},\
    \scrp{+}{r}|\beta_{r,-s}(\epsilon)\rangle\)
    and \(\scrp{+}{s}|\beta_{-r,s}(\epsilon)\rangle\).
    Under the identification
    \begin{align*}
      \tensor[_{\mathcal{K}}]{M}{}(h_{r,s})\cong \tensor[_{\mathcal{K}}]{F}{_{-r,-s}}
    \end{align*}
    these three singular vectors are proportional to each other.
  \end{enumerate}
\end{prop}
\begin{proof}
  By construction the Virasoro representation over \(\mathcal{K}\) is equivalent to working over
  \(\mathbb{C}\) at generic central charge. The above statements on Virasoro representation theory
  at generic central charge are shown in
  \cite{Feigin:1984,Feigin:1988,FeFu:1990}. See \cite{Iohara:2010} for
  detailed proofs. The statements regarding Virasoro representation theory over \(\mathcal{O}\) follow by restriction.
\end{proof}

\subsection{The category {\boldmath \(U(\mathcal{L})\)\unboldmath}-mod of Virasoro 
  modules at central charge {\boldmath \(c_{p_+,p_-}\)\unboldmath}}

The purpose of this subsection is to give a summary of Virasoro representation theory and to give
socle sequence decompositions of Fock-modules in terms of simple Virasoro modules.

\begin{definition}
  Recall that \(X\) was defined to be the lattice of the Heisenberg weights
  that appear in lattice modules:
  \begin{align*}
    X=\mathbb{Z}\frac{1}{\sqrt{p_+p_-}}\,.
  \end{align*}
  Let
  \begin{align*}
    H=\{h_\beta|\beta\in X\}
  \end{align*}
  be
  \emph{the set of highest conformal weights.}
  Two non-equal Heisenberg weights \(\beta,\beta^\prime\in X\) correspond to
  the same conformal weight if and only if
  \(\beta^\prime=\alpha_0-\beta\).
\end{definition}

\begin{definition}
  \begin{enumerate}
  \item For \(1\leq r<p_+,\ 1\leq s<p_-\), let
    \begin{align*}
      \Delta_{r,s}=\Delta_{p_+-r,p_--s}=h_{r,s;0}\,.
    \end{align*}
  \item \emph{The Kac table} \(\mathcal{T}\) is the quotient set
    \begin{align*}
      \mathcal{T}=\{(r,s)|1\leq r<p_+,1\leq s<p_-\}/\sim\,,
    \end{align*}
    where \((r,s)\sim(r^\prime,s^\prime)\) if and only if \(r^\prime=p_+-r,s^\prime=p_--s\). The Kac table
    is the set of all classes \([(r,s)]\) such that the conformal weights \(\Delta_{r,s}\) are distinct.
  \item For \(1\leq r\leq p_+,\ 1\leq s\leq p_-,\ n\geq 0\) let
    \begin{align*}
      \Delta_{r,s;n}^+&=\left\{
        \begin{array}{cc}
          h_{p_+,p_-;-2n}&r=p_+,s=p_-\\
          h_{p_+-r,p_-;-2n-1}&r\neq p_+,s=p_-\\
          h_{p_+,p_--s;2n+1}&r=p_+,s\neq p_-\\
          h_{p_+-r,s;-2n-1}&r\neq p_+, s\neq p_-
        \end{array}\right.\,,&
      \Delta_{r,s;n}^-&=\left\{
        \begin{array}{cc}
          h_{p_+,p_-;-2n-1}&r=p_+,s=p_-\\
          h_{p_+-r,p_-;-2n-2}&r\neq p_+,s=p_-\\
          h_{p_+,p_--s;2n+2}&r=p_+,s\neq p_-\\
          h_{p_+-r,s;-2n-2}&r\neq p_+, s\neq p_-
        \end{array}\right.\,.
    \end{align*}
  \end{enumerate}
\end{definition}

\begin{definition}
  Let \(U(\mathcal{L})\)-Mod be the abelian category whose
  morphisms are Virasoro-homomorphisms and whose objects are left
  \(U(\mathcal{L})\) modules that satisfy the following:
\begin{enumerate}
\item Every object \(M\) decomposes into a direct sum of generalised \(L_0\) eigenspaces
    \begin{align*}
      M&=\bigoplus_{h\in\mathbb{C}}M[h]\\
      M[h]&=\{u\in M| \exists n\geq 1, \text{ s.t. } (L_0-h)^nu=0\}
    \end{align*}
    where \(\dim M[h]<\infty\). For all \(h\in\mathbb{C}\) and there are only a
    countable number of \(h\) for which \(M[h]\) is non-trivial.
  \item For  every object \(M\in U(\mathcal{L})\)-Mod,
    there exists the contragredient object \(M^\ast\)
    \begin{align*}
      M^\ast=\bigoplus_{h\in\mathbb{C}}\hom(M[h],\mathbb{C})\,,
    \end{align*}
    on which the anti-involution \(\sigma(L_n)=L_{-n}\) induces the structure of a left
    \(U(\mathcal{L})\)-module by
    \begin{align*}
      \langle L_n \phi,u\rangle=\langle \phi, \sigma(L_n)u\rangle,\ \phi\in M^\ast, u\in M\,.
    \end{align*}
    Note that \((M^\ast)^\ast\cong M\).
  \end{enumerate}
\end{definition}

\begin{definition}
  Let \(M\in U(\mathcal{L})\)-Mod be a Virasoro module.
  \begin{enumerate}
  \item The Virasoro module \(M\) is called semi-simple if and only if it is
    isomorphic to a at most countably infinite direct sum of simple Virasoro modules.
  \item If there exist non-zero semi-simple submodules of \(M\), then we
    denote the maximal semi-simple submodule of \(M\) by \(S_1(M)\), that is,
    \(S_1(M)\) is semi-simple and any semi-simple submodule of \(M\) lies in
    \(S_1(M)\). Note that if \(M\) is semi-simple, then \(S_1(M)=M\).  The
    semi-simple module \(S_1(M)\) is called the \emph{socle} of \(M\).
  \item Let
    \begin{align*}
      0=M_0\subset M_1 \subset M_2\subset\cdots
    \end{align*}
    be an ascending sequence of submodules of \(M\), such that
    \(S_i(M)=M_i/M_{i-1}\), \(i\geq 1\) is the maximal semi-simple submodule
    of \(M/M_{i-1}\), that is, for each \(i\geq1\), the socle of
    \(M/M_{i-1}\) is \(S_i(M)\). 
    The semi-simple simple modules \(S_i(M)\) are
    called \emph{the components of \(M\)}, while the sequence
    \(\{S_i(M)\}_{i\geq1}\) is called \emph{the socle sequence} of \(M\).  If
    there exists an element \(M_n\) of the filtration, such that, \(M_n=M\)
    and \(M_{n-1}\neq M\), then we say that the socle sequence has length \(n\)
    or that it has finite length. Socle sequences are unique if they exist.
  \end{enumerate}
\end{definition}

\begin{definition}
  We define \(U(\mathcal{L})\)-mod to be the full subcategory of \(U(\mathcal{L})\)-Mod
  such that all objects \(M\) in \(U(\mathcal{L})\)-mod satisfy the following two conditions:
  \begin{enumerate}
  \item The socle sequence of \(M\) has finite length,
  \item The conformal weights \(h\) of the simple modules \(L(h)\),
    appearing in the components of \(M\),
    are elements of \(H\).
  \end{enumerate}
\end{definition}
  Note that the  Verma modules \(M(h_{r,s})\) and their duals \(M(h_{r,s})^\ast\) are not objects of \(U(\mathcal{L})\)-mod
  because they do not admit finite length socle sequences.
  However, the Fock modules \(F_\beta, \beta\in X\) and the simple modules
  \(L(h), h\in H\) are objects of \(U(\mathcal{L})\)-mod.

\begin{prop}\label{sec:socles}\ 
  \begin{enumerate}
  \item For each \(\beta\in X\) the Fock module \(F_\beta\) is an 
    object of \(U(\mathcal{L})\)-mod.
  \item There are four cases of socle sequence for the Fock modules \(F_{r,s;n}\)
  \begin{trivlist}
  \item \((I)\) For \(1\leq r< p_+\), \(1\leq s<p_-\), \(n\in\mathbb{Z}\),
    \begin{align*}
      S_1(F_{r,s;n})&=\bigoplus_{k\geq 0}L(h_{r,p_--s;|n|+2k+1})\,,\\
      S_2(F_{r,s;n})&=\bigoplus_{k\geq a}L(h_{r,s;|n|+2k})\\\nonumber
      &\phantom{=}\oplus\bigoplus_{k\geq 1-a}L(h_{p_+-r,p_--s;|n|+2k})\,,\\\nonumber
      S_3(F_{r,s;n})&=\bigoplus_{k\geq 0}L(h_{p_+-r,s;|n|+2k+1})\,,\\
    \end{align*}
    where \(a=0\) if \(n\geq0\) and \(a=1\) if \(n<0\).
  \item \((II_+)\) For \(1\leq s<p_-\), \(n\in\mathbb{Z}\),
    \begin{align*}
      S_1(F_{p_+,s;n})&=\bigoplus_{k\geq 0}L(h_{p_+,p_--s;|n|+2k+1})\,,\\\nonumber
      S_2(F_{p_+,s;n})&=\bigoplus_{k\geq a}L(h_{p_+,s;|n|+2k})\,,
    \end{align*}
    where \(a=0\) if \(n\geq 1\) and \(a=1\) if \(n<1\).
  \item \((II_-)\) For \(1\leq r< p_+\), \(n\in\mathbb{Z}\),
    \begin{align*}
      S_1(F_{r,p_-;n})&=\bigoplus_{k\geq 0}L(h_{r,p_-;|n|+2k})\,,\\\nonumber
      S_2(F_{r,p_-;n})&=\bigoplus_{k\geq a}L(h_{p_+-r,p_-;|n|+2k-1})\,,
    \end{align*}
    where \(a=1\) if \(n\geq 0\) and \(a=0\) if \(n<0\).
  \item \((III)\) For \(n\in\mathbb{Z}\), the Fock space \(F_{p_+,p_-;n}\) is
    semi-simple as a Virasoro module
    \begin{align*}
      S_1(F_{p_+,p_-;n})=F_{p_+,p_-;n}=\bigoplus_{k\geq 0}L(h_{p_+,p_-;|n|+2k})\,.
    \end{align*}
  \end{trivlist}
  \end{enumerate}
\end{prop}
The above socle sequences are originally due to
\cite{FeFu:1990}, however, a more comprehensive
explanation can be found in \cite{Iohara:2010}.
Figure \ref{fig:fock1} is a visualisation of the socle sequence decomposition of Fock modules.
\begin{figure}
  \centering
  \begin{tikzpicture}[scale=0.9, >=latex]
    \node at (1.5,2.5) {\((\text{I})\ F_{r,s;n}\,\  1\leq r<p_+,\ 1\leq s<p_-\)};
    \node at (-1.5,0) {\(n\geq0\)};
    \path (0,0) node[socmida] (1) {0} 
    ++(2,-0.75) node[soctop] (2l) {1} edge[->] (1) 
    ++(0,1.5) node[socbot] (2r) {1} edge[<-] (1) 
    ++(2,-1.5)  node[socmidb] (3l) {2} edge[<-] (2l) edge[->] (2r) 
    ++(0,1.5) node[socmida] (3r) {2} edge[<-] (2l) edge[->] (2r) 
    ++(2,-1.5) node[soctop] (4l) {3} edge[->] (3l) edge[->]  (3r) 
    ++(0,1.5) node[socbot] (4r) {3} edge[<-] (3l) edge[<-]  (3r)
    ++(2,-1.5) node[socmidb] (5l) {4} edge[<-] (4l) edge[->]   (4r) 
    ++(0,1.5) node[socmida] (5r) {4} edge[<-] (4l)  edge[->] (4r) 
    ++(1,0) node (6l) {\footnotesize\(\cdots\)}
    ++(0,-1.5) node (6r) {\footnotesize\(\cdots\)};
    \path (-1.5,1.5) node[soctop] {\(k\)}
    ++(1.8,0) node {\(=h_{p_+-r,s;n+k},\)}
    ++(1.8,0) node[socmida] {\(k\)}
    ++(1.5,0) node {\(=h_{r,s;n+k},\)}
    ++(1.5,0) node[socmidb] {\(k\)}
    ++(2.1,0) node {\(=h_{p_+-r,p_--s;n+k},\)}
    ++(2.1,0) node[socbot] {\(k\)}
    ++(1.7,0) node {\(=h_{r,p_--s;n+k}\)};
    \node at (-1.5,-2.5) {\(n\leq0\)};
    \path (0,-2.5) node[socmida] (7) {0} 
    ++(2,-0.75) node[soctop] (8l) {1} edge[->] (7) 
    ++(0,1.5) node[socbot] (8r) {1} edge[<-] (7) 
    ++(2,-1.5)  node[socmidb] (9l) {2} edge[<-] (8l) edge[->] (8r) 
    ++(0,1.5) node[socmida] (9r) {2} edge[<-] (8l) edge[->] (8r) 
    ++(2,-1.5) node[soctop] (10l) {3} edge[->] (9l) edge[->]  (9r) 
    ++(0,1.5) node[socbot] (10r) {3} edge[<-] (9l) edge[<-]  (9r)
    ++(2,-1.5) node[socmidb] (11l) {4} edge[<-] (10l) edge[->]   (10r) 
    ++(0,1.5) node[socmida] (11r) {4} edge[<-] (10l)  edge[->] (10r) 
    ++(1,0) node (12l) {\footnotesize\(\cdots\)}
    ++(0,-1.5) node (12r) {\footnotesize\(\cdots\)};
    \path (-1.5,-4) node[soctop] {\(k\)}
    ++(1.8,0) node {\(=h_{p_+-r,s;-n+k},\)}
    ++(1.9,0) node[socmida] {\(k\)}
    ++(2.2,0) node {\(=h_{p_+-r,p_--s;-n+k},\)}
    ++(2.2,0) node[socmidb] {\(k\)}
    ++(1.5,0) node {\(=h_{r,s;-n+k},\)}
    ++(1.5,0) node[socbot] {\(k\)}
    ++(1.8,0) node {\(=h_{r,p_--s;-n+k}\)};
  \end{tikzpicture}\vspace{2mm}\\
  \begin{tikzpicture}[scale=0.9, >=latex]
    \node at (0.5,2) {\((\text{II}_-)\ F_{p_+,s;n}\,, 1\leq s<p_-\)};
    \node at (-1.5,0)   {\(n\geq1\)}; 
    \path (0,0) node[soctop] (1) {0}
    ++(2,0) node[socbot] (2) {1} edge[<-] (1)
    ++(2,0) node[soctop] (3) {2} edge[->] (2)
    ++(2,0) node[socbot] (4) {3} edge[<-] (3)
    ++(2,0) node[soctop] (5) {4} edge[->] (4)
    ++(1,0) node (6) {\footnotesize\(\cdots\)};
    \path (-1.5,1) node[soctop] {\(k\)}
    ++(1.5,0) node {\(=h_{p_+,s;n+k},\)}
    ++(1.6,0) node[socbot] {\(k\)}
    ++(1.8,0) node {\(=h_{p_+,p_--s;n+k}\)};
    \node at (-1.5,-1)   {\(n<1\)}; 
    \path (0,-1) node[socbot] (7) {1}
    ++(2,0) node[soctop] (8) {2} edge[<-] (7)
    ++(2,0) node[socbot] (9) {3} edge[->] (8)
    ++(2,0) node[soctop] (10) {4} edge[<-] (9)
    ++(2,0) node[socbot] (11) {5} edge[->] (10)
    ++(1,0) node (12) {\footnotesize\(\cdots\)};
    \path (-1.5,-2) node[soctop] {\(k\)}
    ++(2,0) node {\(=h_{p_+,p_--s;-n+k},\)}
    ++(2,0) node[socbot] {\(k\)}
    ++(1.6,0) node {\(=h_{p_+,s;-n+k}\)};
  \end{tikzpicture}\vspace{2mm}\\
\begin{tikzpicture}[scale=0.9, >=latex]
    \node at (0.5,2) {\((\text{II}_+)\ F_{r,p_-;n}\,, 1\leq r<p_+\)};
    \node at (-1.5,0)   {\(n\geq0\)}; 
    \path (0,0) node[socbot] (1) {0}
    ++(2,0) node[soctop] (2) {1} edge[->] (1)
    ++(2,0) node[socbot] (3) {2} edge[<-] (2)
    ++(2,0) node[soctop] (4) {3} edge[->] (3)
    ++(2,0) node[socbot] (5) {4} edge[<-] (4)
    ++(1,0) node (6) {\footnotesize\(\cdots\)};
    \path (-1.5,1) node[soctop] {\(k\)}
    ++(1.9,0) node {\(=h_{p_+-r,p_-;n+k},\)}
    ++(2,0) node[socbot] {\(k\)}
    ++(1.5,0) node {\(=h_{r,p_-;n+k}\)};
    \node at (-1.5,-1)   {\(n<0\)}; 
    \path (0,-1) node[soctop] (7) {-1}
    ++(2,0) node[socbot] (8) {0} edge[<-] (7)
    ++(2,0) node[soctop] (9) {1} edge[->] (8)
    ++(2,0) node[socbot] (10) {2} edge[<-] (9)
    ++(2,0) node[soctop] (11) {3} edge[->] (10)
    ++(1,0) node (12) {\footnotesize\(\cdots\)};
    \path (-1.5,-2) node[soctop] {\(k\)}
    ++(1.7,0) node {\(=h_{r,p_-;-n+k},\)}
    ++(1.8,0) node[socbot] {\(k\)}
    ++(1.9,0) node {\(=h_{p_+-r,p_-;-n+k}\)};
  \end{tikzpicture}
  \caption{The socle sequences of Fock modules:
  The arrows indicate which states one can reach by acting with the Virasoro algebra, that is,
  an arrow \(v\rightarrow w\) means \(w\in U(\mathcal{L})v\).
}
\label{fig:fock1}
\end{figure}
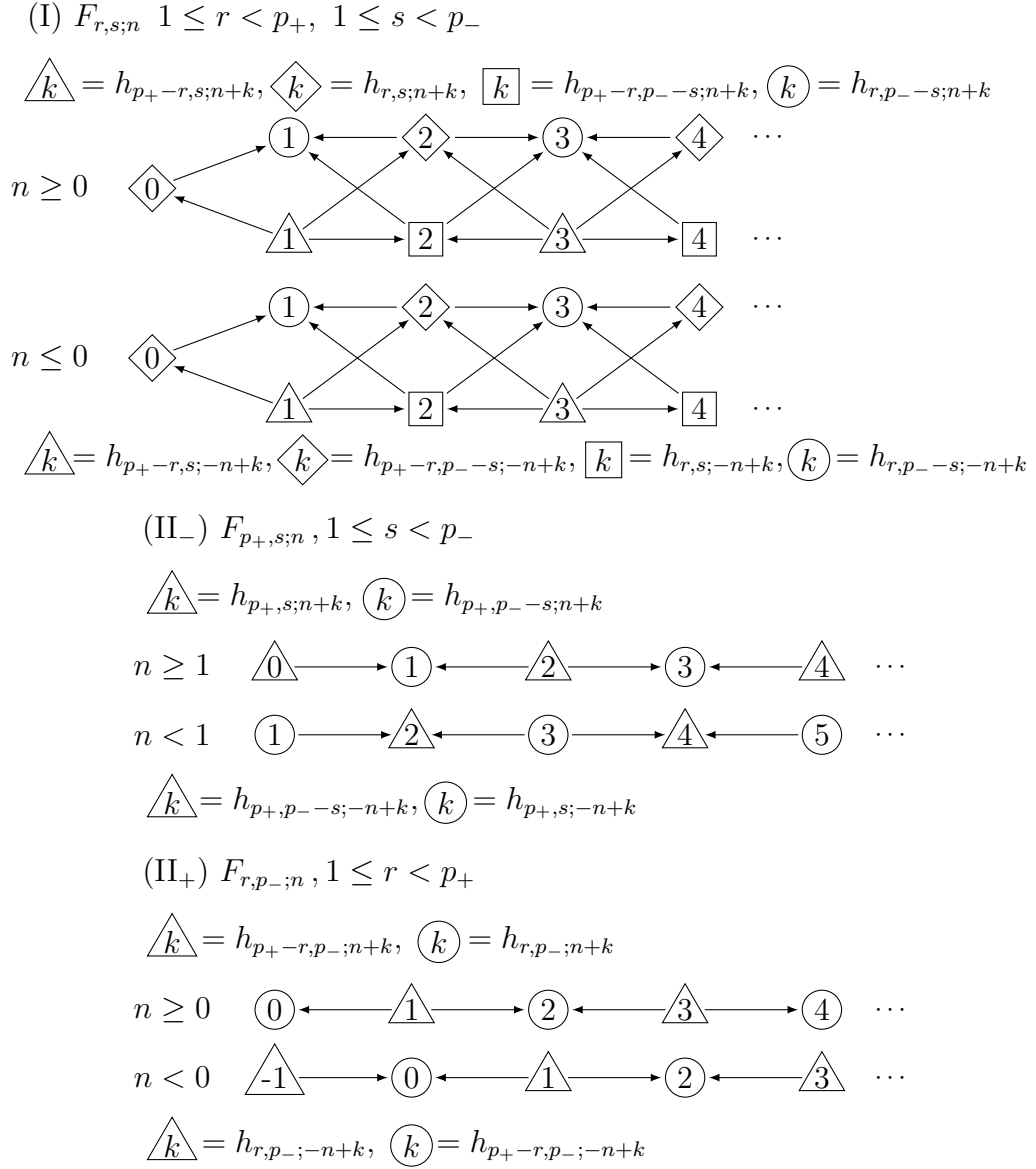

\subsection{Kernels and images of screening operators}

In this subsection we give the socle sequences of the kernels and images of the screening operators \(\scrp{+}{r},\ \scrp{-}{s}\).
\begin{definition}\label{sec:kersandims}
  We denote the kernels and images of screening operators \(\scrp{+}{r},\ \scrp{-}{s}\) by
  \begin{enumerate}
  \item For \(1\leq r<p_+\), \(1\leq s\leq p_-\), \(n\in \mathbb{Z}\)
    \begin{align*}
      K_{r,s;n;+}&=\ker \scrp{+}{r}:F_{r,s;n}\rightarrow F_{p_+-r,s;n+1}\\
      X_{p_+-r,s;n+1;+}&=\im \scrp{+}{r}:F_{r,s;n}\rightarrow F_{p_+-r,s;n+1}
    \end{align*}
  \item For \(1\leq r\leq p_+\), \(1\leq s< p_-\), \(n\in \mathbb{Z}\)
    \begin{align*}
      K_{r,s;n;-}&=\ker \scrp{-}{s}:F_{r,s;n}\rightarrow F_{r,p_--s;n-1}\\
      X_{r,p_--s;n-1;-}&=\im \scrp{-}{s}:F_{r,s;n}\rightarrow F_{r,p_--s;n-1}
    \end{align*}
  \item For \(1\leq r< p_+\), \(1\leq s< p_-\), \(n\in \mathbb{Z}\)
    \begin{align*}
      K_{r,s:n}&=K_{r,s;n;+}\cap K_{r,s;n;-}&X_{r,s:n}&=X_{r,s;n;+}\cap
      X_{r,s;n;-}\\
      K_{r,p_-;n}&=K_{r,p_-;n;+}&K_{p_+,s;n}&=K_{p_+,s;n;-}\\
      K_{p_+,p_-;n}&=X_{p_+,p_-;n}=F_{p_+,p_-;n}\,.
    \end{align*}
  \end{enumerate}
\end{definition}

\begin{prop}\label{sec:ksocs}
  For \(1\leq r<p_+\), \(1\leq s<p_-\)
  the socle sequences of the kernels and images of the screening operator
  \(\scr{+}\)  are given by
  \begin{align*}
    n&\geq 0&n&\leq -1\\
    S_1(K_{r,s;n;+})&=\bigoplus_{k\geq 1}L(h_{r,p_--s;n+2k-1})\,,&
    S_1(K_{r,s;n;+})&=\bigoplus_{k\geq 1}L(h_{r,p_--s;-n+2k-1})\,,\\
    S_2(K_{r,s;n;+})&=\bigoplus_{k\geq 1}L(h_{r,s;n+2(k-1)})\,,&
    S_2(K_{r,s;n;+})&=\bigoplus_{k\geq 1}L(h_{r,s;-n+2k})\,,\\
    S_1(X_{r,s;n+1;+})&=\bigoplus_{k\geq 1}L(h_{r,p_--s;n+2k})\,,&
    S_1(X_{r,s;n+1;+})&=\bigoplus_{k\geq 1}L(h_{r,p_--s;-n+2(k-1)})\,,\\
    S_2(X_{r,s;n+1;+})&=\bigoplus_{k\geq 1}L(h_{r,s;n+2k-1})\,,&
    S_2(X_{r,s;n+1:+})&=\bigoplus_{k\geq 1}L(h_{r,s;-n+2k-1})\,,
  \end{align*}
  and those of the screening operator \(\scr{-}\) by
  \begin{align*}
    n&\geq 1&n&\leq 0\\
    S_1(K_{r,s;n;-})&=\bigoplus_{k\geq 1}L(h_{r,p_--s;n+2k-1)})\,,&
    S_1(K_{r,s;n;-})&=\bigoplus_{k\geq 1}L(h_{r,p_--s;-n+2k-1)})\,,\\    
    S_2(K_{r,s;n;-})&=\bigoplus_{k\geq 1}L(h_{p_+-r,p_--s;n+2k})\,,&
    S_2(K_{r,s;n;-})&=\bigoplus_{k\geq 1}L(h_{p_+-r,p_--s;-n+2(k-1)})\,,\\    
    S_1(X_{r,s;n+1;-})&=\bigoplus_{k\geq 1}L(h_{r,p_--s;n+2(k-1)})\,,&
    S_1(X_{r,s;n+1;-})&=\bigoplus_{k\geq 1}L(h_{r,p_--s;-n+2k})\,,\\    
    S_2(X_{r,s;n+1;-})&=\bigoplus_{k\geq 1}L(h_{p_+-r,p_--s;n+2k-1})\,,&
    S_2(X_{r,s;n+1;-})&=\bigoplus_{k\geq 1}L(h_{p_+-r,p_--s;-n+2k-1})\,.
  \end{align*}
  For \(r=p_+\) or \(s=p_-\) the kernels and images are semisimple and we have
  \begin{align*}
    K_{r,p_-;n;+}&=X_{r,p_-;n;+}=S_1(F_{r,p_-;n})\,,&K_{p_+,s;n;-}&=X_{p_+,s;n;-}=S_1(F_{p_+,s;n})\,.
  \end{align*}
\end{prop}
The above proposition is due to Felder \cite{Fel:1989}. It can be proved by means of the socle sequences in Propositions
\ref{sec:socles} and the following proposition due to Felder \cite{Fel:1989}.
\begin{prop}
  \begin{enumerate}
  \item For \(1\leq r<p_+\), \(1\leq s\leq p_-\) and \(n\in\mathbb{Z}\)
    the screening operator \(\scr{+}\) defines the \emph{Felder complex}
    \begin{align*}
      \cdots\stackrel{\scrp{+}{r}}{\longrightarrow}
      F_{p_+-r,s;n-1}\stackrel{\scrp{+}{p_+-r}}{\longrightarrow}
      F_{r,s,n}\stackrel{\scrp{+}{r}}{\longrightarrow}F_{p_+-r,s;n+1}
      \stackrel{\scrp{+}{p_+-r}}{\longrightarrow}\cdots\,.
    \end{align*}
    This complex is exact for \(s=p_-\). For \(1\leq s< p_-\) it is exact
    everywhere except in \(F_{r,s;0}\), where the cohomology is given by
    \begin{align*}
      K_{r,s;0;+}/X_{r,s;0;+}\cong L(h_{r,s;0})\,.
    \end{align*}
  \item For \(1\leq r \leq p_+\), \(1\leq s< p_-\) and \(n\in\mathbb{Z}\) the
    screening operator \(\scr{-}\) defines the \emph{Felder complex}
    \begin{align*}
      \cdots\stackrel{\scrp{-}{s}}{\longrightarrow}
      F_{r,p_--s;n+1}\stackrel{\scrp{-}{p_--s}}{\longrightarrow}
      F_{r,s,n}\stackrel{\scrp{-}{s}}{\longrightarrow}F_{r,p_--s;n-1}
      \stackrel{\scrp{-}{p_--s}}{\longrightarrow}\cdots\,.
    \end{align*}
    This complex is exact for \(r=p_+\). For \(1\leq r< p_+\) it is exact
    everywhere except in \(F_{r,s;0}\), where the cohomology is given by
    \begin{align*}
      K_{r,s;0;-}/X_{r,s;0;-}\cong L(h_{r,s;0})\,.
    \end{align*}
  \end{enumerate}
\end{prop}

\begin{thm}\label{sec:singvects}
  Let \(1\leq r\leq p_+\), \(1\leq s \leq p_-\), all singular vectors of the
  Fock modules \(F_{r,s;n}\) can be expressed as the images of
  the screening operators \(\scr{+}\) and \(\scr{-}\).
  \begin{enumerate}
  \item  The singular vectors at
    conformal weights \(h_{r,p_--s;|n|+2k-1}\), \(k\geq0\) are given by
    \begin{align*}
      &\scrp{+}{(k+n+1)p_+-r}|\beta_{p_+-r,s;-1-2k-n}\rangle\in F_{r,s;n}
    \end{align*}
    for \(n\geq 0\) and
    \begin{align*}
      &\scrp{+}{(k+1)p_+-r}|\beta_{p_+-r,s;-1-2k+n}\rangle\in F_{r,s;n}
    \end{align*}
    for \(n\leq 0\).
  \item The singular vectors at
    conformal weights \(h_{r,p_--s;|n|+2k-1}\), \(k\geq0\) are given by
    \begin{align*}
      &\scrp{-}{(k+1)p_--s}|\beta_{r,p_--s;2k+n+1}\rangle\in F_{r,s;n}
    \end{align*}
    for \(n\geq 0\) and
    \begin{align*}
      &\scrp{-}{(k-n+1)p_--s}|\beta_{r,p_--s;2k-n+1}\rangle\in F_{r,s;n}
    \end{align*}
    for \(n\leq 0\).
  \end{enumerate}
\end{thm}
\begin{proof}
  The formulae for the singular vectors are precisely those of Proposition
  \ref{sec:fockspacesingvec} with the values of \(r,s,n\) chosen appropriately.
  The fact that these lists exhaust all singular vectors follows from the socle sequence
  decompositions in Proposition \ref{sec:socles}, that is, the generating
  singular vector of each simple module in the socle of \(F_{r,s;n}\) is constructed
  by the above formulae.
\end{proof}

\subsection{From {\boldmath \(U(\mathcal{L})\)\unboldmath}-mod to {\boldmath \(\mathcal{M}_{p_+,p_-}\)\unboldmath}-mod}

The purpose of this subsection is to define certain infinite direct sums of kernels and images of \(\scrp{+}{s},\ \scrp{-}{s}\) which will
later turn out to be modules of \(\mathcal{M}_{p_+,p_-}\).
\begin{definition}\label{sec:solitonweights}\ 
  \begin{enumerate}
  \item For \(1\leq r < p_+,\ 1\leq s < p_-\) and \(n\geq 0\),
    let \(\mathcal{K}_{r,s}^\pm\) and \(\mathcal{X}_{r,s}^\pm\) be the direct sums 
    \begin{align*}
      \mathcal{K}_{r,s}^+&=\ker \scrp{+}{r}|_{V_{r,s}^+}\cap\ker \scrp{-}{s}|_{V_{r,s}^+}=
      \bigoplus_{n\in\mathbb{Z}} K_{r,s;2n}\,,\\
      \mathcal{X}_{r,s}^+&=\im \scrp{+}{p_+-r}|_{V_{p_+-r,s}^+}
      \cap\im \scrp{-}{p_--s}|_{V_{r,p_--s}^+}=
      \bigoplus_{n\in\mathbb{Z}} X_{r,s,;2n}\\
      \mathcal{K}_{r,s}^-&=\ker \scrp{+}{r}|_{V_{r,s}^-}\cap\ker \scrp{-}{s}|_{V_{r,s}^-}=
      \bigoplus_{n\in\mathbb{Z}} K_{r,s;2n+1}\,,\\
      \mathcal{X}_{r,s}^-&=\im \scrp{+}{p_+-r}|_{V_{p_+-r,s}^-}
      \cap\im \scrp{-}{p_--s}|_{V_{r,p_--s}^-}=
      \bigoplus_{n\in\mathbb{Z}} X_{r,s,;2n+1}
    \end{align*}
    \begin{align*}
      \mathcal{K}_{r,p_-}^+&=\mathcal{X}_{r,p_-}^+=\ker \scrp{+}{r}|_{V_{r,p_-}^+}=
      \bigoplus_{n\in\mathbb{Z}} K_{r,p_-;2n}\,,&
      \mathcal{K}_{r,p_-}^-&=\mathcal{X}_{r,p_-}^-=\ker \scrp{+}{r}|_{V_{r,p_-}^-}=
      \bigoplus_{n\in\mathbb{Z}} K_{r,p_-;2n+1}\,,      \\
      \mathcal{K}_{p_+,s}^+&=\mathcal{X}_{p_+,s}^+=\ker \scrp{-}{s}|_{V_{p_+,s}^+}=
      \bigoplus_{n\in\mathbb{Z}} K_{p_+,s;2n}\,,&
      \mathcal{K}_{p_+,s}^-&=\mathcal{X}_{p_+,s}^-=\ker \scrp{-}{s}|_{V_{p_+,s}^-}=
      \bigoplus_{n\in\mathbb{Z}} K_{p_+,s;2n+1}\,,
    \end{align*}
    \begin{align*}
      \mathcal{K}_{p_+,p_-}^+&=\mathcal{X}_{p_+,p_-}^+=
      \bigoplus_{n\in\mathbb{Z}} F_{p_+,p_-;2n}\,,&
      \mathcal{K}_{p_+,p_-}^-&=\mathcal{X}_{p_+,p_-}^-=
      \bigoplus_{n\in\mathbb{Z}} F_{p_+,p_-;2n+1}\,.
    \end{align*}
  \item For \(1\leq r\leq p_+,\ 1\leq s\leq p_-\) and \(n\geq 0\), 
    let \(V_{r,s:n}^\pm\) be the spaces of singular vectors of conformal
    weight \(\Delta_{r,s;n}^\pm\)
    \begin{align*}
      V_{r,s:n}^\pm=\{u\in \mathcal{X}_{r,s}^\pm| L_0u=\Delta_{r,s;n}^\pm u, L_nu=0, n\geq1\}\,.
    \end{align*}
    We call these spaces \emph{soliton sectors}.
    The elements of \(V_{r,s:n}^\pm\) are called \emph{soliton vectors}.
  \end{enumerate}
\end{definition}
Note that by Theorem \ref{sec:singvects} it is easy to give explicit bases of the
soliton sector  \(V_{r,s;n}^\pm\) in terms of screening operators. 
\begin{trivlist}
\item Basis in terms of \(\scr{+}\):
    \begin{align*}
      V_{r,s:n}^+&=\left\{
        \begin{array}{lc}
          \bigoplus_{m=-n}^{n}\mathbb{C}
          \scrp{+}{(n+m)p_+}|\beta_{p_+,p_-;-2n}\rangle& r=p_+,s=p_-\\
          \bigoplus_{m=-n}^{n}\mathbb{C}
          \scrp{+}{(n+m)p_+}|\beta_{p_+,s;-2n}\rangle& r=p_+,s\neq p_-\\
          \bigoplus_{m=-n}^{n}\mathbb{C}
          \scrp{+}{(m+n+1)p_+-r}|\beta_{p_+-r,p_-;-2n-1}\rangle& r\neq p_+,s=p_-\\
          \bigoplus_{m=-n}^{n}\mathbb{C}
          \scrp{+}{(m+n+1)(p_+-r)}|\beta_{p_+-r,s;-2n-1}\rangle& r\neq p_+,s\neq p_-
        \end{array}\right.\\
      V_{r,s:n}^-&=\left\{
        \begin{array}{lc}
          \bigoplus_{m=-n}^{n+1}\mathbb{C}
          \scrp{+}{(n+m)p_+}|\beta_{p_+,p_-;-2n-1}\rangle& r=p_+,s=p_-\\
          \bigoplus_{m=-n}^{n+1}\mathbb{C}
          \scrp{+}{(n+m)p_+}|\beta_{p_+,s;-2n-1}\rangle& r=p_+,s\neq p_-\\
          \bigoplus_{m=-n}^{n+1}\mathbb{C}
          \scrp{+}{(m+n+1)p_+-r}|\beta_{p_+-r,p_-;-2n-2}\rangle& r\neq p_+,s=p_-\\
          \bigoplus_{m=-n}^{n+1}\mathbb{C}
          \scrp{+}{(m+n+1)p_+-r}|\beta_{p_+-r,s;-2n-2}\rangle& r\neq p_+,s\neq p_-
        \end{array}\right.
    \end{align*}
  \item Basis in terms of \(\scr{-}\):
    \begin{align*}
      V_{r,s:n}^+&=\left\{
        \begin{array}{lc}
          \bigoplus_{m=-n}^{n}\mathbb{C}
          \scrp{-}{(n-m)p_-}|\beta_{p_+,p_-;2n}\rangle& r=p_+,s=p_-\\
          \bigoplus_{m=-n}^{n}\mathbb{C}
          \scrp{-}{(n-m+1)p_--s}|\beta_{p_+,p_--s;2n+1}\rangle& r=p_+,s\neq p_-\\
          \bigoplus_{m=-n}^{n}\mathbb{C}
          \scrp{-}{(n-m)p_-}|\beta_{r,p_-;2n}\rangle& r\neq p_+,s=p_-\\
          \bigoplus_{m=-n}^{n}\mathbb{C}
          \scrp{-}{(n-m+1)p_--s}|\beta_{r,p_--s;2n+1}\rangle& r\neq p_+,s\neq p_-
        \end{array}\right.\\
      V_{r,s:n}^-&=\left\{
        \begin{array}{lc}
          \bigoplus_{m=-n}^{n+1}\mathbb{C}
          \scrp{-}{(n-m+1)p_-}|\beta_{p_+,p_-;2n+1}\rangle& r=p_+,s=p_-\\
          \bigoplus_{m=-n}^{n+1}\mathbb{C}
          \scrp{-}{(n-m+2)p_--s}|\beta_{p_+,p_--s;2n+2}\rangle& r=p_+,s\neq p_-\\
          \bigoplus_{m=-n}^{n+1}\mathbb{C}
          \scrp{(n-m+1)p_-}|\beta_{r,p_-;2n+1}\rangle& r\neq p_+,s=p_-\\
          \bigoplus_{m=-n}^{n+1}\mathbb{C}
          \scrp{-}{(n-m+2)p_--s}|\beta_{r,p_--s;2n+2}\rangle& r\neq p_+,s\neq p_-
        \end{array}\right.
    \end{align*}
 \end{trivlist}
\begin{prop}\label{sec:Mppirreds}
  As Virasoro modules the spaces \(\mathcal{K}^\pm_{r,s}\) and \(\mathcal{X}^\pm_{r,s}\) decompose as
  \begin{enumerate}
  \item For \(1\leq r< p_+\) and \(1\leq s<p_-\)
    \begin{align*}
      \mathcal{K}^+_{r,s}&=U(\mathcal{L})|\beta_{r,s;0}\rangle\oplus\bigoplus_{n\geq 1}
      L(\Delta_{r,s:n}^+)\otimes V_{r,s;n}^+\\
      \mathcal{X}^+_{r,s}&=\bigoplus_{n\geq 0}
      L(\Delta_{r,s:n}^+)\otimes V_{r,s;n}^+\\
      \mathcal{K}^-_{r,s}&=\mathcal{X}^-_{r,s}=\bigoplus_{n\geq 0}
      L(\Delta_{r,s:n}^-)\otimes V_{r,s;n}^-
    \end{align*}
  \item For \(r=p_+\) and \(1\leq s<p_-\)
    \begin{align*}
      \mathcal{K}^+_{p_+,s}&=\mathcal{X}^+_{p_+,s}=\bigoplus_{n\geq 0}
      L(\Delta_{p_+,s:n}^+)\otimes V_{p_+,s;n}^+\\
      \mathcal{K}^-_{p_+,s}&=\mathcal{X}^-_{p_+,s}=\bigoplus_{n\geq 0}
      L(\Delta_{p_+,s:n}^-)\otimes V_{p_+,s;n}^-
    \end{align*}
  \item For \(1\leq r<p_+\) and \(s=p_-\)
    \begin{align*}
      \mathcal{K}^+_{r,p_-}&=\mathcal{X}^+_{r,p_-}=\bigoplus_{n\geq 0}
      L(\Delta_{r,p_-:n}^+)\otimes V_{r,p_-;n}^+\\
      \mathcal{K}^-_{r,p_-}&=\mathcal{X}^-_{r,p_-}=\bigoplus_{n\geq 0}
      L(\Delta_{r,p_-:n}^-)\otimes V_{r,p_-;n}^-
    \end{align*}
  \item For \(r=p_+\) and \(s=p_-\)
    \begin{align*}
      \mathcal{K}^+_{p_+,p_-}&=\mathcal{X}^+_{p_+,p_-}=\bigoplus_{n\geq 0}
      L(\Delta_{p_+,p_-:n}^+)\otimes V_{p_+,p_-;n}^+\\
      \mathcal{K}^-_{p_+,p_-}&=\mathcal{X}^-_{p_+,p_-}=\bigoplus_{n\geq 0}
      L(\Delta_{p_+,p_-:n}^-)\otimes V_{p_+,p_-;n}^-
    \end{align*}
  \end{enumerate}
\end{prop}
\begin{proof}
  The above proposition follows by plugging Definition \ref{sec:kersandims}
  and Proposition \ref{sec:ksocs}
  into Definition \ref{sec:solitonweights}.
\end{proof}

\begin{prop}
  For \(1\leq r< p_+\), \(1\leq s<p_-\), the \(\mathcal{K}^+_{r,s}\) satisfy the
  following exact sequence as Virasoro modules
  \begin{align*}
    0\longrightarrow \mathcal{X}^+_{r,s}\longrightarrow \mathcal{K}^+_{r,s}
    \longrightarrow L(h_{r,s;0})\longrightarrow 0\,.
  \end{align*}
\end{prop}

\subsection{Frobenius homomorphisms}\label{sec:frobhomsec}

In this section we introduce a class of Virasoro homomorphisms
\begin{align*}
  E,F\in \ehom_{\mathbb{C}}(\mathcal{K}_{r,s}^\pm),\ 1\leq r<p_+, 1\leq s<p_-
\end{align*}
such that \(E\) and \(F\) define derivations of the VOA
\(\mathcal{M}_{p_+,p_-}\).
We call \(E\) and \(F\)
\emph{Frobenius homomorphisms}. 
These Frobenius homomorphisms will prove to be essential tools for analysing the
VOA structure of \(\mathcal{M}_{p_+,p_-}\) and for analysing the action of
\(\mathcal{M}_{p_+,p_-}\) on the \(\mathcal{K}_{r,s}^\pm\). In
\cite{Adamovic:2010,Adamovic:2011}, for \(p_+=2\) a similar derivation was
constructed as the zero mode of a field that is non-local with the fields of
\(\mathcal{V}_{2,p_-}\). The Frobenius homomorphisms \(E\) and \(F\) are not
constructed as the zero modes of anything, though as we shall see, they are
not well defined on the Fock modules \(F_{r,s}\) but only on the kernels of
\(\scr{+}\) and \(\scr{-}\). So it appears that some of the non-locality
encountered in \cite{Adamovic:2010,Adamovic:2011} is still there, but it
manifests in a different way.

Recall that by the Felder complexes the maps
\begin{align*}
  \scrp{+}{p_+-r}\circ \scrp{+}{r}:F_{r,s;n}&\rightarrow
  F_{r,s;n+2}\,,&\scrp{-}{p_--s}\circ \scrp{-}{s}:F_{r,s;n}\rightarrow F_{r,s;n-2}
\end{align*}
are zero. The Frobenius homomorphisms are regularisations of these maps, such
that they become non-trivial.

We give the details of the construction for \(\scr{+}\), the \(\scr{-}\) case follows
in the same way. At first we fix \(1\leq r\leq p_+\) and \(s\in\mathbb{Z}\) and consider
the action of \(\scr{+}(z_1)\cdots \scr{+}(z_{p_+})\) on \(\tensor[_{\mathcal{K}}]{F}{_{r,s}}\)
\begin{align*}
  \scr{+}(z_1)\cdots \scr{+}(z_{p_+})&=e^{p_+\alpha_+(\epsilon)\hat b}U_{p_+}(z;\kappa_+(\epsilon),r)
  \prod_{i=1}^{p_+}z_i^{s-1}:\prod_{i=1}^{p_+}\overline{\scr{+}(z_i)}:\\
  U_{p_+}(z;\kappa_+(\epsilon),r)&=\prod_{1\leq i\neq j\leq p_+}(z_i-z_j)^{\kappa_+(\epsilon)}
  \prod_{i=1}^{p_+}z_i^{(1-r)\kappa_+(\epsilon)}\,.
\end{align*}
Let \(\tensor*[_{\mathcal{K}}]{\mathcal{L}}{_{p_+}}(\kappa_+(\epsilon),r)\) be
the local system on \(X_{p_+}\)
over \(\mathcal{K}\) defined by the multivaluedness of
\(U_{p_+}(z;\kappa_+(\epsilon),r)\) and let
\(\tensor*[_{\mathcal{K}}]{\mathcal{L}}{_{p_+}^\vee}(\kappa_+(\epsilon),r)\)
be its dual. For \(r=p_+\) we can integrate the above screening operators over the cycle
\([\Gamma_{p_+}(\kappa_+)]\) to obtain \(\scrp{+}{p_+}\), however, for \(r<p_+\)
the homology group
\(H_{p_+}(X_{p_+},\tensor*[_{\mathcal{K}}]{\mathcal{L}}{_{p_+}^\vee}(\kappa_+(\epsilon),r))\)
is trivial and no such cycle exists. We remedy this problem by changing the
domain of the \(z_i\)  to
\begin{align*}
  Y_{p_+}=\{(z_1,\dots,z_{p_+})\in\mathbb{C}^{p_+}|z_i\neq z_j, z_i\neq 0,1\}\subset X_{p_+}\,,
\end{align*}
then it is known that 
\cite{Varchenko:2003,AomotoKita:2011}
\begin{align*}
  \dim_{\mathcal{K}}
  H^{p_+}(Y_{p_+},\tensor*[_{\mathcal{K}}]{\mathcal{L}}{_{p_+}}(\kappa_+(\epsilon),r))&=p_+!\,,\\
  \dim_{\mathcal{K}}
  H^{p_+}(Y_{p_+},\tensor*[_{\mathcal{K}}]{\mathcal{L}}{_{p_+}}(\kappa_+(\epsilon),r))^{\symg{p_+}-}&=1\,,\\
  \dim_{\mathcal{K}}
  H_{p_+}(Y_{p_+},\tensor*[_{\mathcal{K}}]{\mathcal{L}}{_{p_+}^\vee}(\kappa_+(\epsilon),r))&=p_+!\,.
\end{align*}
We introduce new variable names  \(w_i=z_{i+p_+-r},\ i=1,\dots, r\) and 
consider the open domain \(U_1^{p_+-r}\times U_2^r\subset Y_{p_+}\)
\begin{align*}
  U_1^{p_+-r}&=\{(z_1,\dots,z_{p_+-r})\in\mathbb{C}^{p_+-r}
  | z_i\neq z_j, |z_i|> 1\}\,,\\
  U_2^r&=\{(w_1,\dots,w_r)\in\mathbb{C}^{r}|w_i\neq w_j, 1>|w_i|>0\}\,,
\end{align*}
and define the multivalued holomorphic functions
\begin{align*}
  G_{p_+-r}(z;\kappa_+(\epsilon))&=\prod_{1\leq i\neq j\leq
    p_+-r}(z_i-z_j)^{\kappa_+(\epsilon)}\prod_{i=1}^{p_+-r}z_i^{(1+r)\kappa_+(\epsilon)}\,,\\
  U_{r}(w;\kappa_+(\epsilon))&=\prod_{1\leq i\neq j\leq
    r}(w_i-w_j)^{\kappa_+(\epsilon)}\prod_{i=1}^{r}w_i^{(1-r)\kappa_+(\epsilon)}\,.
\end{align*}
on \(U_1^{p_+-r}\) and \(U_2^r\) respectively. Then on \(U_1^{p_+-r}\times
U_2^r\) the product \(\scr{+}(z_1)\cdots \scr{+}(z_{p_-r})\scr{+}(w_1)\cdots \scr{+}(w_{r})\)
factorises as
\begin{align*}
  \prod_{i=1}^{p_+-r}\scr{+}(z_i) \prod_{i=1}^{r}\scr{+}(w_i)&=
  e^{p_+\alpha_+(\epsilon)\hat{b}}G_{p_+-r}(z;\kappa_+(\epsilon))\prod_{i=1}^{p_+-r}z_i^{s-1}:\prod_{i=1}^{p_+-r}\overline{\scr{+}}(z_i):\\
  &\ \ \times
  U_{r}(w;\kappa_+(\epsilon))\prod_{i=1}^{r}w_i^{s-1}:\prod_{i=1}^{r}
  \overline{\scr{+}}(w_i):\,.
\end{align*}
Next we consider the local system
\(\tensor*[_{\mathcal{K}}]{\mathcal{G}}{_{p_+-r}}(\kappa_+(\epsilon))\)
on \(U_1^{p_+-r}\) defined by \(G_{p_+-r}(z;\kappa_+(\epsilon))\) and the local
system \(\tensor*[_{\mathcal{K}}]{\mathcal{L}}{_{r}}(\kappa_+(\epsilon))\) on
\(U_2^r\) defined by \(U_{r}(w;\kappa_+(\epsilon))\). We define regularised cycles
\begin{align*}
  [\Gamma_{p_+-r}^\infty(\kappa_+(\epsilon))]&\in
  H_{p_+-r}^{\operatorname{l.f}}(U_1^{p_+-r},\tensor*[_{\mathcal{K}}]{\mathcal{G}}{_{p_+-r}^\vee}(\kappa_+(\epsilon)))\\
  [\Gamma_{r}(\kappa_+(\epsilon))]&\in
  H_{r}^{\operatorname{l.f}}(U_2^{r},\tensor*[_{\mathcal{K}}]{\mathcal{L}}{_{r}^\vee}(\kappa_+(\epsilon)))\,,
\end{align*}
where the regularised cycles \([\Gamma_r(\kappa_+(\epsilon))]\) are those of
Definition \ref{sec:cycledef}. We define
\([\Gamma_{p_+-r}^\infty(\kappa_+(\epsilon))]\) as
\begin{align*}
  [\Gamma_{p_+-r}^\infty(\kappa_+(\epsilon))]=\frac{1}{c_r^\infty(\kappa_+(\epsilon))}[\Delta_{p_+-r}^\infty\otimes\phi]\,,
\end{align*}
where \(\Delta_{p_+-r}^\infty=\{\infty>z_1>\cdots>z_{p_+-r}>1\}\), \(\phi\)
is the principal branch of \(G_{p_+-r}(z;\kappa_+(\epsilon))\) and
\begin{align*}
  c_r^\infty(\kappa_+(\epsilon))&=\epsilon\int_{[\Delta_{p_+-r}^\infty\otimes\phi]}
  \int_{[\overline{\Gamma}_{r}(\kappa_+(\epsilon))]}
  \prod_{1\leq i\neq j\leq p_+-r}(z_i-z_j)^{\kappa_+(\epsilon)}\prod_{i=1}^{p_+-r}z_i^{2\kappa_+(\epsilon)}
  \prod_{i=1}^{p_+-r}\prod_{j=1}^{r-1}(z_i-y_j)^{2\kappa_+(\epsilon)}\\
  &\qquad \times\prod_{1\leq i\neq j\leq r-1}(y_i-y_j)^{\kappa_+(\epsilon)}\prod_{i=1}^{r-1}y_i^{2\kappa_+(\epsilon)}
  \frac{\d y_1\cdots \d y_{r-1}}{y_1\cdots y_{r-1}}\frac{\d z_1\cdots \d z_{p_+-r}}{z_1\cdots z_{p_+-r}}\,,
\end{align*}
where \([\overline{\Gamma}_r(\kappa_+(\epsilon))]\) is the regularised cycle of
Definition \ref{sec:cycledef}.
By the change of variables \(u_i=1/z_i\) one can express
\(c_r^\infty(\kappa_+(\epsilon))\) as 
\begin{align*}
  c_r^\infty(\kappa_+(\epsilon))&=
  \epsilon\int_{[\Delta_{p_+-r}^\infty\otimes\phi]}
  \int_{[\overline{\Gamma}_{r}(\kappa_+(\epsilon))]}
  \prod_{1\leq i\neq j\leq p_+-r}(u_i-u_j)^{\kappa_+(\epsilon)}\prod_{i=1}^{p_+-r}u_i^{(1-2p_++r)\kappa_+(\epsilon)}
  \prod_{i=1}^{p_+-r}\prod_{j=1}^{r-1}(1-u_iy_j)^{2\kappa_+(\epsilon)}\\
  &\qquad \times\prod_{1\leq i\neq j\leq r-1}(y_i-y_j)^{\kappa_+(\epsilon)}\prod_{i=1}^{r-1}y_i^{2\kappa_+(\epsilon)}
  \frac{\d y_1\cdots \d y_{r-1}}{y_1\cdots y_{r-1}}\frac{\d u_1\cdots \d u_{p_+-r}}{u_1\cdots u_{p_+-r}}\,.
\end{align*}
The product \(\prod_{i=1}^{p_+-r}\prod_{j=1}^{r-1}(1-u_iy_j)^{2\kappa_+(\epsilon)}\) can be expanded as a sum of symmetric 
polynomials in the \(u_i\) times symmetric polynomials in the \(y_i\) and we can therefore use Kadell's integrals of
Proposition \pageref{sec:Kadellintegral} to evaluate
\(c_r^\infty(\kappa_+(\epsilon))\).
\begin{align*}
  I_{\lambda,p_+-r}((1-2p_++r)\kappa_+(\epsilon)+k,0,\kappa_+(\epsilon))&\in\left\{
    \begin{array}{cc}
      \epsilon^{-1}\mathcal{O}^\times&k=p_-, \lambda=0\\
      \mathcal{O}&\text{else}
    \end{array}\right.\,,\\
  I_{\mu,r-1}(2\kappa_+(\epsilon)+k,0,\kappa_+(\epsilon))&\in \mathcal{O}\,,
\end{align*}
for \(k\geq0\) and \(\lambda\) and \(\mu\) partitions of length at most \(p_+-r-1\) and \(r-2\) respectively. It therefore follows that
\(c_r^\infty(\kappa_+(\epsilon))\in\mathcal{O}^\times\).
We define the cycle \([\Gamma_{p_+,r}(\kappa_+(\epsilon))]\in
H_{p_+}(Y_{p_+},\tensor*[_{\mathcal{K}}]{\mathcal{L}}{_{p_+}^\vee}
(\kappa_+(\epsilon),r))\) as the image of
\begin{align*}
  [\Gamma_{p_+-r}^\infty(\kappa_+(\epsilon))]\otimes[\Gamma_r(\kappa_+(\epsilon))]\in 
  H_{p_+-r}^{\operatorname{l.f}}(U_1^{p_+-r},\tensor*[_{\mathcal{K}}]{\mathcal{G}}{_{p_+-r}^\vee}(\kappa_+(\epsilon)))
  \otimes
  H_{r}^{\operatorname{l.f}}(U_2^{r},\tensor*[_{\mathcal{K}}]{\mathcal{L}}{_{r}^\vee}(\kappa_+(\epsilon)))
\end{align*}
under the map
\begin{align*}
    H_{p_+-r}^{\operatorname{l.f}}(U_1^{p_+-r},\tensor*[_{\mathcal{K}}]{\mathcal{G}}{_{p_+-r}^\vee}(\kappa_+(\epsilon)))
  &\otimes
  H_{r}^{\operatorname{l.f}}(U_2^{r},\tensor*[_{\mathcal{K}}]{\mathcal{L}}{_{r}^\vee}(\kappa_+(\epsilon)))\\
  &\rightarrow
  H_{p_+}^{\operatorname{l.f}}(Y_{p_+},\tensor*[_{\mathcal{K}}]{\mathcal{L}}{_{p_+}^\vee}(\kappa_+(\epsilon),r))
  \cong H_{p_+}(Y_{p_+},\tensor*[_{\mathcal{K}}]{\mathcal{L}}{_{p_+}^\vee}(\kappa_+(\epsilon),r))\,.
\end{align*}
\begin{definition}\label{sec:frobopdef}\ 
\begin{enumerate}
  \item  For \(1\leq r\leq p_+\), \(s\in\mathbb{Z}\), the \emph{Frobenius
      operator} \(E\) associated to \(\scr{+}\) is the map
    \(E:\tensor[_{\mathcal{K}}]{F}{_{r,s}}\rightarrow
    \tensor[_{\mathcal{K}}]{F}{_{r-2p_+,s}}\), where
    \begin{align*}
      E&= \int_{[\Gamma_{p_+,r}(\kappa_+(\epsilon))]}
      \prod_{i=1}^{p_+-r}\scr{+}(z_i) \prod_{i=1}^{r}\scr{+}(w_i)
      \prod_{i=1}^{p_+-r}\d z_i \prod_{i=1}^{r} \d w_i\,.
    \end{align*}
  \item 
    For \(1\leq s\leq p_-\), \(r\in\mathbb{Z}\), the \emph{Frobenius
      operator} \(F\) associated to \(\scr{-}\) is the map
    \(F:\tensor[_{\mathcal{K}}]{F}{_{r,s}}\rightarrow
    \tensor[_{\mathcal{K}}]{F}{_{r,s-2p_-}}\), where
    \begin{align*}
      F&= \int_{[\Gamma_{p_-,s}(\kappa_-(\epsilon))]}
      \prod_{i=1}^{p_--s}\scr{-}(z_i) \prod_{i=1}^{s}\scr{-}(w_i)
      \prod_{i=1}^{p_--s}\d z_i \prod_{i=1}^{s} \d w_i\,.
    \end{align*}
  \end{enumerate}
\end{definition}

\begin{thm}\label{sec:frobopthm}
  For \(1\leq r\leq p_+\), \(1\leq s\leq p_-\) and \(n\in\mathbb{Z}\)
  the Frobenius homomorphisms \(E\) and \(F\) induce well defined maps over
  \(\mathbb{C}\) 
  \begin{align*}
    E:K_{r,s;n;+}&\rightarrow K_{r,s;n+2;+}&F:K_{r,s;n;-}&\rightarrow K_{r,s;n-2;-}
  \end{align*}
  that satisfy the following properties:
  \begin{enumerate}
  \item The maps \(E\) and \(F\) are Virasoro homomorphisms.
  \item The maps \(E\) and \(F\) act transitively on the soliton sectors, that
    is, for \(m\geq0\), \(l\geq 1\)
    \begin{align*}
      E \scrp{+}{(m+1)p_+-r}|\beta_{p_+-r,s;-m-1-l}\rangle
      &=\operatorname{const}\scrp{+}{(m+2)p_+-r}|\beta_{p_+-r,s;-m-1-l}\rangle\\
      F \scrp{-}{(m+1)p_--s}|\beta_{r,p_--s;m+1+l}\rangle
      &=\operatorname{const}\scrp{-}{(m+2)p_--s}|\beta_{r,p_--s;m+1+l}\rangle\,,
    \end{align*}
    where the constants are non-zero complex numbers.
  \item The maps \(E\) and \(F\) are derivations, that is, 
    for \(A\in K_{1,1;2n;+}\), \(B\in K_{r,s;m;+}\) and \(n,m\in\mathbb{Z}\)
    \begin{align*}
      EY(A;u)B= Y(EA;u)B+Y(A;u)EB\,,
    \end{align*}
    while for \(C\in K_{1,1;2n;-}\), \(D\in K_{r,s;m;-}\)
    \begin{align*}
      FY(C;u)D= Y(FC;u)D+Y(C;u)FD\,.
    \end{align*}
  \end{enumerate}
\end{thm}
\begin{proof}
  We prove the Theorem for \(E\), since \(F\) follows in the same way. 
  We already know from our previous calculations regarding the screening
  operator \(\scrp{+}{r}\), that integrating the \(w_i\) coordinates over the cycle
  \([\Gamma_{r}(\kappa_+(\epsilon))]\)
  is well defined over \(\mathbb{C}\) after setting \(\epsilon=0\).
  If we evaluate the integral formula,
  \begin{align*}
    E&= \int_{[\Gamma_{p_+,r}(\kappa_+(\epsilon))]}
    \prod_{i=1}^{p_+-r}\scr{+}(z_i) \prod_{i=1}^{r}\scr{+}(w_i)
    \prod_{i=1}^{p_+-r}\d z_i \prod_{i=1}^{r} \d w_i\,,
  \end{align*}
  which defines \(E\), on \(\tensor[_{\mathcal{K}}]{F}{_{r,s}}\) and integrate
  the \(w_i\) coordinates over the cycle \([\Gamma_{r}(\kappa_+(\epsilon))]\)
  then what is left over is sums of integrals of the form
  \begin{align*}
    \int_{[\Gamma^\infty_{p_+-r}(\kappa_+(\epsilon))]}\prod_{1\leq i\neq j\leq p_+-r}(z_i-z_j)^{\kappa_+(\epsilon)}
    \prod_{i=1}^{p_+-r}z_{i}^{(1+r)\kappa_+(\epsilon)+k}f(z)\frac{\d z_1\cdots
      \d z_{p_+-r}}{z_1\cdots z_{p_+-r}}\,,
  \end{align*}
  for \(k\in\mathbb{Z}\) and \(f(z)\in
  \mathcal{O}[z_1^\pm,\dots,x_{p_+-r}^\pm]^{\symg{p_+-r}}\). Due to the
  identities for Jack Polynomials introduced in the third part of the proof of
  Proposition \ref{sec:screeningfields}, we can without loss of generality
  consider \(f(z)=Q_\lambda(z^{-1};\kappa_-(\epsilon))\) to be a dual Jack polynomial in
  \(z_i^{-1}\) with a partition \(\lambda\) of length at most \(p_+-r-1\).
  If we perform a change of variables \(u_i=1/z_i\), then we can evaluate this
  expression using the Kadell integral formulae in Proposition \ref{sec:Kadellintegral}
  \begin{align*}
    &\int_{[\Gamma^\infty_{p_+-r}(\kappa_+(\epsilon))]}\prod_{1\leq i\neq j\leq p_+-r}(z_i-z_j)^{\kappa_+(\epsilon)}
    \prod_{i=1}^{p_+-r}z_{i}^{(1 +r)\kappa_+(\epsilon)+k}Q_\lambda(z^{-1};\kappa_-(\epsilon))\frac{\d z_1\cdots
      \d z_{p_+-r}}{z_1\cdots z_{p_+-r}}\\
    &\ =\frac{1}{c_r^\infty(\kappa_+(\epsilon))}
    \int_{[\Delta_{p_+-r}^\infty\otimes\phi]}\prod_{1\leq i\neq j\leq p_+-r}(u_i-u_j)^{\kappa_+(\epsilon)}
    \prod_{i=1}^{p_+-r}u_{i}^{(1 -2p_+ +r)\kappa_+(\epsilon)-k}Q_\lambda(u;\kappa_-(\epsilon))\frac{\d u_1\cdots
      \d u_{p_+-r}}{u_1\cdots u_{p_+-r}}\\
    &\ =
    \frac{I_{\lambda,p_+-r}((1-2p_++r)\kappa_+(\epsilon)-k,0,\kappa_+(\epsilon))}{c_r^{\infty}(\kappa_+(\epsilon))}
        \in\left\{
      \begin{array}{cl}
        \epsilon^{-1}\mathcal{O}^\times&k=-p_- \text{ and } \lambda=0\\
        \mathcal{O}&\operatorname{else}
      \end{array}\right.
  \end{align*}
  It thus
  follows that \(E:\tensor[_{\mathcal{O}}]{F}{_{r,s}}\rightarrow
  \epsilon^{-1}\tensor[_{\mathcal{O}}]{F}{_{r-2p_+,s}}\), which is why the action of \(E\) is not well defined
  on the Fock modules \(F_{r,s;n}\) when one sets \(\epsilon=0\).
  Furthermore, if \(A\in K_{r,s;n;+}\) then for any lift \(\tilde A\in
  \tensor[_{\mathcal{O}}]{F}{_{r,s}}\), we have
  \begin{align*}
    \int_{[\Gamma_r(\kappa_+(\epsilon))]}
    \prod_{i=1}^{r}\scr{+}(w_i)\tilde A
    \d w_1\cdots\d w_r\in\epsilon\tensor[_{\mathcal{O}}]{F}{_{-r,s}}
  \end{align*}
  Therefore it
  follows that \(E\) induces a well defined map \(E:K_{r,s;n;+}\rightarrow
  F_{r,s;n+2}\) over \(\mathbb{C}\).

  Next we show that \(E\) commutes with the Virasoro algebra. 
  Recall that
  \begin{align*}
    [L_m,\scr{+}(z_i)]=z_i^{m}(z_i\partial_{z_i}+n+1)\scr{+}(z_i)=\partial_{z_i}
    z_i^{m+1}Q_+(z_i)\,.
  \end{align*}
  If we consider the action of Virasoro generator \(L_m\) on integral expressions of
  the form
  \begin{align*}
    \int_{[\Gamma^\infty_{p_+-r}(\kappa_+(\epsilon))]}\prod_{1\leq i\neq j\leq p_+-r}(z_i-z_j)^{\kappa_+(\epsilon)}
    \prod_{i=1}^{p_+-r}z_{i}^{(1+r)\kappa_+(\epsilon)+k}f(z)\frac{\d z_1\cdots
      \d z_{p_+-r}}{z_1\cdots z_r}\,,
  \end{align*}
  then the action of \(L_m\) can be written as a total derivative
  \begin{align*}
    &\int_{[\Gamma^\infty_{p_+-r}(\kappa_+(\epsilon))]}(\sum_{l=1}^{p_+-r}\partial_{z_l}z_l^{m+1}
    )\hspace{-1.5em}\prod_{1\leq
      i\neq j\leq p_+-r}\hspace{-1.5em}(z_i-z_j)^{\kappa_+(\epsilon)}
    \prod_{i=1}^{p_+-r}z_{i}^{(1+r)\kappa_+(\epsilon)+k}
    Q_\lambda(z^{-1};\kappa_-(\epsilon))\frac{\d z_1\cdots
      \d z_{p_+-r}}{z_1\cdots z_{p_+-r}}\\
    &\ =
    \int_{[\Gamma^\infty_{p_+-r}(\kappa_+(\epsilon))]}\hspace{-1.2em}\d \left(\sum_{l=1}^{p_+-r}(-1)^{l+1}z_l^{m}
      \hspace{-1.5em}\prod_{1\leq i\neq j\leq p_+-r}\hspace{-1.5em}(z_i-z_j)^{\kappa_+(\epsilon)}
    \prod_{i=1}^{p_+-r}z_{i}^{(1+r)\kappa_+(\epsilon)+k}
    Q_\lambda(z^{-1};\kappa_-(\epsilon))\frac{\d z_1\cdots \widehat{\d z_l}\cdots
      \d z_{p_+-r}}{z_1\cdots \widehat{z_l} \cdots z_{p_+-r}}\right)\\
  &\ =
  \int_{[\Gamma^\infty_{p_+-r}(\kappa_+(\epsilon))]}\hspace{-1.5em}G_{p_+-r}(z;\kappa_+(\epsilon))\nabla_G\left(\sum_{l=1}^{p_+-r}(-1)^{l+1}z_l^{m}
    \prod_{i=1}^{p_+-r}z_{i}^{k}
    Q_\lambda(z^{-1};\kappa_-(\epsilon))\frac{\d z_1\cdots \widehat{\d z_l}\cdots
      \d z_{p_+-r}}{z_1\cdots \widehat{z_l} \cdots z_{p_+-r}}\right)=0\,,
  \end{align*}
  where \(\d\) denotes the exterior derivative with respect to the \(z_i\)
  coordinates, \(\hat{\ }\) denotes omission of a variable 
  and \(\nabla_G\) is the differential twisted by \(G_{p_+-r}(z;\kappa_+(\epsilon))\)
  \begin{align*}
    \nabla_G=\d +(\d \log G_{p_+-r}(z;\kappa_+(\epsilon)))\wedge\,.
  \end{align*}
  It therefore follows that \(E\) commutes with the Virasoro algebra, that is,
  \(E\) is a Virasoro homomorphism.
  Since \(E\) is a Virasoro homomorphism, it follows
  that \(\im E(K_{r,s;n;+})\subset K_{r,s;n+2;+}\) from the socle sequence
  decompositions of Proposition \ref{sec:ksocs}.

  Next we prove that \(E\) acts transitively on soliton vectors. In order to
  show that
  \begin{align*}
    E \scrp{+}{(m+1)p_+-r}|\beta_{p_+-r,s;-m-1-l}\rangle
      &=\operatorname{const}\scrp{+}{(m+2)p_+-r}|\beta_{p_+-r,s;-m-1-l}\rangle
  \end{align*}
  it is sufficient to show that \(E \scrp{+}{(m+1)p_+-r}|\beta_{p_+-r,s;-m-1-l}\rangle
  \neq0\), since singular vectors of a given conformal weight are unique
  up to normalisation
  and, due to \(E\) being a Virasoro homomorphism, \(E\) maps singular
  vectors to singular vectors. 
  We show that \(E\) acts non-trivially by showing that the following matrix
  element is non-trivial
  \begin{align*}
    \langle \beta_{r,s+k_1-k_2;m-l+2}|\prod_{i=1}^{k_1}V_{\alpha_-/2}(v_i)V_{-k_2\alpha_-/2}(u)E\scrp{+}{(m+1)p_+-r}|\beta_{p_+-r,s;-m-1-l}\rangle\,,
  \end{align*}
  \(k_1=(m+1)p_-\), \(k_2=(l+1)p_--s\) and we assume that \(v_i,u\) lie on the
  positive imaginary axis, satisfying \(|v_{k_1}|>\cdots >|v_1|>|u|>1\).
  We will show that the matrix element is non-trivial by writing it as the
  product of a non-trivial multivalued function times a Laurent series in the
  \(u,v,w\) and \(z\) variables with at
  least one non-vanishing coefficient.
  A lift of this matrix element to \(\mathcal{O}\) is given by
  \begin{align*}
    M(u,v)=\int_{[\Gamma_{p_+,r}(\kappa_+(\epsilon))]}\hspace{-2em}&
    \langle \beta_{r-2p_+,s+k_1-k_2+p_-(m-l)}(\epsilon)|
    \prod_{i=1}^{k_1}V_{\alpha_-(\epsilon)/2}(v_i)V_{-k_2\alpha_-(\epsilon)/2}(u)\\
    &\hspace{-3em}\times \scr{+}(z_1)\cdots
    \scr{+}(z_{p_+-r})\scr{+}(w_1)\cdots \scr{+}(w_r)\\
    &\hspace{-3em}\times
    \rho_{2/\alpha_+(\epsilon)}(Q_{\lambda_{(m+1)p_+-r,(\ell+1)p_--s}}(x;\kappa_-(\epsilon)))|\beta_{r,s+(m-\ell)p_-}(\epsilon)\rangle
    \prod_{i=1}^{p_+-r}\frac{\d z_i}{z_i}\prod_{i=1}^{r}\frac{\d w_i}{w_i}\,,        
  \end{align*}
  where \(\rho\) is the map introduced for Proposition \ref{sec:fockspacesingvec}.
  Multiplying this matrix element by
  \begin{align*}
    &\prod_{i=1}^{k_1}v_i^{-\tfrac{1}{2}\alpha_-(\epsilon)\beta_{r,s+(m-\ell)p_-}(\epsilon)}
    u^{\tfrac{k_2}{2}\alpha_-(\epsilon)\beta_{r,s+(m-\ell)p_-}(\epsilon)+((m+1)p_+-r)((\ell+1)p_--s)}\\
    &\times \prod_{1\leq i\neq j\leq k_1}(v_i-v_j)^{-\tfrac{\kappa_-(\epsilon)}{4}}\prod_{i=1}^{k_1}(v_i-u)^{k_2\tfrac{\kappa_-(\epsilon)}{4}}\,,
  \end{align*}
  and then taking the limit \(u\rightarrow 0\) along the positive imaginary axis the integrand becomes
  \begin{align*}
    \widetilde{M}(v)=&(-1)^{p_+k_2}\prod_{i=1}^{k_1}\prod_{j=1}^{p_+-r}(v_i-z_j)^{-1}\prod_{i=1}^{k_1}\prod_{j=1}^r(v_i-w_j)^{-1}
    \\
    &\Xi_{-k_1\kappa_-(\epsilon)}\left(Q_{\lambda_{(m+1)p_+-r,(\ell+1)p_--s}}(x;\kappa_-(\epsilon))\right)
    \\
    &\times U_{p_+-r}(z;\kappa_+(\epsilon))\prod_{i=1}^{p_+-r}z_i^{(m+2)p_-+\epsilon
      p_+\overline{\kappa_+}(\epsilon)}
    U_r(w;\kappa_+(\epsilon))\prod_{i=1}^r w_i^{(m+1)p_-}\\
    &\times \prod_{i=1}^{p_+-r}\prod_{j=1}^r
    (1-\tfrac{w_j}{z_i})^{2\kappa_+(\epsilon)}\prod_{i=1}^{p_+-r}\frac{\d
      z_i}{z_i}
    \prod_{i=1}^{r}\frac{\d w_i}{w_i}\,,
  \end{align*}
  where \(\Xi\) is the map introduced in Proposition \ref{sec:polytofieldmap}.
  Next we take the limit \(v_1\rightarrow 0\) along the positive imaginary axis
  and then \(v_2\rightarrow 0\) and so on. As the \(v_i\) become sufficiently
  small the integrand \(\widetilde{M}(v)\) picks up degree 1 poles at
  \(w_i=v_j\). Upon evaluating the residues at these poles, one sees that they
  do not contribute in the limit \(v_i\rightarrow0\) and it therefore follows that
  \begin{align*}
    &\lim_{v_{k_1}\rightarrow0}\cdots\lim_{v_1\rightarrow0}\int_{[\Gamma_{p_+,r}(\kappa_+(\epsilon))]}
    \widetilde{M}(v)=\\
    &\ (-1)^{-p_+(k_2-k_1)}\Xi_{-k_1\kappa_-(\epsilon)}\left(Q_{\lambda_{(m+1)p_+-r,(\ell+1)p_--s}}(x;\kappa_-(\epsilon))\right)
    \frac{S_{p_+-r}((1-2p_++r)\kappa_+(\epsilon),0,\kappa_+(\epsilon))}{c_r^{\infty}(\kappa_+(\epsilon))}\,.
  \end{align*}
  This lies in \(\mathcal{O}^\times\) and thus
  \(E\scrp{+}{(m+1)p_+-r}|\beta_{p_+-r,s;-(m+1+\ell)}\rangle\)
  is non-zero.
  
  Finally we prove the derivation property of \(E\). Let Let \(A\in K_{1,1;2n;+}\), \(B\in K_{r,s;2m;+}\), \(n,m\in\mathbb{Z}\)
  and let \(\tilde A\in \tensor[_{\mathcal{O}}]{K}{_{1,1+2np_-;+}}\),
  \(\tilde B \in \tensor[_{\mathcal{O}}]{K}{_{r,s+2mp_-;+}}\) be lifts of
  \(A\) and \(B\). Then
  \begin{align*}
    E Y(\tilde A;u) \tilde B
    =[E, Y(\tilde A;u)] \tilde B + Y(\tilde A;u)
    E \tilde B\,.
  \end{align*}
  Let
  \([\overline{\Gamma}_{r}(\kappa_+(\epsilon))]\) be the renormalised cycles
  given in Definition \ref{sec:cycledef}.
  The
  commutator above is given by evaluating the Frobenius operator \(E\) of Definition \ref{sec:frobopdef} 
  on \(Y(\tilde A;u)\) with the integration contour centred at \(u\)
  \begin{align*}
    &[E, Y(\tilde A;u)] \\
    &=\int_{[\Gamma^\infty_{p_+-r}(\kappa_+(\epsilon))]}\hspace{-1em}
    Y\left(\prod_{i=1}^{p_+-r}\!\! \scr{+}(z_i-u)
    \res _{w=u} \int_{[\overline{\Gamma}_{r}(\kappa_+(\epsilon))]}
    \prod_{i=1}^r\scr{+}(w y_{i-1}-u)\tilde A;u\right)  \!\! w^{r-1}\!\! \prod_{i=1}^{p_+-r}\d z_i
    \d w\prod_{i=1}^{r-1} \d y_i\\
    &= \int_{[\Gamma^\infty_{p_+-r}(\kappa_+(\epsilon))]}
    \int_{[\overline{\Gamma}_{r}(\kappa_+(\epsilon))]}Y\left(\prod_{i=1}^{p_+-r}\scr{+}(z_i-u)
    \prod_{i=1}^r\scr{+}(u (y_{i-1}-1))(\scr{+} \tilde A);u\right)  u^{r-1} \prod_{i=1}^{p_+-r}\d z_i
    \d w\prod_{i=1}^{r-1} \d y_i\,.
  \end{align*}
  Note that we have suppressed total derivative terms in the \(y_i\) coordinates and that since
  \(\tilde A\in \tensor[_{\mathcal{O}}]{K}{_{1,1+2np_-;+}}\) for \(n\) not necessarily 0,
  the vertex operator \(Y(\ ,\ )\) is actually a generalised vertex operator in the sense of \cite{DongLepowski:1993}
  though this does not change how the calculation is performed.
  Setting \(u=-1\), then implies
  \begin{align*}
    &[E, Y(\tilde A;-1)] \\
    &= \int_{[\Gamma^\infty_{p_+-r}(\kappa_+(\epsilon))]}
    \int_{[\overline{\Gamma}_{r}(\kappa_+(\epsilon))]}\hspace{-3em}(-1)^{r-1}
    Y\left(\prod_{i=1}^{p_+-r}\scr{+}(z_i+1)
    \prod_{i=1}^r\scr{+}(1-y_{i-1})(\scr{+} \tilde A);-1\right) 
    \prod_{i=1}^{p_+-r}\d z_i
    \d w\prod_{i=1}^{r-1} \d y_i\,.
  \end{align*}
  Therefore we have an orientation reversal of the usual integration of the
  \(y_i\) coordinates and a shift in the lower bound of the cycle for the
  \(z_i\) coordinates. Since changing the lower bound of the cycle for the
  \(z_i\) coordinates only leads to corrections of order \(\epsilon\) and
  greater, we can shift the lower bound of the cycle for the \(z_i\) to
  1 and obtain
  \begin{align*}
    &[E, Y(\tilde A;-1)]\\
    &= \int_{[\Gamma^\infty_{p_+-r}(\kappa_+(\epsilon))]}
    \int_{[\overline{\Gamma}_{r}(\kappa_+(\epsilon))]}
    Y\left(\prod_{i=1}^{p_+-r}\scr{+}(z_i)\prod_{i=1}^{r-1}\scr{+}(y_i)
    (\scr{+} \tilde A);-1\right)\prod_{i=1}^{p_+-r}\d z_i\d y_i\quad \operatorname{mod} \epsilon\,.
  \end{align*}
  Recall that \(H^{p_+}(Y_{p_+},\tensor*[_{\mathcal{K}}]{\mathcal{L}}{_{p_+}}(\kappa_+(\epsilon),r))^{\symg{p_+}-}\) is 1 dimensional, so
  the two integrals 
  \begin{align*}
    &\int_{[\Gamma^\infty_{p_+-r}(\kappa_+(\epsilon))]}
    \int_{[\overline{\Gamma}_{r}(\kappa_+(\epsilon))]}
    Y\left(\prod_{i=1}^{p_+-r}\scr{+}(z_i)\prod_{i=1}^{r-1}\scr{+}(y_i)
    (\scr{+} \tilde A);-1\right)\prod_{i=1}^{p_+-r}\d z_i\d y_i\,,\\
    &
    \int_{[\Gamma^\infty_{p_+-1}(\kappa_+(\epsilon))]}
    Y\left(\prod_{i=1}^{p_+-1}\scr{+}(z_i)
    (\scr{+} \tilde A);-1\right)\prod_{i=1}^{p_+-1}\d z_i\,
  \end{align*}
  can at most differ by the multiplication with some constant.
  The multivalued part of the operator product expansion of \(\prod_{i=1}^{p_+-r}\scr{+}(z_i)\prod_{i=1}^{r-1}\scr{+}(y_i)\)
  on an element of \(\tensor[_{\mathcal{O}}]{K}{_{-1,1+2np_-;+}}\) is
  \begin{align*}
    \prod_{1\leq i\neq j\leq
      p_+-r}(z_i-z_j)^{\kappa_+(\epsilon)}\prod_{i=1}^{p_+-r}z_i^{2\kappa_+(\epsilon)}
    \prod_{i=1}^{p_+-r}\prod_{j=1}^{r-1}(z_i-y_j)^{2\kappa_+(\epsilon)}
    \prod_{1\leq i\neq j\leq
      r-1}(y_i-y_j)^{\kappa_+(\epsilon)}\prod_{i=1}^{r-1}y_i^{2\kappa_+(\epsilon)}
  \end{align*}
  and the normalising factor \(c_r^\infty(\kappa_+(\epsilon))\) in the definition of 
  \([\Gamma^\infty_{p_+-r}(\kappa_+(\epsilon))]\)
  was chosen such that
  \begin{align*}
    &\int_{[\Gamma^\infty_{p_+-r}(\kappa_+(\epsilon))]}
    \int_{[\overline{\Gamma}_{r}(\kappa_+(\epsilon))]}
    Y\left(\prod_{i=1}^{p_+-r}\scr{+}(z_i)\prod_{i=1}^{r-1}\scr{+}(y_i)
    (\scr{+} \tilde A);-1\right)\prod_{i=1}^{p_+-r}\d z_i\d y_i\\
    &=
    \int_{[\Gamma^\infty_{p_+-1}(\kappa_+(\epsilon))]}
    Y\left(\prod_{i=1}^{p_+-1}\scr{+}(z_i)
    (\scr{+} \tilde A);-1\right)\prod_{i=1}^{p_+-1}\d z_i\,.
  \end{align*}
  Setting \(\epsilon=0\), it therefore follows that
  \begin{align*}
    [E, Y(A;-1)]=
    Y(E A;-1)\,.
  \end{align*}
  Since \(E\) is a
  Virasoro homomorphism it then follows that
  \begin{align*}
    [E, Y(A;u)]&=e^{(1-u)L_{-1}}[E, Y( A;-1)]e^{(u-1)L_{-1}}\\
    &= Y(E A;u)\,.
  \end{align*}
\end{proof}

\begin{remark}
  The Frobenius operators \(E\) and \(F\) are only well defined on the kernels of the screening operators \(\scr{+}\)
  and \(\scr{-}\) respectively. Attempting to evaluate \(E\) and \(F\) on the Fock modules \(F_{r,s;n}\) leads to
  poles in the deformation parameter \(\epsilon\).
\end{remark}

\begin{definition}
  For \(1\leq r\leq p_+\), \(1\leq s\leq p_-\) and \(n\geq0\), bases
  of the sectors \(V_{r,s;n}^\pm\subset \mathcal{K}_{r,s;n}^\pm\) in terms of
  the Frobenius homomorphisms \(E\) and \(F\) are given by:
\begin{enumerate}
  \item \(E\) basis
    \begin{align*}
      V_{r,s:n}^+&=\left\{
        \begin{array}{lc}
          \bigoplus_{m=-n}^{n}\mathbb{C}
          E^{n+m}|\beta_{p_+,p_-;-2n}\rangle& r=p_+,s=p_-\\
          \bigoplus_{m=-n}^{n}\mathbb{C}
          E^{n+m}|\beta_{p_+,s;-2n}\rangle& r=p_+,s\neq p_-\\
          \bigoplus_{m=-n}^{n}\mathbb{C}
          E^{n+m}\scrp{+}{p_+-r}|\beta_{p_+-r,p_-;-2n-1}\rangle& r\neq p_+,s=p_-\\
          \bigoplus_{m=-n}^{n}\mathbb{C}
          E^{n+m}\scrp{+}{p_+-r}|\beta_{p_+-r,s;-2n-1}\rangle& r\neq p_+,s\neq p_-
        \end{array}\right.\\
      V_{r,s:n}^-&=\left\{
        \begin{array}{lc}
          \bigoplus_{m=-n}^{n+1}\mathbb{C}
          E^{n+m}|\beta_{p_+,p_-;-2n-1}\rangle& r=p_+,s=p_-\\
          \bigoplus_{m=-n}^{n+1}\mathbb{C}
          E^{n+m}|\beta_{p_+,s;-2n-1}\rangle& r=p_+,s\neq p_-\\
          \bigoplus_{m=-n}^{n+1}\mathbb{C}
          E^{n+m}\scrp{+}{p_+-r}|\beta_{p_+-r,p_-;-2n-2}\rangle& r\neq p_+,s=p_-\\
          \bigoplus_{m=-n}^{n+1}\mathbb{C}
          E^{n+m}\scrp{+}{p_+-r}|\beta_{p_+-r,s;-2n-2}\rangle& r\neq p_+,s\neq p_-
        \end{array}\right.
    \end{align*}
  \item \(F\) basis
    \begin{align*}
      V_{r,s:n}^+&=\left\{
        \begin{array}{lc}
          \bigoplus_{m=-n}^{n}\mathbb{C}
          F^{n-m}|\beta_{p_+,p_-;2n}\rangle& r=p_+,s=p_-\\
          \bigoplus_{m=-n}^{n}\mathbb{C}
          F^{n-m}\scrp{-}{p_--s}|\beta_{p_+,p_--s;2n+1}\rangle& r=p_+,s\neq p_-\\
          \bigoplus_{m=-n}^{n}\mathbb{C}
          F^{n-m}|\beta_{r,p_-;2n}\rangle& r\neq p_+,s=p_-\\
          \bigoplus_{m=-n}^{n}\mathbb{C}
          F^{n-m}\scrp{-}{p_--s}|\beta_{r,p_--s;2n+1}\rangle& r\neq p_+,s\neq p_-
        \end{array}\right.\\
      V_{r,s:n}^-&=\left\{
        \begin{array}{lc}
          \bigoplus_{m=-n}^{n+1}\mathbb{C}
          F^{n-m+1}|\beta_{p_+,p_-;2n+1}\rangle& r=p_+,s=p_-\\
          \bigoplus_{m=-n}^{n+1}\mathbb{C}
          F^{n-m+1}\scrp{-}{p_--s}|\beta_{p_+,p_--s;2n+2}\rangle& r=p_+,s\neq p_-\\
          \bigoplus_{m=-n}^{n+1}\mathbb{C}
          F^{n-m+1}|\beta_{r,p_-;2n+1}\rangle& r\neq p_+,s=p_-\\
          \bigoplus_{m=-n}^{n+1}\mathbb{C}
          F^{n-m+1}\scrp{-}{p_--s}|\beta_{r,p_--s;2n+2}\rangle& r\neq p_+,s\neq p_-
        \end{array}\right.
    \end{align*}
 \end{enumerate}  
\end{definition}

\section{The extended {\boldmath \(W\)-algebra \(\mathcal{M}_{p_+,p_-}\)\unboldmath}}\label{sec:Mppalg}

In this section we apply the theory developed in the previous sections to
analysing the extended \(W\)-algebra \(\mathcal{M}_{p_+,p_-}\).

\subsection{The algebra structure of {\boldmath\(\mathcal{M}_{p_+,p_-}\)\unboldmath}}

\begin{definition}\label{sec:Mpqdef}\ 
  \begin{enumerate}
  \item The extended \(W\)-algebra
    \(\mathcal{M}_{p_+,p_-}=(\mathcal{K}^+_{1,1},|0\rangle,T,Y)\) is a
    subVOA of \(\mathcal{V}_{p_+,p_-}\), where the vacuum vector,
    conformal vector and vertex operator map are those of
    \(\mathcal{V}_{p_+,p_-}\); and
    \begin{align*}
      \mathcal{K}^+_{1,1}=\ker \scr{+}\cap\ker \scr{-}\subset V^+_{1,1}\,.
    \end{align*}
  \item Let \(\mathcal{M}_{p_+,p_-}^{(0)}=\mathcal{M}_{p_+,p_-}\cap F_{1,1;0} = (K_{1,1;0},|0\rangle,T,Y)\)
    be the restriction of \(\mathcal{M}_{p_+,p_-}\) to its Heisenberg weight 0 subspace. Then
    \(\mathcal{M}_{p_+,p_-}^{(0)}\) contains the vacuum and Virasoro states \(|0\rangle,\ T\) and is a subVOA of
    \(\mathcal{M}_{p_+,p_-}\).
  \end{enumerate}
\end{definition}

We specialise some of the notation of Section \ref{sec:Virrepthy}.
\begin{definition}\label{sec:specialnotation}\ 
  \begin{enumerate}
  \item For \(n\geq 0\) let
    \begin{align*}
      \alpha_n^+&=\beta_{p_+-1,1;-2n-1}=\alpha_+ - (n+1)\alpha\,,\\
      \alpha_n^-&=\beta_{1,p_--1;2n+1}=\alpha_- + (n+1)\alpha\,,
    \end{align*}
    such that \(\alpha_n^-={\alpha_n^+}^\vee\). The conformal weight
    corresponding to these Heisenberg weights is
    \begin{align*}
      \Delta^+_{1,1;n}&=h_{\alpha_n^\pm}=((n+1)p_+-1)((n+1)p_--1)\,.
    \end{align*}
    We abbreviate \(\Delta_{1,1;n}^+\) as \(\Delta_n\).
  \item Let \(\lambda^+_{n,m}\) and \(\lambda^-_{n,m}\) be the conjugate partitions
    \begin{align*}
      \lambda^+_{n,m}&=\lambda_{(n+m+1)p_+-1,(n-m+1)p_--1}\\
      \lambda^-_{n,m}&=\lambda_{(n-m+1)p_--1,(n+m+1)p_+-1}\,,
    \end{align*}
    where \(\lambda_{N,M}=(M,\dots,M)\) is the partition consisting of \(N\) copies of \(M\).
  \item For \(n\geq0\) and \(-n\leq m\leq n\) let \(W_{n,m}\) be the basis of
    the soliton sector \(V^+_{1,1;n}\) given by
    \begin{align*}
      W_{n,m}=E^{n+m}\scrp{+}{p_+-1}|\alpha_n^+\rangle\,.
    \end{align*}
  \end{enumerate}
\end{definition}

\begin{prop}
  For \(1\leq r<p_+\) and \(1\leq s<p_-\) the screening operators
  \(\scrp{+}{r}\) and \(\scrp{-}{s}\) are \(\mathcal{M}_{p_+,p_-}\) homomorphisms, that is, for
  \(A\in \mathcal{K}_{1,1}^+\)
  \begin{align*}
    [\scrp{+}{r},Y(A;w)]&=0\,,& [\scrp{-}{s},Y(A;w)]=0\,.
  \end{align*}
\end{prop}
\begin{proof}
  We show the \(\scr{+}\) case. 
  \begin{align*}
    &[\scrp{+}{r},Y(A;w)]=\res_{z=w} \scrp{+}{r}(z) Y(A;w)\d z\\
    &\qquad= \res_{z=w}\int_{[\overline{\Gamma}_r(\kappa_+)]}
    z^{r-1}Y(\prod_{i=1}^{r-1}\scr{+}((z-w)y_i)\scr{+}(z-w)
    A;w)\d z\d y_1\cdots \d y_{r-1}\\
    &\qquad=\int_{[\overline{\Gamma}_r(\kappa_+)]}
    z^{r-1}Y(\prod_{i=1}^{r-1}\scr{+}((z-w)y_i)
    (\scr{+}A);w)\d y_1\cdots \d y_{r-1}=0\,,
  \end{align*}
  since \(\scr{+}A=0\).
  Note that we have suppressed total derivative terms in the 
  \(y_i\) coordinates. Note that the vertex operator
    \(Y(\ ,\ )\) is technically a generalised vertex operator in the sense of \cite{DongLepowski:1993}, but this does
    not change how the calculation is performed.
\end{proof}

\begin{thm}\label{sec:indecstructure}\ 
  The extended \(W\)-algebra \(\mathcal{M}_{p_+,p_-}\) is generated by the fields \(T(z), Y(W_{1,1};z)\) and \(Y(W_{1,-1};z)\),
  while \(\mathcal{M}_{p_+,p_-}^{(0)}\) is generated by the fields \(T(z)\) and \(Y(W_{1,0};z)\).
  Furthermore the VOAs \(\mathcal{M}_{p_+,p_-}\)  and 
  \(\mathcal{M}_{p_+,p_-}^{(0)}\) satisfy:
  \begin{enumerate}
  \item The Frobenius homomorphisms \(E\) and \(F\) are derivations of the extended \(W\)-algebra \(\mathcal{M}_{p_+,p_-}\).
  \item The module \(\mathcal{X}^+_{1,1}\subset \mathcal{K}^+_{1,1}\) is simple as an \(\mathcal{M}_{p_+,p_-}\)-module and is
    generated by the soliton vector \(W_{0,0}\).
  \item The quotient \(\mathcal{K}^+_{1,1}/\mathcal{X}^+_{1,1}\cong L(0)\) is a simple \(\mathcal{M}_{p_+,p_-}\)-module
    on which \(Y(A;z)=0\) for \(A\in \mathcal{X}^+_{1,1}\).
  \item The sequence
    \begin{align*}
      0\longrightarrow \mathcal{X}^+_{1,1}\longrightarrow \mathcal{K}^+_{1,1}\longrightarrow L(0)\longrightarrow 0\,,
    \end{align*}
    is an exact sequence of \(\mathcal{M}_{p_+,p_-}\)-modules.
  \item The module \(X_{1,1;0}\) is simple as an \(\mathcal{M}_{p_+,p_-}^{(0)}\)-module and is
    generated by the soliton vector \(W_{0,0}\).
  \item The quotient \(K_{1,1;0}/X_{1,1;0}\cong L(0)\) is a simple \(\mathcal{M}_{p_+,p_-}^{(0)}\)-module
    on which \(Y(A;z)=0\) for \(A\in X_{1,1;0}\).
  \item The sequence
    \begin{align*}
      0\longrightarrow X_{1,1;0}\longrightarrow K_{1,1;0}\longrightarrow L(0)\longrightarrow 0\,,
    \end{align*}
    is an exact sequence of \(\mathcal{M}_{p_+,p_-}^{(0)}\)-modules.
  \end{enumerate}
\end{thm}
\begin{remark}
  The modes of the fields  \(T(z), Y(W_{1,1};z)\) and \(Y(W_{1,-1};z)\) generate all of \(\mathcal{M}_{p_+,p_-}\)
  by acting on the vacuum vector \(|0\rangle\), however, the operator product expansion of \(Y(W_{1,1};z)\)
  with \(Y(W_{1,-1};z)\) contains \(Y(W_{1,0};z)\) as one of its singular terms. Therefore \(T(z), Y(W_{1,1};z)\)
  and \(Y(W_{1,-1};z)\) are what is sometimes referred to as weak generators while \(T(z), Y(W_{1,1};z), Y(W_{1,0};z)\)
  and \(Y(W_{1,-1};z)\) are strong generators.
\end{remark}

In order to prove the above Theorem we first state the following
proposition.
\begin{prop}\ \label{sec:indecstructure_technical}
  \begin{enumerate}
  \item For \(n\geq0\)
    \begin{align*}
      W_{1,1}[\Delta_n-\Delta_{n+1}]W_{n,n}&=a_nW_{n+1,n+1}\,,&
      W_{1,-1}[\Delta_n-\Delta_{n+1}]W_{n,-n}&=b_n W_{n+1,-(n+1)}\,,\\
      W_{1,1}[\Delta_{n+1}-\Delta_{n}]W_{n+1,-(n+1)}&=c_nW_{n,-n}\,,&
      W_{1,-1}[\Delta_{n+1}-\Delta_{n}]W_{n+1,n+1}&=d_n W_{n,n}\,,
    \end{align*}
    where \(a_{n},b_n,c_n\) and \(d_{n}\) are non-zero constants.
  \item For \(m>\Delta_1-\Delta_0\)
    \begin{align*}
      W_{1,-1}[m] W_{1,1}&=0&
      W_{1,1}[m] W_{1,-1}&=0
    \end{align*}
  \end{enumerate}
\end{prop}
\begin{proof}
  Since up to multiplication by non-zero constants the \(W_{n,\pm n}\) are
  given by \(\scrp{\mp}{p_\mp-1}|\alpha_n^\mp\rangle\), 1. is proven by
  showing that the matrix elements
  \begin{align*}
    &\int_{[\Gamma_{p_\mp-1}(\kappa_\mp)]}\langle
    \alpha_{n+1}^\mp|\scr{-}(z_1+w)\cdots
    \scr{-}(z_{p_\mp-1}+w)V_{\alpha_1^\mp}(w)|\alpha_n^\mp\rangle\prod_{i=1}^{p_\mp-1}\d
    z_i\,,\\
    &\int_{[\Gamma_{p_\mp-1}(\kappa_\mp)]}\langle
    \alpha_{n-1}^\pm|\scr{-}(z_1+w)\cdots
    \scr{-}(z_{p_\mp-1}+w)V_{\alpha_1^\mp}(w)|\alpha_n^\pm\rangle\prod_{i=1}^{p_\mp-1}\d z_i\,,
  \end{align*}
  are non-zero.
  We detail how to evaluate one of these matrix elements. The rest follow in
  the same way.
  \begin{align*}
    &\int_{[\Gamma_{p_--1}(\kappa_-)]}\langle
    \alpha_{n+1}^-|\scr{-}(z_1+w)\cdots
    \scr{-}(z_{p_\mp-1}+w)V_{\alpha_1^-}(w)|\alpha_n^-\rangle\prod_{i=1}^{p_--1}\d
    z_i\,,\\
    &\ =
    \int_{[\Gamma_{p_--1}(\kappa_-)]}\prod_{1\leq i\neq j\leq
      p_--1}(z_i-z_j)^{\kappa_-}\prod_{i=1}^{p_--1}z_i^{(2-p_-))\kappa_-+1-3p_+}\\
    &\ \quad\times
    \prod_{i=1}^{p_-1}(1+\tfrac{z_i}{w})^{2\kappa_--2(n+1)p_+}w^{2(n+1)p_+p_--2p_+}\prod_{i=1}^{p_--1}\frac{\d
      z_i}{z_i}\\
    &\
    =w^{\Delta_{n+1}-\Delta_n-\Delta_1}\Xi_{2(n+1)p_--1}(Q_{\lambda^-_{2,1}}(x;\kappa_+))\neq 0\,,
  \end{align*}
  where \(\Xi\) is the map described in Proposition
  \ref{sec:polytofieldmap}. 

  2. follows from the fact that
  \begin{align*}
    W_{1,-1}[m] |\alpha_1^-\rangle&=0&
    W_{1,1}[m] |\alpha_1^+\rangle&=0&
  \end{align*}
  for \(m>\Delta_1-\Delta_0\).
\end{proof}

\begin{proof}[Proof of Theorem \ref{sec:indecstructure}]
  1. follows directly from Theorem \ref{sec:frobopthm}.
  We prove the points 2. - 4. pertaining to \(\mathcal{M}_{p_+,p_-}\),
  points 5. - 7. follow by the same methods.

  2. Since \(W_{0,0}\) is a Virasoro descendant of the vacuum \(|0\rangle\), proving that \(\mathcal{M}_{p_+,p_-}\) is generated by
  \(T(z), Y(W_{1,-1};z)\) and \(Y(W_{1,1};z)\), reduces to showing that all soliton vectors \(W_{n,m}\) can be reached by acting on
  \(W_{0,0}\) with the modes of \(Y(W_{1,-1};z)\) and \(Y(W_{1,1};z)\). 
  We prove this by showing that \(W_{1,-1}[\Delta_n-\Delta_{n\pm1}]W_{n,m}\)
  and \(W_{1,1}[\Delta_n-\Delta_{n\pm1}]W_{n,m}\) have non-zero contributions
  from \(W_{n\pm1,m-1}\) and \(W_{n\pm1,m+1}\) respectively. 

  We show the
  \(W_{1,-1}\) case.
  Since \(F W_{1,-1}=0\), the derivation property of the Frobenius homomorphisms implies that 
  \(Y(W_{1,-1};z)\) commutes with \(F\). 
  Therefore
  \begin{align*}
    F^{m+n}(W_{1,-1}[\Delta_n-\Delta_{n+1}]W_{n,m})=\operatorname{const} W_{n+1,-(n+1)}\,,
  \end{align*}
  which implies that \(W_{1,-1}[\Delta_n-\Delta_{n+1}]W_{n,m}\) has non-zero
  contributions from \(W_{n+1,m-1}\), similarly
  \begin{align*}
    W_{1,-1}[\Delta_n-\Delta_{n-1}]W_{n,m}=\operatorname{const}F^{n-m} W_{n-1,n-1}\,,
  \end{align*}
  which implies that \(W_{1,-1}[\Delta_n-\Delta_{n-1}]W_{n,m}\) has non-zero
  contributions from \(W_{n-1,m-1}\).

  3. and 4. Since \(\mathcal{K}_{1,1}^+\) and \(\mathcal{X}_{1,1}^+\) are \(\mathcal{M}_{p_+,p_-}\)-modules their quotient 
  \(\mathcal{K}_{1,1}^+/\mathcal{X}_{1,1}^+\cong L(0)\) is an \(\mathcal{M}_{p_+,p_-}\)-module on which all fields 
  \(Y(A;z), A\in \mathcal{X}_{1,1}^+\) act
  trivially. Thus the sequence 
  \begin{align*}
    0\longrightarrow \mathcal{X}^+_{1,1}\longrightarrow \mathcal{K}^+_{1,1}\longrightarrow L(0)\longrightarrow 0\,,
  \end{align*}
  is an exact sequence of \(\mathcal{M}_{p_+,p_-}\)-modules.
\end{proof}

\begin{prop}\label{sec:Wzeromodeeigenvalues}\ 
  Let \(\omega_n(\beta)\) be the eigenvalues of the zero modes \(W_{n,0}[0]\) acting on the generating state \(|\beta\rangle\in F_\beta\)
  for \(\beta\in\mathbb{C}\) and \(n\geq 0\).
  \begin{enumerate}
  \item For \(n\geq 0\) the eigenvalue \(\omega_n(\beta)\) is a degree \(\Delta_n\) polynomial in \(\beta\)
    \begin{align*}
      \omega_n(\beta)=\langle\beta| W_{n,0}[0]|\beta\rangle
      =\operatorname{const}\cdot\prod_{i=1}^{(n+1)p_+-1}\prod_{j=1}^{(n+1)p_--1}
      (\beta-\beta_{i,j})\,.
    \end{align*}
  \item For \(n\geq0\) and \(h_\beta=\tfrac12\beta(\beta-\alpha_0)\) there exist polynomials \(g_n(h)\in\mathbb{C}[h]\) such that
    \begin{align*}
      g_n(h_\beta)=\left\{
        \begin{array}{cc}
          \omega_n(\beta)& n\operatorname{\ even}\\
          \omega_n(\beta)^2& n\operatorname{\ odd}
        \end{array}\right.\,.
    \end{align*}
    For \(n=0,1,2\) these polynomials are
    \begin{align*}
      g_0(h_\beta)=\omega_0(\beta)&=\operatorname{const}\cdot \prod_{[(i,j)]\in \mathcal{T}}(h_\beta-\Delta_{i,j})\,,\\
      g_1(h_\beta)=\omega_1(\beta)^2&=\operatorname{const}\cdot \prod_{[(i,j)]\in \mathcal{T}}(h_\beta-\Delta_{i,j})^4
      \prod_{i=1}^{p_+-1}\prod_{j=1}^{p_--1}(h_\beta-\Delta_{i,j;0}^+)^2\\
      &\ \prod_{i=1}^{p_+-1}(h_\beta-\Delta_{i,p_-;0}^+)^2
      \prod_{j=1}^{p_--1}(h_\beta-\Delta_{p_+,j;0}^+)^2\\
      &\ (h_\beta-\Delta_{p_+,p_-;0}^+)\,,\\
      g_2(h_\beta)=\omega_2(\beta)&=\operatorname{const}\cdot \prod_{[(i,j)]\in \mathcal{T}}(h_\beta-\Delta_{i,j})^3\\
      &\ \prod_{i=1}^{p_+-1}\prod_{j=1}^{p_--1}(h_\beta-\Delta_{i,j;0}^+)^2
      \prod_{i=1}^{p_+-1}\prod_{j=1}^{p_--1}(h_\beta-\Delta_{i,j;0}^-)\\
      &\ \prod_{i=1}^{p_+-1}(h_\beta-\Delta_{i,p_-;0}^+)^2
      \prod_{i=1}^{p_+-1}(h_\beta-\Delta_{i,p_-;0}^-)\\
      &\ \prod_{j=1}^{p_--1}(h_\beta-\Delta_{p_+,j;0}^+)^2
      \prod_{j=1}^{p_--1}(h_\beta-\Delta_{p_+,j;0}^-)\\
      &\ (h_\beta-\Delta_{p_+,p_-;0}^+)(h_\beta-\Delta_{p_+,p_-;0}^-)\,,
    \end{align*}
    where \(\mathcal{T}\) is the Kac table and \(\Delta_{r,s}\) and
    \(\Delta_{r,s;0}^\pm\) are the conformal weights of Definition \ref{sec:solitonweights}.
\end{enumerate}
\end{prop}
\begin{proof}
  1. can be expressed in terms of Jack polynomials and evaluated accordingly.
  \begin{align*}
    &\int_{[\Gamma_{(n+1)p_+-1}^+]}\hspace{-2em} w^{\Delta_n}\langle\beta|\scr{+}(z_1+w)\cdots
    \scr{+}(z_{(n+1)p_+-1}+w)V_{\alpha_n^+}(w)|\beta\rangle \hspace{-0.5em}\prod_{i=1}^{(n+1)p_+-1}\hspace{-0.5em}\d z_i\\
    &=\Xi_{\alpha_-\beta}(Q_{\lambda_{n,0}^+}(x;\kappa_-))=\hspace{-0.5em}\prod_{i=1}^{(n+1)p_+-1}
    \prod_{j=1}^{(n+1)p_--1}\hspace{-0.5em}\frac{\beta-\beta_{i,j}}{\beta_{(n+1)p_+-i,1+j-(n+1)p_-}}\,,
  \end{align*}
  where \(\Xi\) is the map described in Proposition \ref{sec:polytofieldmap}.

  2. follows from 1. by using the identity
  \begin{align*}
    (\beta-\beta_{i,j})(\beta-\beta_{-i,-j})
      =\beta^2-\beta(\beta_{i,j}+\beta_{-i,-j})+\beta_{i,j}\beta_{-i,-j}=2(h_\beta-h_{i,j})\,.
  \end{align*}
\end{proof}

\subsection{The zero mode algebra and the 
  {\boldmath\(c_2\)\unboldmath}-cofiniteness condition}
\label{sec:zhualgsec}

In this section we will determine the structure of the zero mode algebra
\(A_0(\mathcal{M}_{p_+,p_-})\) and the Poisson algebra 
\(\mathfrak{p}(\mathcal{M}_{p_+,p_-})\), as well as prove that
\(\mathcal{M}_{p_+,p_-}\) satisfies Zhu's \(c_2\)-cofiniteness condition.
Let
\(A_0=A_0(\mathcal{M}_{p_+,p_-})\). We will first compute relations for the zero mode algebra
\(A_0\) and then show that they imply that the Poisson algebra is
finite dimensional and that therefore \(\mathcal{M}_{p_+,p_-}\) satisfies Zhu's \(c_2\)-cofiniteness condition.

\begin{prop}\label{sec:firstMppZhuprop}\ 
  \begin{enumerate}
\item The \(\mathcal{M}_{p_+,p_-}\) derivations \(E\) and \(F\) induce derivations on \(A_0\).
\item The \(\mathcal{M}_{p_+,p_-}\) anti-involution \(\sigma\) defines an anti-involution of \(A_0\) and
  \(A_0(\mathcal{M}_{p_+,p_-}^{(0)})\). On \(A_0\) it satisfies
  \begin{align*}
    \sigma([T])&=[T]&\sigma([W_{1,m}])&=-[W_{1,m}]
  \end{align*}
  for \(m=-1,0,1\) and on \(A_0(\mathcal{M}_{p_+,p_-}^{(0)})\)
  \begin{align*}
    \sigma([T])&=[T]&
    \sigma([W_{n,0}])&=(-1)^{n}[W_{n,0}]\,,
  \end{align*}
  for \(n\geq 0\).
\item We have the following surjection onto the zero mode algebra \(A_0\)
  from a subspace of \(\mathcal{K}_{1,1}^+\)
    \begin{align*}
      \mathbb{C}[T]\oplus\bigoplus_{m=-1}^{1}
      \mathbb{C}[T]\ast W_{1,m} \rightarrow A_0\,,
    \end{align*}
    where \(\mathbb{C}[T]\) denotes polynomials in \(T\) with multiplication
    \(\ast\).
  \item We have the following surjection onto the zero mode algebra \(A_0(\mathcal{M}_{p_+,p_-}^{(0)})\) from a
    subspace of \(K_{1,1;0}\)
    \begin{align*}
      \mathbb{C}[T]\oplus\bigoplus_{n\geq 1}
      \mathbb{C}[T]\ast W_{n,0} \rightarrow A_0(\mathcal{M}_{p_+,p_-}^{(0)})\,,
    \end{align*}
    where \(\mathbb{C}[T]\) denotes polynomials in \(T\) with multiplication
    \(\ast\).
  \end{enumerate}
\end{prop}

\begin{proof}
  1. and 2. follow because they hold in \(U(\mathcal{M}_{p_+,p_-})[0]\) and
  descend to \(A_0\) and \(A_0(\mathcal{M}_{p_+,p_-}^{(0)})\).

  We prove 3. by using Zhu's formulation of the zero mode algebra. It follows from 
  Lemma \ref{sec:indecstructure_technical} and Proposition \ref{sec:zhuproperties} that for \(n\geq 2\)
  \begin{align*}
    W_{n,n}=\operatorname{const}\cdot W_{1,1}[\Delta_{n-1}-\Delta_{n}]W_{n-1,n-1}\in O(\mathcal{M}_{p_+,p_-})\,.
  \end{align*}
  By applying \(F\) multiple times to \(W_{n,n}\) we see that
  \(W_{n,m}\in O(\mathcal{M}_{p_+,p_-})\) for \(n\geq2\) and \(-n\leq m\leq n\). From Proposition
  \ref{sec:zhuproperties} it also follows that the image of \(L_{-n}\) for
  \(n\geq 3\) satisfies
  \begin{align*}
    L_{-n}A\in \mathbb{C}[L_{-2},L_{-1},L_0]A \text{ mod } O(\mathcal{M}_{p_+,p_-})\,.
  \end{align*}
  Since \((L_0+L_{-1})A\in O(\mathcal{M}_{p_+,p_-})\) we therefore have the following surjection of
  vector spaces
  \begin{align*}
    \mathbb{C}[L_{-2}]|0\rangle\oplus \bigoplus_{m=-1}^1 \mathbb{C}[L_{-2}] W_{1,m}
    \longrightarrow A_0\,.
  \end{align*}
  Thus \(A_0\) is spanned by polynomials in \(T\) and polynomials in \(T\)
  times \(W_{1,m}\).

  4. follows from the same arguments used to prove 3. and
  the decomposition of \(K_{1,1;0}\) as a \(U(\mathcal{L})\)
  module. 
\end{proof}

\begin{thm}\label{sec:zhualgrelations}\ 
  \begin{enumerate}
  \item The zero mode algebra \(A_0\) is finite dimensional and is generated by \([T],[W_{1,m}],\ m=-1,0,1\).
  \item In the zero mode algebra \(A_0(\mathcal{M}_{p_+,p_-}^{(0)})\) the Virasoro element satisfies the polynomial relations
    \begin{align*}
      g_n([T])=\left\{
        \begin{array}{ll}
          [W_{n,0}]&n\ \operatorname{even}\\
          {[W_{n,0}]}^2&n\ \operatorname{odd}
        \end{array}\right.\,,
    \end{align*}
    where \(n\geq0\) and the the \(g_i\) are the polynomials of Proposition \ref{sec:Wzeromodeeigenvalues}.
  \item In the zero mode algebra \(A_0\) the Virasoro element satisfies the polynomials relations
    \begin{align*}
      [W_{0,0}]&=g_0([T])&[W_{1,0}]^2&=g_1([T])\\
      [W_{2,0}]&=g_2([T])=0\,,      
    \end{align*}
    where the \(g_i\) are the polynomials of Proposition \ref{sec:Wzeromodeeigenvalues}.
  \item The products of the generators \([W_{1,m}],\ m=-1,0,1\) are given by\\
    \begin{center}
      \begin{tabular}{c|c|c|c}
            \(\cdot\)&\([W_{1,-1}]\)&\([W_{1,0}]\)&\([W_{1,1}]\)\\\hline
            \([W_{1,-1}]\)&\(0\)&\(-f([T])[W_{1,-1}]\)&\(-g_1([T])-f([T])[W_{1,0}]\)\\\hline
            \([W_{1,0}]\)&\(f([T])[W_{1,-1}]\)&\(g_1([T])\)&\(-f([T])[W_{1,1}]\)\\\hline
            \([W_{1,1}]\)&\(-g_1([T])+f([T])[W_{1,0}]\)&\(f([T])[W_{1,1}]\)&\(0\)
          \end{tabular}
    \end{center}
    where \(f\) is a degree less than \(\Delta_1/2\) polynomial and \(g_1\) 
    is that of Proposition \ref{sec:Wzeromodeeigenvalues}.
  \item The commutators of the generators \([W_{1,m}],\ m=-1,0,1\) are given by
    \begin{align*}
      [[W_{1,0}],[W_{1,1}]]&=-2f([T])[W_{1,1}]\,,&[[W_{1,0}],[W_{1,-1}]]&=2f([T])[W_{1,-1}]\,,\\
      [[W_{1,1}],[W_{1,-1}]]&=2f([T])[W_{1,0}]\,,
    \end{align*}
\end{enumerate}
\end{thm}
\begin{proof}
  We first prove the polynomial relations of 2. 
  Let \(k\geq 0\). Since
  \begin{align*}
    W_{2k+1,0}\ast W_{2k+1,0}\in U(\mathcal{L})|0\rangle\oplus\bigoplus_{n=1}^{m-1}U(\mathcal{L})W_{n,0}\,,
  \end{align*}
  where \(m\) is the smallest integer such that \(\Delta_m >2\Delta_{2k+1}\), it follows that in 
  \(A_0(\mathcal{M}_{p_+,p_-}^{(0)})\)
  \begin{align*}
    [W_{2k+1,0}]^2=f_0^{(2k+1)}([T])+\sum_{n=1}^{m-1} f_n^{(2k+1)}([T])[W_{n,0}]\,,
  \end{align*}
  for some polynomials \(f_n^{(2k+1)}\). Furthermore, since
  \begin{align*}
    W_{1,0}[\Delta_{2k+1}-\Delta_{2k+2}]W_{2k+1,0}=\operatorname{const}W_{2k+2,0} +
    \bigoplus_{n=0}^{2k+1}U(\mathcal{L})W_{n,0}
  \end{align*}
  lies in \(O(\mathcal{M}_{p_+,p_-}^{(0)})\), it follows that in 
  \(A_0(\mathcal{M}_{p_+,p_-}^{(0)})\)
  \begin{align*}
    [W_{2k+2,0}]=f_0^{(2k+2)}([T])+\sum_{n=1}^{2k+1} f_n^{(2k+2)}([T])[W_{n,0}]\,,
  \end{align*}
  for some polynomials \(f_n^{(2k+2)}\). Additionally, since
  \(\sigma([W_{2k+1,0}]^2)=[W_{2k+1,0}]^2\) and \(\sigma([W_{2k+2,0}])=[W_{2k+2,0}]\), the
  polynomials \(f_n^{2k+1},f_n^{2k+2}\) must vanish for \(n\) odd. Finally, it follows by induction in \(k\)
  that in \(A_0(\mathcal{M}_{p_+,p_-}^{(0)})\) the elements \([W_{2k+1,0}]^2\) and \([W_{2k+2,0}]\) for \(k\geq0\) are polynomials in \([T]\), since
  \(W_{0,0}\) is a Virasoro descendant of the vacuum. The formulae for these polynomials follow from
  Proposition \ref{sec:Wzeromodeeigenvalues}.

  Next we prove 3.
  The above calculations for \(A_0(\mathcal{M}_{p_+,p_-}^{(0)})\) imply that in \(A_0\),
  \([W_{0,0}]=g_0([T]),\ [W_{1,0}]^2=g_1([T]),\ [W_{2,0}]=g_2([T])\) and since \([W_{2,0}]\) vanishes in \(A_0\), 
  it follows that \(g_2([T])=0\).  
  The finite dimensionality of \(A_0\) then follows from Proposition
  \ref{sec:firstMppZhuprop} and \(g_2([T])=0\), thus proving 1.
 
  Finally we prove 4. and 5.
  The relation \([W_{1,\pm1}]^2=0\) follows from the fact that
  \(W_{1,\pm1}\ast W_{1,\pm1}=0\) in \(\mathcal{K}_{1,1}^+\).
  To show the relation for \([W_{1,0}]\cdot[W_{1,-1}]\) recall that
  \begin{align*}
    W_{1,0}\ast W_{1,-1}\in U(\mathcal{L})W_{1,-1} \,.
  \end{align*}
  This implies that in \(A_0\)
  \begin{align*}
    [W_{1,0}]\cdot[W_{1,-1}]=f([T]) [W_{1,-1}]\,.
  \end{align*}
  for some degree less than \(\Delta_1/2\) polynomial \(f\). We take this to
  be the polynomial \(f\) in the theorem.
  The remaining relations follow by the application of \(E\) and \(\sigma\)
  to the above relations.
\end{proof}

\begin{thm}\label{sec:poissonalgrelations}\ 
  \begin{enumerate}
  \item   The Poisson algebra \(\mathfrak{p}(\mathcal{M}_{p_+,p_-})\)
    is finite dimensional and \(\mathcal{M}_{p_+,p_-}\) therefore satisfies
    Zhu's \(c_2\)-cofiniteness condition.
  \item As a commutative algebra \(\mathfrak{p}(\mathcal{M}_{p_+,p_-})\) is
    generated by \([T]_{\mathfrak{p}}\) and \([W_{1,m}]_{\mathfrak{p}},\ m=-1,0,1\).
  \item The maps \(E\) and \(F\) act as derivations on \(\mathfrak{p}(\mathcal{M}_{p_+,p_-})\).
  \item 
    \begin{align*}
      [W_{1,\pm1}]^2_{\mathfrak{p}}&=0\\
      [W_{1,1}]_{\mathfrak{p}}\cdot[W_{1,-1}]_{\mathfrak{p}}&=-[W_{1,0}]^2_{\mathfrak{p}}\\
      [W_{1,0}]^2_{\mathfrak{p}}&=\operatorname{const}\cdot[T]_{\mathfrak{p}}^{\Delta_1}
    \end{align*}
  \end{enumerate}
\end{thm}
\begin{proof}
  2. and 3. follow directly from Theorem \ref{sec:indecstructure}.
  
  The first relation of 4. follows from
  \begin{align*}
    W_{1,\pm1}[-\Delta_1]W_{1,\pm1}=0\,.
  \end{align*}
  The second relation follows by applying \(E^2\) to \([W_{1,-1}]_{\mathfrak{p}}^2=0\).
  The third relation is a consequence of Proposition \ref{sec:zhutopoisson}
  which implies that the relation \([W_{1,0}]^2=g_1([T])\) in \(A_0\) becomes
  \begin{align*}
    [W_{1,0}]^2=\operatorname{const}\cdot[T]^{\Delta_1}
  \end{align*}
  in \(\text{Gr}_{\Delta_1}(A_0)\) and therefore the same relation follows in \(\mathfrak{p}(\mathcal{M}_{p_+,p_-})\).

  The Poisson algebra
  \(\mathfrak{p}(\mathcal{M}_{p_+,p_-})\) is finite dimensional because
  it is finitely generated and all of its generators are nilpotent, thus
  proving 1.
\end{proof}

\subsection{Classification of simple {\boldmath \(\mathcal{M}_{p_+,p_-}\)\unboldmath}
  modules}
\label{sec:repthy}

In this section we classify all simple modules of both the zero mode algebra \(A_0\) and of the VOA \(\mathcal{M}_{p_+,p_-}\).
We will see that all simple \(\mathcal{M}_{p_+,p_-}\)-modules are either
isomorphic to
minimal model modules \(L_{[r,s]}\) or
submodules of the lattice modules \(V^\pm_{r,s}\).

\begin{prop}\label{sec:Mppsimplemodules}
  Let \(\mathcal{K}_{r,s}^\pm\) and \(\mathcal{X}_{r,s}^\pm\) be the Virasoro modules of
  Proposition \ref{sec:Mppirreds}, where \(1\leq r\leq p_+,\ 1\leq s\leq p_-\).
  \begin{enumerate}
  \item The \(\mathcal{K}_{r,s}^\pm\) and \(\mathcal{X}_{r,s}^\pm\) are
    \(\mathcal{M}_{p_+,p_-}\)-modules.
  \item The \(\mathcal{X}_{r,s}^\pm\) are simple \(\mathcal{M}_{p_+,p_-}\)-modules
    for \(1\leq r\leq p_+\), \(1\leq s\leq p_-\).
  \item For \(1\leq r< p_+,\ 1\leq s< p_-\) the modules \(\mathcal{K}_{r,s}^+\)
    satisfy the exact sequences
    \begin{align*}
      0\rightarrow \mathcal{X}_{r,s}^+\rightarrow \mathcal{K}_{r,s}^+\rightarrow
      L_{[r,s]}\rightarrow 0\,,
    \end{align*}
    where the arrows are \(\mathcal{M}_{p_+,p_-}\)-homomorphisms and
    \begin{align*}
      L_{[r,s]}=L_{[p_+-r,p_--s]}=\mathcal{K}_{r,s}^+/\mathcal{X}_{r,s}^+=L(\Delta_{r,s})
    \end{align*}
    are the simple modules of the minimal model VOA 
    \(\text{MinVir}_{c_{p_+,p_-}}\).
  \item For \(1\leq r\leq p_+,\ 1\leq s\leq p_-\) the spaces of least conformal weight
    \(\mathcal{K}_{r,s}^+[\Delta_{r,s}]\) and \(\mathcal{X}_{r,s}^+[\Delta_{r,s;0}^+]\) of 
    \(\mathcal{K}_{r,s}^+\) and \(\mathcal{X}_{r,s}^+\) are one dimensional and the zero modes 
    \(W_{1,m}[0], m=-1,0,1\) act trivially. The spaces of least conformal weight
    \(\mathcal{K}_{r,s}^-[\Delta_{r,s;0}^-]=\mathcal{X}_{r,s}^-[\Delta_{r,s;0}^-]\) of
    \(\mathcal{K}_{r,s}^-=\mathcal{X}_{r,s}^-\) are two dimensional and the zero modes 
    \(W_{1,m}[0], m=-1,0,1\) act non-trivially.
  \end{enumerate}
\end{prop}
\begin{proof}
  The fact that \(\mathcal{K}_{r,s}^\pm\) and \(\mathcal{X}_{r,s}^\pm\) are
  \(\mathcal{M}_{p_+,p_-}\)-modules follows analogously to the proof of 
  Theorem \ref{sec:indecstructure} by showing that \(Y(W_{1,\pm1};z)\) act transitively
  on the soliton vectors. This then also implies that the \(\mathcal{X}_{r,s}^\pm\) are
  simple.
    
  The \(L_{[r,s]}\) are \(\mathcal{M}_{p_+,p_-}\)-modules, because they are 
  quotients of \(\mathcal{M}_{p_+,p_-}\)-modules. They are simple,
  because they are already simple as Virasoro modules.

  The spaces of least conformal weight of \(\mathcal{K}_{r,s}^+\) are
  just \(\mathbb{C}|\beta_{r,s:0}\rangle\) while for \(\mathcal{X}_{r,s}^\pm\) 
  they are the soliton sectors
  \(V_{r,s:0}^\pm\). 
\end{proof}

For a simple \(\mathcal{M}_{p_+,p_-}\)-module \(M\) let \(\overline{M}\) be its
corresponding simple \(A_0\)-module.
\begin{prop}\label{sec:A_0simples}
  The simple modules of the zero mode algebra corresponding to the simple
  \(\mathcal{M}_{p_+,p_-}\)-modules of Proposition \ref{sec:Mppsimplemodules}
  have the following structure:
  \begin{enumerate}
  \item For \(1\leq r<p_+\), \(1\leq s<p_-\) 
    the simple \(A_0\)-modules \(\overline{L_{[r,s]}}=\overline{L_{[p_+-r,p_--s]}}\) are
    1-dimensional. The Virasoro \([T]\) element acts as \(\Delta_{r,s}\cdot\id\), while
    the \([W_{1,m}]\) act trivially.
  \item For \(1\leq r<p_+\), \(1\leq s<p_-\), the simple \(A_0\)-modules 
    \(\overline{\mathcal{X}_{r,s}^+}\) are
    1-dimensional. The \([W_{1,m}]\) act trivially, while the Virasoro element
    acts as
    \begin{align*}
      [T]=\Delta_{r,s;0}^+\cdot \id\,.
    \end{align*}
  \item For \(1\leq r<p_+\), \(1\leq s<p_-\), the simple \(A_0\)-modules 
    \(\overline{X^-_{r,s}}\) are
    2-dimensional. Let \(v_\pm\) be a basis such that
    \begin{align*}
      v_-=[W_{1,-1}] v_+\,,
    \end{align*}
    then
    \begin{align*}
      [T]v_\pm&=\Delta_{r,s;0}^- v_\pm\\
      [W_{1,0}]v_\pm&=\mp f([T]) v_{\pm}
    \end{align*}
    and \(\Delta_{r,s;0}^-\) is not a root of \(f\).
  \end{enumerate}
\end{prop}
\begin{proof}
  Recall the structure of the \(\mathcal{K}_{r,s}^\pm\) and \(\mathcal{X}_{r,s}^\pm\) as laid out
  in Proposition \ref{sec:Mppirreds}.
  \begin{enumerate}
  \item For \(1\leq r<p_+\), \(1\leq s<p_-\) 
    the factorisation of \(\omega_1(\beta)\) in 
    Proposition \ref{sec:Wzeromodeeigenvalues} contains a factor 
    \(h_\beta-\Delta_{r,s}\). Therefore the \(A_0\) generator \([W_{1,0}]\)
    acts trivially on \(\overline{L_{[r,s]}}\). By applying \(E\) and \(F\) one sees that 
    \([W_{1,\pm1}]\) must also act trivially. Thus \([T]\) is the only
    generator of \(A_0\) that acts non-trivially on \(\overline{L_{[r,s]}}\) and
    it acts by multiplying by the conformal weight \(\Delta_{r,s}\).
  \item The case for \(\overline{\mathcal{X}_{r,s}^+}\) follows in the same way as for 
    \(\overline{L_{[r,s]}}\).
  \item For \(1\leq r\leq p_+\), \(1\leq s\leq p_-\), the zero modes 
    \(W_{1,m}[0], m=-1,0,1\) act non-trivially on the two dimensional soliton sector
    \(V_{r,s;0}^-\subset \mathcal{X}_{r,s}^-\). Therefore there exits a basis \(v_\pm\) of 
    \(\overline{\mathcal{X}^-_{r,s}}\) such that
  \begin{align*}
    v_-&=[W_{1,-1}]v_+& v_+&=\operatorname{const}[W_{1,1}]v_-.
  \end{align*}
  It then follows from the relations of \(A_0\) and the conformal weight of 
  \(V_{r,s;0}^-\) that
  \begin{align*}
    [T]v_\pm&=\Delta_{r,s;0}^- v_\pm\\
    [W_{1,0}]v_\pm&=\mp f([T]) v_{\pm}\,.
  \end{align*}
  Since \([W_{1,0}]\) acts non-trivially  on \(\overline{X_{r,s}^-}\),
  the conformal weight \(\Delta_{r,s;0}^-\)
  cannot be a root of \(f(h_\beta)\).
  \end{enumerate}
\end{proof}

\begin{thm}\label{sec:classificationofmodules}
  The list of \(2p_+p_-+(p_+-1)(p_--1)/2\) simple 
  \(\mathcal{M}_{p_+,p_-}\)-modules and 
  simple \(A_0\)-modules of Propositions \ref{sec:Mppsimplemodules} and 
  \ref{sec:A_0simples}
  are a full classification of all simple \(\mathcal{M}_{p_+,p_-}\)- and 
  \(A_0\)-modules.
\end{thm}
\begin{proof}
  From Theorem \ref{sec:zhualgrelations} we know that \(g_2([T])=0\) in
  \(A_0\). Therefore the minimal polynomial of \([T]\) must divide \(g_2([T])\),
  that is the conformal weight of any simple \(A_0\)-module must be a root of
  \(g_2(h_\beta)\). The list of simple \(A_0\)-modules in Proposition
  \ref{sec:A_0simples} gives one example of a simple \(A_0\)-module for every
  root of \(g_2(h_\beta)\). What remains to be shown is that there exist no other inequivalent simple
  \(A_0\)-modules with the same conformal weights as those listed above.
  
  Consider a simple \(A_0\)-module \(\overline{M}\) with conformal weight \(\Delta_{r,s}\) or
  \(\Delta_{r,s;0}^+\). By the polynomial relations between \([T]\) and
  \([W_{1,0}]\) in Theorem \ref{sec:zhualgrelations}, the generator
  \([W_{1,0}]\) must act trivially on \(\overline{M}\) and hence \([W_{1,\pm1}]\) must
  also act trivially. Therefore the image of of \(A_0\) in \(\ehom{(\overline{M}})\) is a
  commutative algebra generated by \([T]\). Since finite dimensional simple modules
  of commutative algebras are 1 dimensional, it follows that \(\overline{M}\) is 1
  dimensional and therefore included in the list in Proposition
  \ref{sec:A_0simples}.

  Consider a simple \(A_0\)-module \(\overline{M}\) with conformal weight
  \(\Delta_{r,s;0}^-\). By the polynomial relations between \([T]\) and
  \([W_{1,0}]\) in Theorem \ref{sec:zhualgrelations}, the generator
  \([W_{1,0}]\) must act non-trivially on \(\overline{M}\) and hence also
  \([W_{1,\pm1}]\). By Proposition \ref{sec:A_0simples}, it also follows that
  \(f(\Delta_{r,s;0}^-)\neq0\). Thus the image of \(A_0\) in \(\ehom(\overline{M})\)
  contains the Lie algebra \(\mathfrak{sl}_2\) with generators
  \begin{align*}
    \mathbb{H}=\frac{1}{f(\Delta_{r,s;0}^-)}[W_{1,0}],\quad
    \mathbb{E}=\frac{1}{\sqrt{2}f(\Delta_{r,s;0}^-)}[W_{1,-1}],\quad
    \mathbb{F}=\frac{-1}{\sqrt{2}f(\Delta_{r,s;0}^-)}[W_{1,1}],
  \end{align*}
  and \(\overline{M}\) must therefore be a simple finite dimensional \(\mathfrak{sl}_2\)
  module. Since \(\mathbb{E}\) and \(\mathbb{F}\) are non-trivial but
  \(\mathbb{E}^2=0\) and \(\mathbb{F}^2=0\), it follows that \(\overline{M}\) must be the
  2 dimensional simple \(\mathfrak{sl}_2\) module and is therefore included in the list in Proposition
  \ref{sec:A_0simples}.
\end{proof}

\subsection{The Whittaker category {\boldmath \(\mathcal{M}_{p_+,p_-}\)\unboldmath}-Whitt-mod}
\label{sec:whittcat}

Let \(\mathcal{M}_{p_+,p_-}\)-mod be the category of left \(\mathcal{M}_{p_+,p_-}\)-modules. Then in the sense of abelian categories
any object \(M\) of \(\mathcal{M}_{p_+,p_-}\)-mod has a finite Jordan-H\"older composition series and the simple objects of 
\(\mathcal{M}_{p_+,p_-}\)-mod are given by
\begin{align*}
  \{L_{[r,s]}|1\leq r<p_+,1\leq s<p_-\}\cup\{\mathcal{X}_{r,s}^\pm|1\leq r\leq p_+,1\leq s\leq p_-\}\,.
\end{align*}
\begin{definition}
  We define the full subcategory \(\mathcal{M}_{p_+,p_-}\)-Whitt-mod of
  \(\mathcal{M}_{p_+,p_-}\)-mod as the category of all objects \(M\in
  \mathcal{M}_{p_+,p_-}\)-mod such that the composition series of \(M\)
  only contains simple objects in
  \begin{align*}
    \{\mathcal{X}_{r,s}^\pm|1\leq r\leq p_+,1\leq s\leq p_-\}\,.
  \end{align*}
  We call the category \(\mathcal{M}_{p_+,p_-}\)-Whitt-mod the \emph{Whittaker category of \(\mathcal{M}_{p_+,p_-}\)-mod} and
  the objects of \(\mathcal{M}_{p_+,p_-}\)-Whitt-mod \emph{Whittaker \(\mathcal{M}_{p_+,p_-}\)-modules}.
\end{definition}

The full subcategory \(\mathcal{M}_{p_+,p_-}\)-Whitt-mod has properties crucial to developing the conformal field theory
associated to \(\mathcal{M}_{p_+,p_-}\).

\section{Concluding remarks and further problems}
\label{sec:conclusion}

Since the VOA \(\mathcal{M}_{p_+,p_-}\) satisfies the \(c_2\)-cofiniteness condition, one can develop the
conformal field theory over general Riemann surfaces associated to \(\mathcal{M}_{p_+,p_-}\). A paper on this subject is being
prepared by the first author together with Y.~Hashimoto \cite{Hashimoto:20XX}.

We have some conjectures regarding the conformal field theory over general Riemann surfaces associated to \(\mathcal{M}_{p_+,p_-}\). By
considering the conformal field theory on the Riemann sphere associated to \(\mathcal{M}_{p_+,p_-}\), the fusion tensor product
\(\dot{\otimes}\) induces the structure of a braided monoidal category on \(\mathcal{M}_{p_+,p_-}\)-mod. However, as was noted in 
\cite{Gaberdiel:2009ug},
this fusion tensor product is not exact on \(\mathcal{M}_{p_+,p_-}\)-mod. We conjecture the following:
\begin{enumerate}
\item The full abelian subcategory
  \(\operatorname{Vir}_{\operatorname{min}}\subset
  \mathcal{M}_{p_+,p_-}\)-mod, generated by the simple modules
  \begin{align*}
    \{L_{[r,s]}=|1\leq r<p_+,1\leq s< p_-\}
  \end{align*}
  forms a tensor ideal in \(\mathcal{M}_{p_+,p_-}\)-mod .
  Thus \(\mathcal{M}_{p_+,p_-}\)-Whitt-mod is endowed
  with a quotient braided monoidal structure 
\((\mathcal{M}_{p_+,p_-}\)-Whitt-mod, \(\dot{\otimes})\).
\item The category \((\mathcal{M}_{p_+,p_-}\text{-Whitt-mod},\dot{\otimes})\) is rigid as a monoidal category.
\item In their seminal paper \cite{Feigin:2006xa} Feigin, Gainutdinov, Semikhatov and Tipunin defined the quantum group
  \(\mathfrak{g}_{p_+,p_-}\) -- a finite dimensional complex Hopf algebra. They determined all simple modules of 
  \(\mathfrak{g}_{p_+,p_-}\) and their projective covers. There are exactly \(2p_+p_-\) simple modules, which is also the number
  of simple objects in \(\mathcal{M}_{p_+,p_-}\)-Whitt-mod. The Hopf algebra \(\mathfrak{g}_{p_+,p_-}\) is not quasi-triangular,
  nevertheless, we expect the monoidal category \((\mathfrak{g}_{p_+,p_-}\text{-mod},\dot{\otimes})\) to be braided. Since
  \(\mathfrak{g}_{p_+,p_-}\) is a Hopf algebra, \((\mathfrak{g}_{p_+,p_-}\text{-mod},\dot{\otimes})\) is a rigid monoidal category.
  We conjecture that
  \begin{align*}
    (\mathcal{M}_{p_+,p_-}\text{-Whitt-mod},\dot{\otimes})\cong(\mathfrak{g}_{p_+,p_-}\text{-mod},\dot{\otimes})
  \end{align*}
  as braided monoidal categories.\footnote{The first author attributes this conjecture to exciting discussions with
  Semikhatov and Tipunin during his stay in Moscow in January 2013 and is convinced of its validity.}
\item In \cite{Feigin:2006xa} it was also shown that the centre \(Z(\mathfrak{g}_{p_+,p_-})\) of \(\mathfrak{g}_{p_+,p_-}\)
  is \(\tfrac12(3p_+-1)(3p_--1)\) dimensional and carries the structure of a \(\operatorname{SL}(2,\mathbb{Z})\)-module. The 
  space of vacuum amplitudes of \(\mathcal{M}_{p_+,p_-}\)-mod on the torus also carries an \(\operatorname{SL}(2,\mathbb{Z})\)
  action. Explaining the relation between these two \(\operatorname{SL}(2,\mathbb{Z})\) 
  modules will be an important future problem.
\item For any rank \(\ell\) simple Lie algebra \(\mathfrak{g}\) of ADE type, we fix coprime positive integers \(p_+,p_-\geq h=h^\vee\),
  where \(h\) is the Coxeter number of \(\mathfrak{g}\). By using free bosons and \(2\ell\) screening operators, one can define
  an extended \(W\) algebra of type \(\mathfrak{g}\), which we call \(\mathcal{M}_{p_+,p_-}(\mathfrak{g})\). When
  \(\mathfrak{g}=\mathfrak{sl}_2\), then \(\mathcal{M}_{p_+,p_-}(\mathfrak{sl}_2)\) is just \(\mathcal{M}_{p_+,p_-}\), the VOA
  considered in this paper.
  We conjecture that \(\mathcal{M}_{p_+,p_-}(\mathfrak{g})\) satisfies Zhu's \(c_2\)-cofiniteness condition and that one can
  construct Frobenius homomorphisms \(E_i,F_i,\ i=1,\dots,\ell\) subject to relations that
  are similar to those encountered in the case \(\mathfrak{g}=\mathfrak{sl}_2\).
\end{enumerate}


\begin{thebibliography}{10}

\bibitem{Feigin:2006iv}
B.~L. Feigin, A.~M. Gainutdinov, A.~M. Semikhatov, and I.~Y. Tipunin,
  ``{Logarithmic extensions of minimal models: Characters and modular
  transformations},'' {\em Nucl. Phys.} {\bfseries B757} (2006) 303--343,
\href{http://arxiv.org/abs/hep-th/0606196}{{\ttfamily arXiv:hep-th/0606196}}.

\bibitem{Borcherds:1983sq}
R.~E. Borcherds, ``{Vertex algebras, Kac-Moody algebras, and the monster},''
{\em Proc. Nat. Acad. Sci.} {\bfseries 83} (1986) 3068--3071.

\bibitem{Frenkel:2001}
E.~Frenkel and D.~Ben-Zvi, ``{Vertex Algebras and Algebraic Curves},'' {\em
  Mathematical Surveys and Monographs, Amer. Math. Soc.} {\bfseries 88} (2001)
  .

\bibitem{Nagatomo:2002}
K.~Nagatomo and A.~Tsuchiya, ``{Conformal field theories associated to regular
  chiral vertex operator algebras, I: Theories over the projective line},''
  {\em Duke Math J.} {\bfseries 128} (2005) 393--471,
  \href{http://arxiv.org/abs/math.QA/0206223}{{\ttfamily
  arXiv:math.QA/0206223}}.

\bibitem{Matsuo:2005}
A.~Matsuo, K.~Nagatomo, and A.~Tsuchiya, ``{Quasi-finite algebras graded by
  Hamiltonian and vertex operator algebras},'' {\em London Mathematical
  Society, Lecture Note Series} {\bfseries 372} (2010) 282 -- 329,
  \href{http://arxiv.org/abs/math.QA/0505071}{{\ttfamily
  arXiv:math.QA/0505071}}.

\bibitem{Zhu:1996}
Y.~Zhu, ``{Modular invariance of characters of vertex operator algebras},''
  {\em J. Amer. Math. Soc.} {\bfseries 9} (1996) 237--302.

\bibitem{Frenkel:1992}
I.~B. Frenkel and Y.~Zhu, ``{Vertex operator algebras associated to
  representations of affine and Virasoro algebras},'' {\em Duke Math J.}
  {\bfseries 66} (1992) 123--168.

\bibitem{Gaberdiel:1996np}
M.~R. Gaberdiel and H.~G. Kausch, ``{A rational logarithmic conformal field
  theory},'' {\em Phys. Lett.} {\bfseries B386} (1996) 131--137,
\href{http://arxiv.org/abs/hep-th/9606050}{{\ttfamily arXiv:hep-th/9606050}}.

\bibitem{Gaberdiel:2007jv}
M.~R. Gaberdiel and I.~Runkel, ``{From boundary to bulk in logarithmic CFT},''
  {\em J. Phys. A: Math. Theor.} {\bfseries 41} (2008) ,
\href{http://arxiv.org/abs/0707.0388}{{\ttfamily arXiv:0707.0388 [hep-th]}}.

\bibitem{Adamovic:2007er}
D.~Adamovi\'{c} and A.~Milas, ``{On the triplet vertex algebra $W(p)$},'' {\em
  Adv. Math.} {\bfseries 217} (2008) 2664--2699,
\href{http://arxiv.org/abs/0707.1857}{{\ttfamily arXiv:0707.1857 [math.QA]}}.

\bibitem{Nagatomo:2009xp}
K.~Nagatomo and A.~Tsuchiya, ``{The Triplet Vertex Operator Algebra W(p) and
  the Restricted Quantum Group at Root of Unity},'' {\em Adv. Stdu. in Pure
  Math., Exploring new Structures and Natural Constructions in Mathematical
  Physics, Amer. Math. Soc.} {\bfseries 61} (2011) 1--49,
\href{http://arxiv.org/abs/0902.4607}{{\ttfamily arXiv:0902.4607 [math.QA]}}.

\bibitem{Tsuchiya:2012}
A.~Tsuchiya and S.~Wood, ``{The tensor structure on the representation category
  of the $W_p$ triplet algebra},''  (2012) ,
  \href{http://arxiv.org/abs/1201.0419}{{\ttfamily arXiv:1201.0419}}.

\bibitem{Fuchs:2003yu}
J.~Fuchs, S.~Hwang, A.~Semikhatov, and I.~Y. Tipunin, ``{Nonsemisimple fusion
  algebras and the Verlinde formula},''
  \href{http://dx.doi.org/10.1007/s00220-004-1058-y}{{\em Commun.Math.Phys.}
  {\bfseries 247} (2004) 713--742},
\href{http://arxiv.org/abs/hep-th/0306274}{{\ttfamily arXiv:hep-th/0306274
  [hep-th]}}.

\bibitem{Kazhdan:1993a}
D.~Kazhdan and G.~Lusztig, ``{Tensor Structures Arising from Affine Lie
  Algebras. I},'' {\em J. Amer. Math. Soc} {\bfseries 6} (1993) 905--947.

\bibitem{Kazhdan:1993b}
D.~Kazhdan and G.~Lusztig, ``{Tensor Structures Arising from Affine Lie
  Algebras. II},'' {\em J. Amer. Math. Soc} {\bfseries 6} (1993) 949--1011.

\bibitem{Kazhdan:1994a}
D.~Kazhdan and G.~Lusztig, ``{Tensor Structures Arising from Affine Lie
  Algebras. III},'' {\em J. Amer. Math. Soc} {\bfseries 7} (1994) 335--381.

\bibitem{Kazhdan:1994b}
D.~Kazhdan and G.~Lusztig, ``{Tensor Structures Arising from Affine Lie
  Algebras. IV},'' {\em J. Amer. Math. Soc} {\bfseries 7} (1994) 383--453.

\bibitem{Feigin:2006xa}
B.~L. Feigin, A.~M. Gainutdinov, A.~M. Semikhatov, and I.~Y. Tipunin,
  ``{Kazhdan-Lusztig-dual quantum group for logarithmic extensions of Virasoro
  minimal models},'' {\em J. Math. Phys.} {\bfseries 48} (2007) 032303,
\href{http://arxiv.org/abs/math/0606506}{{\ttfamily arXiv:math/0606506}}.

\bibitem{Bushlanov:2009cv}
P.~Bushlanov, B.~Feigin, A.~Gainutdinov, and I.~Y. Tipunin, ``{Lusztig limit of
  quantum sl(2) at root of unity and fusion of (1,p) Virasoro logarithmic
  minimal models},''
  \href{http://dx.doi.org/10.1016/j.nuclphysb.2009.03.016}{{\em Nucl.Phys.}
  {\bfseries B818} (2009) 179--195},
\href{http://arxiv.org/abs/0901.1602}{{\ttfamily arXiv:0901.1602 [hep-th]}}.

\bibitem{Bushlanov:2011ha}
P.~Bushlanov, A.~Gainutdinov, and I.~Y. Tipunin, ``{Kazhdan-Lusztig equivalence
  and fusion of Kac modules in Virasoro logarithmic models},''
  \href{http://dx.doi.org/10.1016/j.nuclphysb.2012.04.018}{{\em Nucl.Phys.}
  {\bfseries B862} (2012) 232--269},
\href{http://arxiv.org/abs/1102.0271}{{\ttfamily arXiv:1102.0271 [hep-th]}}.

\bibitem{Adamovic:2010}
D.~Adamovi\'{c} and A.~Milas, ``{On W-Algebras Associated to $(2, p)$ Minimal
  Models and Their Representations},'' {\em Int. Math. Res. Not} {\bfseries
  2010} 3896--3934, \href{http://arxiv.org/abs/0908.4053}{{\ttfamily
  arXiv:0908.4053 [math.QA]}}.

\bibitem{Adamovic:2011}
D.~Adamovi\'{c} and A.~Milas, ``{On $W$-algebra extensions of $(2,p)$ minimal
  models: $p>3$},'' {\em J. Alg} {\bfseries 344} (2011) 313--332,
  \href{http://arxiv.org/abs/1101.0803}{{\ttfamily arXiv:1101.0803 [math.QA]}}.

\bibitem{Macdonald:1987}
I.~G. Macdonald, ``{Commuting differential operators and zonal spherical
  functions},'' {\em Lecture Notes in Math.} {\bfseries 1271} (1987) 189--200.

\bibitem{Kadell:1997}
K.~W.~J. Kadell, ``{The Selberg-Jack symmetric functions},'' {\em Adv. Math.}
  {\bfseries 130} (1997) 33--102.

\bibitem{Belavin:1984}
A.~A. Belavin, A.~M. Polyakov, and A.~B. Zamolodchikov, ``{Infinite conformal
  symmetry in two-dimensional quantum field theory},'' {\em Nucl. Phys. B}
  {\bfseries 241} (1984) 333--380.

\bibitem{Tsuchiya:1987}
A.~Tsuchiya and Y.~Kanie, ``{Vertex operators in the conformal field theory on
  P1 and monodromy representations of the braid group},'' {\em Adv. Stdu. Pure
  Math., Amer. Math. Soc.} {\bfseries 16} (1988) 297--373.

\bibitem{Orlik:2001}
P.~Orlik and H.~Terao, {\em {Arrangements and hypergeometric integrals}}.
\newblock Math. Soc. Japan, 2001.

\bibitem{Varchenko:2003}
A.~Varchenko, {\em {Special Functions, KZ Type Equations, and Representation
  Theory}}.
\newblock Amer. Math. Soc., 2003.

\bibitem{MacDonald:1999}
I.~G. Macdonald, {\em {Symmetric Functions and Hall Polynomials}}.
\newblock Clarendon Press, Oxford, 1995.

\bibitem{Feigin:1984}
B.~L. Feigin and D.~B. Fuchs, ``{Verma modules over the Virasoro algebra},''
  {\em Lect. Notes in Math.} {\bfseries 1060} (1984) 230--245.

\bibitem{Feigin:1988}
B.~L. Feigin and D.~B. Fuchs, ``{Cohomologies of some nilpotent subalgebras of
  the Virasoro and Kac-Moody Lie algebras},'' {\em Jour. Geom. Phys.}
  {\bfseries 5} (1988) 209--235.

\bibitem{FeFu:1990}
B.~L. Feigin and D.~B. Fuchs, ``{Representations of the Virasoro algebra},''
  {\em Representations of Lie Groups and Related Topics, Adv. Stud. Contemp.
  Math., Gordon and Breach, New York,} {\bfseries 7} (1990) 465--554.

\bibitem{Brungs:1998ij}
D.~Brungs and W.~Nahm, ``{The Associative algebras of conformal field
  theory},'' \href{http://dx.doi.org/10.1023/A:1007525300192}{{\em Lett. Math.
  Phys.} {\bfseries 47} (1999) 379--383},
\href{http://arxiv.org/abs/hep-th/9811239}{{\ttfamily arXiv:hep-th/9811239
  [hep-th]}}.

\bibitem{Dotsenko:1984}
V.~Dotsenko and V.~Fateev, ``{Conformal Algebra and Multipoint Correlation
  Functions in Two-Dimensional Statistical Models},''
\href{http://dx.doi.org/10.1016/0550-3213(84)90269-4}{{\em Nucl.Phys.}
  {\bfseries B240} (1984) 312}.

\bibitem{Dotsenko:1985}
V.~Dotsenko and V.~Fateev, ``{Four Point Correlation Functions and the Operator
  Algebra in the Two-Dimensional Conformal Invariant Theories with the Central
  Charge $c < 1$},''
\href{http://dx.doi.org/10.1016/S0550-3213(85)80004-3}{{\em Nucl.Phys.}
  {\bfseries B251} (1985) 691}.

\bibitem{AomotoKita:2011}
K.~Aomoto and M.~Kita, {\em {Theory of Hypergeometric Functions}}.
\newblock Springer, 2011.

\bibitem{Tsuchiya:1986}
A.~Tsuchiya and Y.~Kanie, ``{Fock space representations of the Virasoro algebra
  -- Intertwining operators},'' {\em Publ. RIMS, Kyoto Univ.} {\bfseries 22}
  (1986) 259--327.

\bibitem{Selberg:1944}
A.~Selberg, ``{Bemerkninger om et Mulitiplet Integral},'' {\em Norsk Mat.
  Tidsskrift} {\bfseries 26} (1944) 71--78.

\bibitem{Forrester:2008}
P.~J. Forrester and S.~O. Warnaar, ``{The importance of the Selberg
  integral},'' {\em Bull. Amer. Math. Soc.} {\bfseries 45} (2008) 489--534.

\bibitem{Mimachi:1995}
K.~Mimachi and Y.~Yamada, ``{Singular vectors of the Virasoro algebra in terms
  of Jack symmetric polynomials},'' {\em Comm. Math. Phys.} {\bfseries 174}
  (1995) 447--455.

\bibitem{Iohara:2010}
K.~Iohara and Y.~Koga, {\em {Representation Theory of the Virasoro Algebra}}.
\newblock Springer, 2010.

\bibitem{Fel:1989}
G.~Felder, ``{BRST approach to minimal models},'' {\em Nuc. Phy. B} {\bfseries
  317} (1989) 215--236.

\bibitem{DongLepowski:1993}
C.~Dong and J.~Lepowsky, {\em Generalized Vertex Algebras and Relative Vertex
  Operators}.
\newblock Birkh{\"a}user, 1993.

\bibitem{Hashimoto:20XX}
Y.~Hashimoto and A.~Tsuchiya. In preparation.

\bibitem{Gaberdiel:2009ug}
M.~R. Gaberdiel, I.~Runkel, and S.~Wood, ``{Fusion rules and boundary
  conditions in the c=0 triplet model},'' {\em J.Phys.} {\bfseries A42} (2009)
  325403,
\href{http://arxiv.org/abs/0905.0916}{{\ttfamily arXiv:0905.0916 [hep-th]}}.

\end{thebibliography}

\providecommand{\href}[2]{#2}\begingroup\raggedright\endgroup

\end{document}